\documentclass{amsart}
\usepackage{amsmath, amsthm, amssymb}
\usepackage{verbatim}
\usepackage{tikz}
\usetikzlibrary{matrix}

\usepackage[mathscr]{eucal} 
\usepackage{mathrsfs} 
\usepackage{cancel}

\renewcommand{\d}{\partial}

\newcommand{\dbar}{\overline{\d}}
\newcommand{\ii}{\sqrt{-1}}
\newcommand{\wt}[1]{\widetilde{#1}}
\newcommand{\lc}{\left<}
\newcommand{\rc}{\right>}

\newcommand{\wh}[1]{\widehat{#1}}
\newcommand{\cw}[1]{\check{#1}}
\newcommand{\bb}{{\bf B \;}}
\newcommand{\ul}[1]{\underline{#1}}

\newcommand{\eps}{\epsilon}
\newcommand{\veps}{\varepsilon}
\newcommand{\vphi}{\varphi}

\newcommand{\al}{\alpha} 
 
\newcommand{\be}{\beta}
\newcommand{\ga}{\gamma}
\newcommand{\de}{\delta}
\newcommand{\la}{\lambda}
\newcommand{\om}{\omega}
\newcommand{\te}{\theta}

\newcommand{\el}{\ell}

\newcommand{\Te}{\Theta}
\newcommand{\Ga}{\Gamma}
\newcommand{\Om}{\Omega}

\newcommand{\na}{\nabla}

\newcommand{\cB}{\mathcal{B}}

\newcommand{\cE}{\mathcal{E}}

\newcommand{\cL}{\mathcal{L}}

\newcommand{\cO}{\mathcal{O}}

\newcommand{\cU}{\mathcal{U}}

\newcommand{\supp}{\mbox{supp }}

\newcommand{\bR}{\mathbb{R}}

\newcommand{\bC}{\mathbb{C}}

\newcommand{\cali}[1]{\mathscr{#1}}

\newcommand{\Lc}{\cali{L}}

\newtheorem{thm}{Theorem}
\newtheorem{prop}[thm]{Proposition}
\newtheorem{lem}[thm]{Lemma}
\newtheorem{cor}[thm]{Corollary}

\newtheorem*{thmx}{Localization Principle}

\theoremstyle{definition}
\newtheorem{defn}[thm]{Definition}
\newtheorem{remark}[thm]{Remark}

\numberwithin{thm}{section}
\numberwithin{equation}{section}

\renewcommand{\[}{\begin{equation}}
\renewcommand{\]}{\end{equation}}

\newcommand{\wed}{\wedge}

\newcommand{\ov}[1]{\overline{#1}}

\title[Complex Hessian measures with respect to a Hermitian form]{Complex Hessian measures with respect to a background Hermitian form}

\author{S\l awomir Ko\l odziej and Ngoc Cuong Nguyen} 
\address{Faculty of Mathematics and Computer Science, Jagiellonian University, \L ojasiewicza 6, 30-348 Krak\'ow, Poland}
\email{slawomir.kolodziej@im.uj.edu.pl}
\address{Department of Mathematical Sciences, KAIST, 291 Daehak-ro, Yuseong-gu, Daejeon 34141, South Korea}
\email{cuongnn@kaist.ac.kr}

\begin{document}

\maketitle 

\begin{center}
\em To the memory of Jean-Pierre Demailly \rm
\end{center}

\bigskip

\begin{abstract} We develop potential theory for $m$-subharmonic functions with respect to a Hermitian metric on a Hermitian manifold. First, we  show that the complex Hessian operator is well-defined for bounded functions in this class. This allows to define the $m$-capacity and  then showing the quasi-continuity of $m$-subharmonic functions. Thanks to this we derive  other results parallel to those  in pluripotential theory such as the equivalence between polar sets and negligible sets.  The theory is then used to study the complex Hessian equation on compact Hermitian manifold with boundary, with the right hand side of the equation admitting a bounded subsolution. This is an extension of a recent result of Collins and Picard dealing with classical solutions.
\end{abstract}

\section{Introduction}

The $m$-Hessian operator is defined in terms of elementary symmetric polynomials of degree $m$ of eigenvalues of the Hessian matrix of the given function. If the degree is equal to the dimension of the space then one deals with  the most important 
case of the Monge-Amp\`ere operator.  One can also consider more general symmetric functions of eigenvalues. The nonlinear equations involving such operators will be called in this article \it Hessian type \rm equations.
They do appear in geometry in problems involving curvatures, like prescribed the Gauss curvature equation or the Lagrangian mean curvature equation. The $m$-Hessian equations in $\bR^n$ were first solved by Caffarelli-Nirenberg-Spruck \cite{CNS85} for smooth, non-degenerate data. The study of weak solutions for measures on the right hand side was initiated by Trudinger and Wang \cite{TrWa97, TrWa99, TrWa02} (see also \cite{Wa09}).

Here we are interested in the complex setting and weak solutions.  
For smooth data  the first solutions in complex variables were obtained by Vinacua \cite{V88} and S.Y. Li \cite{Li04} who followed the method of \cite{CNS85}.
B\l ocki \cite{Bl05}  adopted the methods of pluripotential theory (initiated by Bedford and Taylor \cite{BT76, BT82} in relation to 
the complex Monge-Amp\`ere equation) to define the action of the $m$-Hessian operator on non smooth functions
and study weak solutions of the associated equation.
 
Let $\Om\subset \bC^n$ be an open set and $\om$ is a positive Hermitian $(1,1)$-form on $\Om$. Let $1\leq m\leq n$ be an integer and consider  a function $u\in C^2(\Om,\bR)$. The complex Hessian operator  with respect to $\om$ acts on  $u$ by
$$
	H_m(u) = (dd^c u)^m \wed \om^{n-m}.
$$
The operator is elliptic if we restrict ourselves to functions $u$ whose  eigenvectors $\la = (\la_1,...,\la_n)$ of the complex Hessian matrix $[u_{i\bar j}]_{1\leq i,j\leq n}$, with respect to $\om$, belong to the G\aa rding cone
$$
	\Ga_m  = \left\{ \la \in \bR^n: S_1(\la)>0,...,S_m(\la)>0 \right\},
$$
where $S_k(\la)$ is the $k$-th elementary symmetric polynomial on $\la$. Such a function is called $m-\om$-subharmonic (or $m-\om$-sh for short).

For $\om = dd^c |z|^2$  the standard K\"ahler form on $\bC^n$  B\l ocki  defined non-smooth $m$-subharmonic functions. He showed that the Hessian operator acting on a bounded $m$-subharmonic function is a well-defined positive Radon measure, that the operator is stable under decreasing sequences, and that the homogeneous Dirichlet problem is solvable.
The non-homogeneous one with the right hand side in $L^p ,\ p>n/m,$ was solved by Dinew and the first author in \cite{DK14}.

On a compact Hermitian manifold  $(X, \om )$ the right $m$-Hessian operator to consider is
$$
	H_m(u) = (dd^c u + \om )^m \wed \om^{n-m},
$$
 or more generally 
$$
	H_{m, \alpha} (u) = (dd^c u+\alpha )^m \wed \om^{n-m},
$$
where $\alpha$ is another  $(1,1)$ form.

For $\omega$ K\"ahler the counterpart of Calabi-Yau theorem was shown by  Dinew and the first author in \cite{DK17},
with a use of  earlier $C^2$ estimates of Hou-Ma-Wu \cite{HMW10}. Having this result an analogue of pluripotential theory
yields weak solutions (see \cite{DK14}). We refer the readers to  \cite{DiL15}, \cite{chinh13a, chinh13b,chinh15} and \cite{chinh-dong}  for results in  potential theory for $m-\om$-sh functions on a compact K\"ahler manifold.

 Our first goal here is to develop   potential theory for $m$-subharmonic functions (with respect to a Hermitian metric) on a  Hermitian manifold. The results  often parallel those of pluripotential theory. 

Now we assume that $\om$ is a general Hermitian metric. The complex $m$-Hessian equation on compact manifolds was solved independently by Sz\'ekelyhidi \cite{Sz18} and Zhang \cite{Zha17}. 
The authors \cite{KN3} obtained weak continuous solutions for  the right hand side in $L^p ,\ p>n/m.$
This partially motivates the development of  potential theory for $m-\om$-sh functions on Hermitian manifolds. Unlike in the K\"ahler case, we have to deal with the non-zero torsion terms $dd^c \om$ and $d\om\wed d^c\om$.  A direct computation shows that for  a smooth $m-\om$-sh function $u$,  $0\leq p \leq n-m-1$ and $k\geq 1$, the form $(dd^c u)^k\wed \om^p$ may not be positive.  Those terms appear when we perform integration by parts. This makes the proofs of basic potential estimates in the  Hermitian setting substantially more difficult. For example,  we need to fully exploit the properties of the positive cone $\Ga_m$, and show new inequalities on elementary symmetric polynomials to prove the Chern-Levine-Nirenberg (CLN) inequality \cite{KN3} and a variant of Cauchy-Schwarz inequality in this paper. This coupled with the uniform convergence allows us to define
the complex Hessian measure of a continuous $m-\om$-sh function $u$ as the weak limit of
\[\label{eq:intro:weak-limit}
	H_m(u) := \lim_{\de\to 0} H_m(u^\de) = \lim_{\de\to 0}(dd^c u^\de)^m \wed \om^{n-m},
\]
where $\{u^\de\}$ is a sequence of smooth $m-\om$-sh functions converging uniformly to $u$.

However, if $u$ is a bounded $m-\om$-sh function up till now we have not been able to define its complex Hessian measure. Thanks to a basic observation (Lemma~\ref{lem:zero-order-1}) that $dd^c u$ is a current of order zero we found a very natural way to extend 
the classical approach of Bedford-Taylor in this case.
The first main result of the paper (Theorem~\ref{thm:w-m-functions}) shows that  the measure on the left hand side of \eqref{eq:intro:weak-limit} can be defined by taking weak limit inductively. Thus, for bounded $m-\om$-subharmonic functions $u$ and $u^\de \downarrow u$ point-wise as $\de\to 0$,  
$$	H_m(u) := \lim_{\de_m\to0} \cdots\lim_{\de_1\to 0} dd^c u^{\de_m} \wed \cdots \wed dd^c u^{\de_1} \wed \om^{n-m}
$$
exists in the  sense of currents of order zero.
The proof is based on the CLN inequality in \cite{KN3}. 
This is the starting point for proving analogues of Bedford-Taylor results presented in Chapter 1 of \cite{K05}. The main difference  is that we no longer have a nice integration by parts formula  for closed positive currents. Instead we need to work with the (non-closed) currents of order zero in general. This makes our proof of  weak convergence of currents  under decreasing limits
 of $m-\om$-sh  bounded functions in Proposition~\ref{prop:w-convergence-m} more complicated than the one in  \cite{BT82}.
Next, we obtain (Theorem~\ref{thm:quasi-continuity})  the quasi-continuity of $m-\om$-sh functions with respect to a suitable $m$-capacity:  for a Borel set $E\subset \Om$,
\[\label{eq:cap-intro}
	cap_m (E) = \sup\left\{\int_E H_m(u) : u \text{ is }m-\om\text{-sh in }\Om, -1\leq u\leq 0 \right\}.
\]
To define this capacity we needed Theorem~\ref{thm:w-m-functions}.
Once quasi-continuity is proven, one obtains weak convergence  of mixed wedge products of the forms $dd^c u_j$ for
$m-\om$-sh functions $u_j$ under monotone convergence (Lemma~\ref{lem:bounded-decreasing} and Lemma~\ref{lem:bounded-increasing}) following the classical arguments in \cite{BT82, BT87}.
Next we  study the polar sets and negligible sets of $m-\om$-sh functions. In this setting it seems impossible to obtain  nice formulae for the capacity of compact or open sets in terms of Hessian measures of extremal functions as it is the case for Monge-Amp\`ere measures. Exploiting further the properties of  $\Ga_m$ in Section~\ref{ssec:integral-smooth-fct}, especially Proposition~\ref{prop:cap-integral} we are able to compare the outer capacity and the Hessian measures of relative extremal functions in Lemma~\ref{lem:cap-formula}. This suffices to give a characterization of a polar set by $cap_m^*(E) =0$ in Proposition~\ref{prop:polar-set-characterization}. Consequently, we conclude  the equivalence of  polar sets and negligible sets (Theorem~\ref{thm:NP}).

In the last sections we apply the above results to the complex $m$-Hessian equation.
Recently, there is a lot of  active research on  fully non-linear elliptic equations on compact Hermitian manifolds with or without boundary (cf. \cite{CM21}, \cite{CP22}, \cite{D21}, \cite{DL21} \cite{Gu98}, \cite{GL10}, \cite{GLu23}, \cite{GP22}, \cite{LS20}, \cite{PT21}, \cite{Sz18},  \cite{SzTW17},   \cite{TW10b}, \cite{Yu21}) in various geometric contexts.

In particular Collins and Picard  \cite{CP22} solved the Dirichlet problem for the $m$-Hessian equation  in an open subset of a Hermitian manifold under the hypothesis of existence of  a subsolution and smooth data.
We extend it in Sections \ref{sec:DP}, \ref{sec:subsolution}  showing that the existence of a bounded subsolution implies the existence of  a bounded solution for both bounded domains  (Theorem~\ref{thm:bounded-subsolution}) and compact complex manifolds with boundary (Theorem~\ref{thm:bounded-sub-mfd}).
 In the proof the above equivalence of polar and negligible sets will play a crucial role (see Lemma~\ref{lem:L1-convergence}).
 The homogeneous  $m$-Hessian equation on a (K\"ahler) manifold with boundary was recently solved in a particular case in \cite{W23} in relation to the
Wess-Zumino-Witten type equation proposed by Donaldson  in \cite{Do99}.

\medskip

{\em Acknowledgement.} The first author is partially supported by  grant  no. \linebreak 2021/41/B/ST1/01632 from the National Science Center, Poland. The second author is  partially supported by  the National Research Foundation of Korea (NRF) grant  no. 2021R1F1A1048185. This project was initiated during the visit of the first author in KAIST, and he wishes to express his gratitude for great hospitality and prefect working conditions. 

We are very grateful to  Zywomir Dinew for kindly pointing out to us the gap in the proof of Lemma~3.2 in the previous version. We also thank Le Mau Hai for useful comments. 

\section{Generalized $m$-subharmonic functions}

\subsection{Elementary symmetric positive cones}
\label{ssec:positive-cone}

In this section we prove important point-wise estimates for elementary symmetric polynomials. 
Let $1\leq m \leq n$ be two integers. The positive cone $\Ga_m$ is given by
\[\label{eq:m-cone}
	\Ga_m = \{\la = (\la_1,...,\la_n) \in \bR^n: S_1(\la)>0,...,S_m(\la)>0\},
\]
where $S_k(\la) =\sum_{1\leq i_1<...<i_k\leq n} \la_{i_1} \cdots \la_{i_k}$; and conventionally:  $S_0(\la)=1$ and $S_k(\la)=0$ for $k<0$ or $k>n$. 
Let  $\la = (\la_1,...,\la_n) \in \Ga_m$ be arranged in the decreasing manner, i.e., 
$$\la_1 \geq \cdots \geq \la_m \geq \cdots  \geq \la_n.$$
Then, we know from \cite[Lemma~8]{Iv83}  that $\la_m>0$ which is a consequence of a characterization of the cone $\Ga_m$. Namely, for $\{i_1,...,i_t\} \subset\{1,...,n\}$ such that $k+t \leq m$, one has
\[\label{eq:char-Ivochina}
	S_{k;i_1\cdots i_t} (\la) >0, 
\]
where $S_{k;i_1\cdots i_t} (\la) = S_k|_{\la_{i_1} = \cdots = \la_{i_t}=0}.$ This implies also that for $1\leq k \leq m$,
\[\label{eq:Ivochina}
	S_{k-1} (\la) \geq \la_1 \cdots \la_{k-1}.
\]

\begin{lem} \label{lem:sym-e} 
There exists $\te = \te(n,m)>0$ such that the following statements hold.
\begin{itemize}
\item
[(a)] for $1\leq j \leq m$,  
$
	\la_j S_{m-1;j} \geq \te S_{m};
$
\item
[(b)]  for $1\leq i \leq m-1$, 
$
	\la_i S_{m-2;im} \geq \te S_{m-1;m}.
$
\end{itemize}
\end{lem}

\begin{proof} The item (a) follows from \cite[Eq. (2.7)]{KN3}, while (b) follows from (a) if we replace $n, m$ and $S_m$ by $n-1, m-1$ and $S_{m-1;m}$, respectively.
\end{proof}

\begin{lem} \label{lem:sym-ineq} 
There exists a uniform constant $C$, depending on $n,m$,  such that the following inequalities are satisfied.
\begin{itemize}
\item[(a)]
 For $1\leq i \leq m-1$ and $\la \in \Ga_m$
\[\notag
	\frac{\la_{1} \cdots \la_{m}}{\la_{i}} \leq  C (S_{m-1;i}  \, S_{m-1})^\frac{1}{2}.
\]
\item
[(b)] Generally, for $1\leq \ell \leq n$ and increasing multi-indices $(i_1,...,i_{m-1})$,
$$
	 \prod_{i_s\neq \ell; s=1}^{m-1} |\la_{i_s}|  \leq C \left(S_{m-1;\ell} S_{m-1} \right)^\frac{1}{2}.
$$
\end{itemize}
\end{lem}

\begin{proof}  (a) The inequality is equivalent to saying that there exist uniform constants $c_1,c_2>0$ such that for  every positive number $a$, 
\[\label{eq:CS-elementary}
	a \frac{\la_{1} \cdots \la_m}{\la_i} \leq c_1 a^2 S_{m-1;i} + c_2 S_{m-1}, 
\]
where  $1 \leq i \leq m-1$. In fact as we will see in the proof  $c_1, c_2$ are explicitly given constants. We observe that if $a \leq 1$, then we can easily get the claim as
$$\la_1 \cdots \la_m /\la_i \leq \la_1 \cdots \la_{m-1} \leq S_{m-1}.$$
Now we consider $a> 1$. We prove (a) for the case $i=1$, the other cases $1\leq i \leq m-1$ follow in the same way. The basic identities/inequalities  are
\[\label{eq:i1}
\begin{aligned}
	S_{m;1m}+ \la_1 S_{m-1;1m} = S_{m;m} &= S_m - \la_m S_{m-1;m} \\
&\geq -\la_m S_{m-1;m},
\end{aligned}\]
and
\[\label{eq:i2}	S_{m-1;1m} + \la_1 S_{m-2;1m} = S_{m-1;m}.
\]
Multiplying \eqref{eq:i1} by $S_{m-2;1m}$ and \eqref{eq:i2} by $S_{m-1;1m}$ to eliminate $\la_1$ we get that
\[\label{eq:i3}
\begin{aligned}
	S_{m-1;1m}^2 - S_{m;1m} S_{m-2;1m} 
&\leq S_{m-1;m} (S_{m-1;1m} + \la_m S_{m-2;1m}) \\
&= S_{m-1;m} S_{m-1;1}.
\end{aligned}
\]
The Newton inequality holds for every $\la \in \bR^n$ and tells us
$$
	S_{m;1m} S_{m-2;1m} \leq \frac{(m-1) (n-m+1)}{m (n-m+2)} [S_{m-1;1m}]^2 =: c_{m}  [S_{m-1;1m}]^2.
$$
Notice that $0< c_m <1$. Hence, we derive from the above and \eqref{eq:i3} that
$$
	S_{m-1;1m}^2 - c_{m} [S_{m-1;1m}]^2 \leq S_{m-1;m} S_{m-1;1}.
$$
Therefore,
\[\label{eq:key}
	S_{m-1;1m}^2 \leq \frac{1}{1- c_{m}}S_{m-1;m} S_{m-1;1}.
\]
Using  Cauchy-Schwarz' inequality,  the inequality \eqref{eq:key} and the formula 
$$S_{m-1;1m} = S_{m-1;1} - \la_m S_{m-2;1m},$$ we get
$$
\begin{aligned}
 	a^2 S_{m-1;1} + \frac{1}{4(1-c_m)} S_{m-1;m} 
&\geq a \left[\frac{S_{m-1;1} S_{m-1;m}}{1-c_m}\right]^\frac{1}{2} \\
&\geq  a |S_{m-1;1m}| \\
&\geq a (-S_{m-1;1} + \la_m S_{m-2;1m}).
\end{aligned}
$$
This implies the inequality
\[\label{eq:basic}
	(a^2 + a) S_{m-1;1} + \frac{1}{4(1-c_m)} S_{m-1;m}  \geq  a \la_m S_{m-2;1m}.
\]
Since  $a\geq 1$, it follows that
$2a^2 S_{m-1;1} +  C S_{m-1} \geq  a \la_m S_{m-2;1m}.$
So, using Lemma~\ref{lem:sym-e} and \eqref{eq:Ivochina} we have
\[\begin{aligned} 
 (2 a^2 S_{m-1;1} + C S_{m-1} ) \la_1 
 &\geq  a \la_m \la_1 S_{m-2;1m}  \\
 &\geq a \te^2 \la_m S_{m-1} \\
 &\geq a \te^2 \la_1 \cdots \la_m.
\end{aligned}\]
The proof of the lemma is completed with   $c_1 = 2$ and $c_2 = \frac{1}{4(1-c_m)}$ and 
$C = \sqrt{2/(1-c_m)}.$

(b) The characterization \eqref{eq:char-Ivochina} implies that a sum of  any $(n-m+1)$ entries of $\la$ is positive. Hence, we have for $\la_{i_s} \leq 0$, 
$$|\la_{i_s}| \leq  (n-m)\la_{m}.$$
So,  for $1\leq \ell \leq m-1$, 
\[\label{eq:point-3} \begin{aligned}
	 \prod_{i_s \neq \ell; s=1}^{m-1} |\la_{i_s}|
\leq 	(n-m)^{m-1} \frac{ \la_1\cdots \la_m}{\la_\ell} 
\leq 	C \left[ S_{m-1;\ell} S_{m-1}\right]^\frac{1}{2},
\end{aligned}\]
where we used (a) for the second inequality. 

Now we treat the remaining range $m\leq \ell \leq n$. By a result of Lin and Trudinger \cite[Theorem~1.1]{LT94}, we know that $S_{m-1;\ell} \geq \te S_{m-1}$ for a constant $\te =\te(n,m)$ depending only on $n,m$. This implies $$S_{m-1} \leq (S_{m-1;\ell} S_{m-1})^\frac{1}{2}/\sqrt{\te}.$$ Thus, the desired inequality easily follows from this and the bound
\[\label{eq:point-4}\begin{aligned}
	   \prod_{i_s \neq \ell; s=1}^{m-1} |\la_{i_s}| 
\leq 	(n-m)^{m-1} \la_1 \cdots \la_{m-1} 
\end{aligned}
\]
for $m\leq \ell \leq n$.
The proof of (b) is completed.
\end{proof}

\subsection{Cauchy-Schwarz's inequality}

Let $\om$ be a Hermitian metric on $\bC^n$ and let $\Om$ be a bounded open set in $\bC^n$. The positive cone $\Ga_m(\om)$, associated to $\om$, of real $(1,1)$-forms is defined as follows. A real form $\ga$ is said to belong to $\Ga_m(\om)$ if at any point $z\in \Om$,
$$
	\ga^k \wed \om^{n-k} (z)>0  \quad\text{for } k=1,...,m.
$$
Equivalently, in the normal coordinate system with respect to $\om$ at $z$, diagonalizing $\ga = \ii \sum_i \la_i dz_i\wed d\bar z_i$, we have $\la = (\la_1,...,\la_n) \in \Ga_m$. Now we will translate the estimates in Section~\ref{ssec:positive-cone} into the integral forms.

We can state the following versions of Cauchy-Schwarz's inequality in this setting. 
Let $h$ be a smooth real-valued function and $\phi, \psi$ be Borel functions. Let $T$ be a positive current of bidegree $(n-2,n-2)$. 
\begin{lem}\label{lem:CS-classic} There exists a uniform constant $C$ depending on $\om$ such that
$$\begin{aligned}
&	\left |\int \phi\psi\; dh \wed d^c \om \wed T \right|^2 
&\leq 	C\int |\phi|^2 \;dh\wed d^c h \wed \om\wed T 
  \int |\psi|^2\; \om^2 \wed T. 
\end{aligned}$$
\end{lem}
\begin{proof} The proof of \cite[Proposition~1.4]{Ng16} can be easily adapted.
\end{proof}

The above lemma can be applied  for the case $T= \ga^{s} \wed \om^{n-m + \ell}$, where $\ga\in \Ga_m(\om)$ and  $0\leq s, \ell\leq m-1$ and $s+\ell = m-1$. Next, we also need to deal with possible non-positive forms  $T' = \ga^{m-1} \wed \om^{n-m-1}$ where the classical Cauchy-Schwarz is not immediately applicable. However, we still have

 \begin{lem} \label{lem:CS} 
There exists a uniform constant $C$ depending on $\om, n,m$ such for every $\ga \in \Ga_{m}(\om)$,
$$\begin{aligned}
&	\left |\int \phi\psi\; dh \wed d^c \om \wed \ga^{m-1}  \wed  \om^{n-m-1} \right|^2 \\
&\leq 	C\int |\phi|^2 \;dh\wed d^c h \wed \ga^{m-1} \wed \om^{n-m}  \times  \int |\psi|^2\;\ga^{m-1} \wed \om^{n-m+1}.
\end{aligned}$$
 \end{lem}

\begin{proof} We express the integrands of both sides as follows.
$$\begin{aligned}
	dh \wed d^c \om \wed \ga^{m-1} \wed \om^{n-m-1} = f_1(z) \om^n, \\
	dh \wed d^c h \wed \ga^{m-1} \wed \om^{n-m} = [f_2(z)]^2 \om^n,\\
	\ga^{m-1} \wed \om^{n-m+1} = [f_3(z)]^2 \om^n.
\end{aligned}$$
Thus, the inequality will follow from the classical Cauchy-Schwarz inequality if we have  point-wise  $|f_1(z)| \leq C f_2(z) f_3(z)$ for every $z\in \Om$. This is proved by using the normal coordinate system at a given point $z$ with respect to $\om$  which diagonalizes also $\ga$, i.e., 
$$
 \om = \ii \sum_{i=1}^n  dz_i \wed d\bar z_i, \quad \ga = \ii \sum_i \la_i dz_i \wed d\bar z_i, 
$$
where $\la =(\la_1, ..., \la_n) \in \Ga_m$. Denote $h_i = \d h/\d z_i$. Then,  at the point $z$,
\[\label{eq:point-1}
	\binom{n}{m-1} (f_2)^2 = \sum_{i=1}^n |h_i|^2 S_{m-1;i},\quad \binom{n}{m-1} (f_3)^2 =  S_{m-1}.
\]
Now, observe that $\ga^{m-1}\wed \om^{n-m-1}$ is a $(n-2,n-2)$ form, so after taking the wedge product with $dh\wed d^c\om$ the non-zero contribution give only  $\ii \d h \wed \dbar \om$ and $\ii \;\dbar h \wed \d \om$. As $h$ is a real valued function, these two forms are mutually  conjugate. Let us write 
$$\dbar \om = \sum \om_{i\bar j \bar k} d\bar z_k \wed dz_i \wed d\bar z_j.$$
Denote $dV= (\ii)^{n^2} dz_1 \wed \cdots dz_n \wed d\bar z_1 \wed \cdots d\bar z_n.$ 
Let $J = (j_1,...,j_{n-m-1})$ be an increasing multi-index. Then, 
$$
	\frac{1}{(m-1)!}\d h \wed \dbar \om \wed \ga^{m-1} \wed dz_J \wed d\bar z_J/ dV = \sum_{j,\ell \not\in J;j\neq \ell}  c_{j \bar j \bar \ell} \;h_\ell \prod_{i_s \not\in J\cup\{j,\ell \}} \la_{i_s},
$$
where $c_{i\bar i \bar \ell}$ is $\om_{i\bar i \bar \ell}$ or $\om_{i \bar \ell \, \bar i}$, and $i_s \in I = (i_1,..,i_{m-1})$ which is an increasing multi-index satisfying $I \cap J = \emptyset$.
Then, at the point $z$,
\[\label{eq:point-2}
	|f_1(z)| \leq  c_0  \sum_{|I|=m-1} \sum_{\ell =1}^n  |h_\ell|   \prod_{i_s \neq \ell; s=1}^{m-1} |\la_{i_s}|,
\]
where $c_0$ is a uniform constant depending only on $\om$. 

By \eqref{eq:point-1} and \eqref{eq:point-2} we have reduced the inequality $|f_1| \leq C f_2f_3$  to the one for symmetric polynomials. To show the latter, by Lemma~\ref{lem:sym-ineq}-(b)
for $1\leq \ell \leq n$, we have
$$\begin{aligned}
	|h_\ell|  \prod_{i_s \neq \ell; s=1}^{m-1} | \la_{i_s}|  
&\leq 	C \left[ |h_\ell|^2 S_{m-1;\ell} S_{m-1}\right]^\frac{1}{2} \\
&\leq 	C \left( \sum_{i=1}^n |h_i|^2 S_{m-1;i} \right)^\frac{1}{2} [S_{m-1}]^\frac{1}{2}. 
\end{aligned}$$
Taking the sum of (finitely many)  terms on the  left hand side of  \eqref{eq:point-2} the proof of the theorem follows.
\end{proof}

We also need this inequality for wedge products of two forms and more. This is done by solving a  linear system of inequalities  as in  [KN16, page 2226]. 

\begin{cor}\label{cor:CS} There exists a uniform constant $C$, depending on $\om,n, m$, such that the following inequalities hold.
\begin{itemize}
\item
[(a)]
For $\eta, \ga \in \Ga_m(\om)$, 
$$
\begin{aligned}
&\left| \int \phi \psi  dh \wed d^c \om \wed \eta^k \wed \ga^{m-k-1} \wed \om^{n-m-1}\right|^2  \\
& \leq C \int  |\phi|^2 dh \wed d^c h \wed (\eta+\ga)^{m-1} \wed \om^{n-m}
\int |\psi|^2 (\eta + \ga)^{m-1} \wed \om^{n-m+1}.
\end{aligned} 
$$
\item
[(b)] Generally, for $\ga_1,...,\ga_{m-1} \in \Ga_m(\om)$, 
$$
\begin{aligned}
&\left| \int\phi\psi dh \wed d^c \om \wed \ga_1 \wed \cdots \wed \ga_{m-1} \wed \om^{n-m-1}\right|^2  \\
& \leq C \int |\phi|^2\; dh \wed d^c h \wed (\sum_{i=1}^{m-1}\ga_i )^{m-1} \wed \om^{n-m}  \times \\
&\qquad \times\int |\psi|^2 ( \sum_{i=1}^{m-1} \ga_i )^{m-1} \wed \om^{n-m+1}.
\end{aligned} 
$$
\end{itemize} 
\end{cor}

\subsection{$m$-subharmonic functions on Hermitian manifolds}

Let us recall the definition of generalized $m$-subharmonic function in the Hermitian setting (see \cite{Bl05}, \cite{KN3} and \cite{GN18}). Let $\Om$ be a bounded open set in $\bC^n$ and let $\om$ be a Hermitian metric on $\Om$.

\begin{defn} An upper semi-continuous function $u:\Om \to [-\infty,+\infty)$ is called $m-\om$-subharmonic if $u\in L^1_{\rm loc}(\Om)$ and for any collection $\ga_1,...,\ga_{m-1} \in \Ga_m(\om)$, 
$$
	dd^c u \wed \ga_1 \wed \cdots \wed \ga_{m-1} \wed \om^{n-m}\geq 0
$$
in the sense of currents.
\end{defn}

\begin{remark}\label{rmk:Garding}
By G\aa rding's \cite{Ga59}  results  a function $u\in C^2(\Om)$ is $m-\om$-sh if and only if $dd^c u \in \ov{\Ga_m(\om)}$, that is
$
	dd^c u \wed \ga_1 \wed \cdots \wed \ga_{m-1} \geq 0
$
point-wise in $\Om$. Thus, the estimates for forms in $\Ga_m(\om)$ are applicable to $dd^cu$ if $u$ is a smooth $m-\om$-sh function.
\end{remark}

 It follows from Michelsohn \cite{Mi82} that  for $\ga_1,...,\ga_{m-1} \in \Ga_m(\om)$ there is a unique $(1,1)$ positive form $\al$ such that 
$$
	\al^{n-1} = \ga_1 \wed \cdots \wed \ga_{m-1} \wed \om^{n-m}.
$$
The above definition of $m-\om$-sh function can be expressed more familiarly, in terms of potential theory, using  the notion of  $\al$-subharmonicity (see e.g., \cite[Definition~2.1, Lemma~9.10]{GN18}). Thanks to this many potential-theoretic properties of $m-\om$-sh functions can be derived from those of $\al$-sh functions. For example, if two $m-\om$-sh  functions are equal almost everywhere (with respect to the Lebesgue measure), then they are equal everywhere \cite[Corollary~9.7]{GN18}.
One also has the "gluing property" that allows to modify the function outside a compact subset.  This is an immediate consequence of \cite[Lemma~9.5]{GN18}.

\begin{lem} Let $U \subset \Om$ be two open sets in $\bC^n$. Let $u$ be $m-\om$-sh in $U$ and $v$ be $m-\om$-sh in $\Om$. Assume that $\limsup_{z\to x} u(z) \leq v(x)$ for every $x \in \d U \cap \Om$. Then, the function 
$$ \wt u
= \begin{cases} \max\{u, v\} &\quad\text{on } U, \\
	v &\quad\text{on } \Om\setminus U,
\end{cases}
$$
is $m-\om$-sh in $\Om$.
\end{lem}

Because of this we have the following way of reducing a proof to a simpler case (see \cite[page 7]{K05} for the proof).

\begin{thmx} If we are to prove the weak convergence or local estimate for a family of locally uniformly bounded $m-\om$-sh functions, it is no loss of generality to we assume that the functions are defined in a ball and are all equal on some neighborhood of the boundary.
\end{thmx}

For a bounded psh function its convolution with a radial smoothing kernel provides locally a smooth, decreasing approximation of this function. It is no longer the case for generalized $m-\om$-sh. However, in a strictly  $m$-pseudoconvex domain $\Om$, that is for $\Om = \{\rho <0\}$, where  $\rho \in C^2(\ov \Om)$  is strictly $m-\om$-sh and $d \rho \neq 0$ on $\d\Om$,  
we still have (non-explicit) way of approximation by smooth $m-\om$-sh  functions.

\begin{prop}\label{prop:smoothing} Let $\Om \subset \subset \bC^n$ be  strictly $m$-pseudoconvex domain. Let $u$ be a bounded $m-\om$-sh function in a neighborhood of $\ov\Om$. Then, there exists a sequence of smooth $m-\om$-sh functions $u_j \in C^\infty(\ov\Om)$ that decreases to $u$ point-wise in $\ov\Om$.
\end{prop}

\begin{proof} The proof is the same as the one of \cite[Theorem~3.18]{GN18} if we replace the ball by a strictly $m$-pseudoconvex domain and invoke \cite[Theorem~1.1]{CP22} for the smooth solution of the Dirichlet problem on a strictly $m$-pseudoconvex domain.
\end{proof}

\subsection{Integral estimates for smooth functions}
\label{ssec:integral-smooth-fct}

Let $\Om$ be a bounded open set in $\bC^n$. Let $-1 \leq v \leq w\leq 0$ be smooth $m-\om$-sh functions in $\Om$ such that $v=w$ in a neighborhood of $\d\Om$.  Let $\rho$ be a smooth $m-\om$-sh function such that $-1 \leq \rho \leq 0$. 
Using the notation $\ga := dd^c \rho$, $h = w-v$  we consider
\[\label{eq:integral-fct}
	e_{(q,k,s)} = \int h^{q+1} \ga^{k} \wed (dd^c v)^s \wed \om^{n-k-s}, 
\]
where $q\geq 0$, the integers $0\leq k \leq m$ and $0\leq s \leq m-k$.  We wish to bound 
$$
	e_{(q,m, 0)} = \int h^{q+1} \ga^m \wed \om^{n-m},
$$
by the integrals 
$$
	e_{(r, 0, i)} = \int h^{r+1} (dd^cv)^i \wed \om^{n-i},
$$
where  $i=0,...,m$ and $1\leq r < q$.
This is done via repeated use of  the integration by parts to replace $\ga$ by $dd^c v$. 
However, there are three different cases depending on the total degree $k+s$ of the form $\ga^k\wed (dd^cv)^s$ that we need to deal with separately.

\begin{itemize}
\item Case 1: $k+s =m$; 
\item Case 2: $k+s = m-1$;
\item Case 3: $k+s \leq m-2$.
\end{itemize}

Let us start with an auxiliary inequality that we use frequently bellow. 
\begin{lem} \label{lem:CLN-grad} 
Let $p\geq 1$ and $0\leq k \leq m-1$. There exists a constant $C$ depending on $\om, n,m$ such that 
\begin{itemize}
\item
[(a)] for $0\leq s \leq m-1-k$: 
$$\begin{aligned}
&	\int h^{p-1}dh \wed d^c h \wed \ga^k \wed (dd^c v)^{s}  \wed \om^{n-k-s-1} \\
&\leq 	\int h^{p} \ga^k\wed(dd^cv)^{s+1} \wed \om^{n-k-s-1}  
+ C  \int h^{p+1} (\ga+dd^cv)^{k+s} \wed \om^{n-k-s}.
\end{aligned}$$
\item
[(b)] for  $0\leq s \leq m-3-k$:
$$\begin{aligned}
&	\int h^{p-1}dh \wed d^c h \wed \ga^k \wed (dd^c v)^{s}  \wed \om^{n-k-s-1} \\
&\leq 	\int h^{p}  \ga^k\wed(dd^cv)^{s+1} \wed \om^{n-k-s-1}  + C  \int h^{p+1} \ga^k \wed (dd^cv)^{s} \wed \om^{n-k-s}.
\end{aligned}$$
\end{itemize}
 \end{lem}

\begin{proof}  (a) Note first that $0\leq h \leq 1$ and also $T := \ga^k \wed (dd^c v)^s \wed \om^{n-k-s-1}$ and $dd^cw \wed T$ are positive forms for $n-s-k-1 \geq n-m$. Therefore,
$$\begin{aligned}
	p(p+1) h^{p-1} dh\wed d^c h  \wed T
&=	 [dd^c h^{p+1} - (p+1) h^{p} dd^c h] \wed T \\
&\leq		 [dd^c h^{p+1} + (p+1) h^{p} dd^c v] \wed T.
\end{aligned}$$  Hence, 
\[\label{eq:grad1}\begin{aligned}
&	\int h^{p-1} dh\wed d^c h \wed \ga^k\wed(dd^cv)^{s} \wed \om^{n-s-k-1} \\
&\leq \int (dd^c h^{p+1} + h^{p} dd^cv) \wed \ga^k\wed (dd^cv)^{s}  \wed \om^{n-s-k-1}.
\end{aligned}\]
It remains  to estimate the product involving the first term in the bracket. By integration by parts and [Lemma 2.3, KN16],
\[\label{eq:grad2} \begin{aligned}
&	\int dd^c h^{p+1} \wed \om^{n-s-k-1} \wed \ga^k\wed(dd^cv)^{s} \\
&	= \int h^{p+1} dd^c (\om^{n-s-k-1}) \wed \ga^k\wed(dd^cv)^{s}  \\
&\leq C \int h^{p+1} (\ga+ dd^cv)^{k+s} \wed \om^{n-m+1}.
\end{aligned}\]
Combining the last two inequalities the proof of the lemma follows.

(b) The proof is very similar, we first have \eqref{eq:grad1}. Then, in the middle integral of  \eqref{eq:grad2}
one can express $	dd^c (\om^{n-s-k-1}) = \eta \wed \om^{n-m}
$ for a smooth $(m-s-k,m-s-k)$-form $\eta$. Then, since $\ga, dd^c v \in \Ga_m(\om)$, in this case the inequality has a better form
$$
	\left|\int h^{p+1} \eta \wed \ga^k \wed (dd^c v)^s \wed \om^{n-m} \right| \leq C \int h^{p+1} \ga^k \wed (dd^c v)^s \wed \om^{n-k-s}.
$$
The item (b) is proven.
\end{proof}

Let us start with the simplest situation in Case 1. We are going to show that 
\[ \label{eq:case1}
	e_{(q,m,0)} \leq C \left(e_{(q-1, m-1, 1)} +  e_{(q-1, m-1,0)} \right),
\]
where $C$ is a uniform constant depending on $\om, m,n$.
Equivalently, 
\begin{lem}\label{lem:case1} Let $q\geq 1$. Then,
$$\begin{aligned}
	\int_\Om (w-v)^{q+1} \ga^{m} \wed \om^{n-m}  &\leq  C \int_\Om (w-v)^{q}  \ga^{m-1}  \wed dd^c v \wed \om^{n-m} \\
&\quad + C \int_\Om (w-v)^{q} \ga^{m-1} \wed \om^{n-m+1}.
\end{aligned}$$
\end{lem}

\begin{proof} 
Recall that $h := w-v \geq 0$ and $h=0$ near $\d\Om$. A direct computation gives
\[\label{eq:basic-computation}\begin{aligned}
&	dd^c (h^{q+1}\om^{n-m}) \\
&= q(q+1) h^{q-1} dh \wed d^c h\wed  \om^{n-m} \\
&\quad+ (q+1) h^{q} dd^c h \wed \om^{n-m} \\ 
&\quad+ (q+1) (n-m) h^{q} d\om \wed d^c h \wed \om^{n-m-1} \\
&\quad+ (q+1) (n-m) h^{q}dh\wed d^c \om \wed \om^{n-m-1} \\
&\quad + h^{q+1} dd^c \om^{n-m} \\
&\quad =: T_0 + T_1 + T_2 + T_3 +T_4.
\end{aligned}\]
By integration by parts,
\[\label{eq:int-by-parts}\begin{aligned}
	\int h^{q+1} dd^c \rho \wed \ga^{m-1} \wed \om^{n-m} 
&= \int \rho dd^c (h^{q+1} \om^{n-m})  \wed \ga^{m-1}  \\
&= \int \rho (T_0 + T_1 + T_2 + T_3 + T_4) \wed \ga^{m-1}.
\end{aligned}\]

{\bf Case 1a: Estimate of $T_0$.} Since $-1 \leq \rho \leq 0$ and $T_0$ is a positive current,
\[\label{eq:T0} \rho T_0 \wed \ga^{m-1} \leq 0 \] 

{\bf Case 1b: Estimate of $T_1$}.
Similarly, because $v$ is a $m-\om$-sh function,
\[\label{eq:T1}
\begin{aligned}
	\rho T_1 \wed \ga^{m-1}
&= (q+1) \rho h^{q} (dd^c w - dd^c v) \wed \om^{n-m} \wed \ga^{m-1} \\
& \leq (q+1) h^{q} dd^c v \wed \ga^{m-1} \wed\om^{n-m}.
\end{aligned}\]

{\bf Case 1c: Estimate of $T_4$.}
Using the inequality [Lemma 2.3, KN16]   
\[\label{eq:T4-rough} 
	\ga^{m-1} \wed dd^c \om^{n-m} \leq C_{m,n} \ga^{m-1} \wed \om^{n-m+1},
\] 
we can estimate the last term $T_4$,
\[\label{eq:T4}\begin{aligned}
	\left| \int \rho h^{q+1} \ga^{m-1} \wed  dd^c \om^{n-m} \right| 
&\leq 	C\int |\rho| h^{q+1} \ga^{m-1} \wed \om^{n-m+1}  \\
&\leq C e_{(q, m-1,0)},
\end{aligned}
\]
where we used again the fact that $|\rho| \leq 1$. 

{\bf Case 1d: Estimate of  $T_2$ and $T_3$.} We use Cauchy-Schwarz' inequality in Lemma~\ref{lem:CS}, where an extra uniform constant  appears.  Let us estimate   $T_2$ (for  $T_3$ the estimate  is completely the same). By Cauchy-Schwarz' inequality 
\[\label{eq:T23} \begin{aligned}
&	\left |\int \rho h^{q} d\om \wed d^c h \wed \om^{n-m-1} \wed \ga^{m-1} \right|^2   \\
&\leq	 C \int |\rho| h^{q-1} dh\wed d^c h \wed \om^{n-m} \wed \ga^{m-1} \int |\rho| h^{q+1} \om^{n-m+1} \wed \ga^{m-1}  \\
&\leq C \left( \int h^{q-1} dh\wed d^c h \wed \ga^{m-1} \wed \om^{n-m} + \int h^{q+1}  \ga^{m-1} \wed\om^{n-m+1}  \right)^2 \\
&\leq C \left( e_{(q-1, m-1,1)} + e_{(q,m-1,0)}\right)^2
\end{aligned}\]
where we used Lemma~\ref{lem:CLN-grad} with $s=0, k=m-1$ and $p=q$ for the first integral in the  last inequality, namely,
\[\notag
	\int h^{q-1} dh \wed d^c h \wed \om^{n-m} \wed \ga^{m-1} \leq C (e_{(q-1, m-1,1)} + e_{(q,m-1,0)}).
\]

Combining the estimates \eqref{eq:T0}, \eqref{eq:T1}, \eqref{eq:T4} and \eqref{eq:T23} and the fact that $e_{(q, \bullet, \bullet)} \leq e_{(q-1, \bullet, \bullet)}\leq e_{(q-2, \bullet, \bullet)}$ we conclude the proof of the lemma.
\end{proof}

We can now state the general inequality in Case 1.

\begin{cor} \label{cor:case1} For $1 \leq k \leq m$ and $s =m-k \geq 0$ and $q\geq 2$,  
$$
	e_{(q,k,s)}   \leq   c_k \sum_{i=0}^{k-1} e_{(q-1, i, m-i)} + C \sum_{i=0}^{m-1}e_{(q-2, i, m-1-i)}.
$$
\end{cor}

\begin{proof} The proof is by induction in $k$ but "downward".  For $k=m$ it is Lemma~\ref{lem:case1}. Assume that it is true for every $k+1 \leq \ell \leq m$, i.e.,  we have 
\[\label{eq:case1-induction}
	e_{(q,\ell, m-\ell)} \leq c_{\ell} \sum_{i=0}^{k} e_{(q-1, i, m-i)} + c_\ell \sum_{i=0}^{m-1} e_{(q-2,i,m-1-i)}.
\]
We proceed to prove the conclusion holds for $k$. 
The strategy is  the same as in  the proof  of the last  lemma, however we need to estimate $T_2$ and $T_3$ more carefully.
We repeat the steps of the proof of Lemma~\ref{lem:case1}   replacing  $dd^c\rho \wed \Ga$ by $dd^c\rho \wed \Ga^{(s)}$, where 
$$\Ga=\ga^{m-1}\wed \om^{n-m} \text{ and }\Ga^{(s)}= \ga^{m-1-s} \wed (dd^c v)^s \wed \om^{n-m} ,$$ in the integrand on the left hand side. The corresponding estimates for $T_0, T_1$ are similar. Namely,
\[
\tag{\ref{eq:T0}$'$}
	\rho T_0 \wed \ga^{m-1-s}\wed (dd^cv)^s \leq 0, 
\]
and
\[\tag{\ref{eq:T1}$'$} \rho T_1\wed \ga^{m-s-1} \leq   (q+1) h^q (dd^cv)^{s+1} \wed \ga^{m-1-s}  \wed\om^{n-m} .
\]
The one  for $T_4$ is 
\[\tag{\ref{eq:T4-rough}$'$} \begin{aligned}
&	\ga^{m-1-s} \wed (dd^c v)^s \wed dd^c \om^{n-m} \\
&\leq C_{m,n} (\ga + dd^cv)^{m-1} \wed \om^{n-m+1} \\
& \leq C_{m,n}  [\ga^{m-1} + \ga^{m-2}
\wed dd^c v+  \cdots + (dd^cv)^{m-1}] \wed \om^{n-m+1}.
\end{aligned}\]
Hence, integrating both sides and using $0\leq h\leq 1$, leads to
\[
\tag{\ref{eq:T4}$'$} 	
	\left|\int h^{q+1} \ga^{m-1-s}\wed (dd^c v)^{s} \wed dd^c \om^{n-m} \right|\leq C\sum_{i=0}^{m-1}e_{(q,i,m-1-i)}.
\]

Lastly for $T_2$ and $T_3$, we need to use Corollary~\ref{cor:CS}, 
$$
\begin{aligned}
&	I^2:=\left|\int \rho h^q d\om \wed d^c h \wed \om^{n-m-1} \wed \ga^{m-1-s} \wed (dd^cv)^s\right|^2  \\
&\leq  C \int  h^{q-1} (\ga+ dd^cv)^{m-1} \wed \om^{n-m+1} \\
&\qquad \times \int h^{q+1} dh\wed d^c h \wed (\ga+ dd^c v)^{m-1}\wed \om^{n-m} \\   
\end{aligned}
$$
The standard  Cauchy-Schwarz inequality (Lemma~\ref{lem:CS-classic}) gives for $\veps>0$ to be determined later,
$$\begin{aligned}
	I 
&\leq \frac{C}{\veps}  \int  h^{q-1} (\ga+ dd^cv)^{m-1} \wed \om^{n-m+1} \\
&\qquad+  \veps \int h^{q+1} dh\wed d^c h \wed (\ga+ dd^c v)^{m-1}\wed \om^{n-m}.
\end{aligned}$$
By the last inequality in (\ref{eq:T4-rough}$'$) the first integral is bounded by
$$
	\frac{C}{\veps} \sum_{i=0}^{m-1} e_{(q-2, i, m-1-i)}.
$$
To bound the second integral we use 
$$ (\ga+dd^cv)^{m-1} \wed\om^{n-m}\leq C_{m,n} \sum_{i=0}^{m-1} \ga^i \wed (dd^cv)^{m-1-i} \wed \om^{n-m},$$
and then Lemma~\ref{lem:CLN-grad}-(a) for $k=i, s= m-1-i$ and $p=q+2$.
This gives a bound for the second integral by $$ C\veps  \sum_{i=0}^{m-1} e_{(q+1,i,m-i)} + C\veps\sum_{i=0}^{m-1} e_{(q+2, i, m-1-i)}.$$  Let us consider the first sum above:
$$
	\veps \sum_{i=0}^{m-1} e_{(q+1,i,m-i)} =  \veps \sum_{i\geq k+1} e_{(q+1,i,m-i)} + \veps \sum_{i=0}^{k-1} e_{(q+1,i,m-i)} + \veps e_{(q+1,k,m-k)}.
$$
Applying the induction hypothesis \eqref{eq:case1-induction} 
 to the first term on the right, we derive
$$\begin{aligned}
	\veps \sum_{i=0}^{m-1} e_{(q+1,i,m-i)} 
&\leq  \veps b_k\left(  e_{(q, k,m-k)} + \sum_{i=0}^{k-1} e_{(q,i,m-i)} \right) +  \veps b_k \sum_{i=0}^{m-1} e_{(q-1,i,m-1-i)} \\
&\qquad + \veps \sum_{i=0}^{k-1} e_{(q+1,i,m-i)} +  \veps e_{(q+1,k,m-k)},
\end{aligned}$$
where $b_k =  c_m+\cdots + c_{k+1}$. 

Since $e_{(q+2,\bullet,\bullet)} \leq e_{(q+1, \bullet, \bullet)} \leq e_{(q, \bullet, \bullet)}\leq e_{(q-1, \bullet, \bullet)}$, it follows from the above   estimates  that
\[\tag{\ref{eq:T23}$'$}
\begin{aligned}
	I
&\leq \veps (1+b_k)\;  e_{(q, k,m-k)} + \veps \sum_{i=0}^{k-1} e_{(q-1,i,m-i)} \\
&\qquad + [ \veps (b_k+1) + C/\veps + C\veps] \sum_{i=0}^{m-1} e_{(q-2,i,m-1-i)}.
\end{aligned}\]

Thus, 
combining (\ref{eq:T0}$'$), (\ref{eq:T1}$'$), (\ref{eq:T4}$'$) and (\ref{eq:T23}$'$) we have 
$$\begin{aligned}
	e_{(q,k,s)}
& \leq    C e_{(q-1, k-1, s+1)} + C \sum_{i=0}^{m-1} e_{(q,i,m-1-i)} \\
&\quad + \veps (1+b_k) e_{(q,k,s)}+ \veps \sum_{i=0}^{k-1} e_{(q-1,i,m-i)} \\
&\quad+ [ \veps (b_k+1) + C/\veps + C\veps ] \sum_{i=0}^{m-1} e_{(q-2,i,m-1-i)}.
\end{aligned}$$
Now we can choose $\veps$ so that $\veps(1+b_k) =1/2$. Since $s=m-k$,  regrouping terms on the right hand side (decreasing the first parameter in $e_{(q, i, m-1-i)}$ if necessary) we get for possibly larger $C>0$ that
$$
	e_{(q,k,m-k)}/2 \leq (C+\veps) \sum_{i=0}^{k-1} e_{(q-1,i,m-i)} + (C + C/\veps) \sum_{i=0}^{m-1} e_{(q-2,i,m-1-i)}.
$$
The induction step is proven  and the lemma follows.
\end{proof}

Next, we consider Case 2.

\begin{lem}\label{lem:case2} For $1 \leq k \leq m-1$ and $s = m-1-k \geq 0$ and $q\geq 1$, we have 
$$
	e_{(q,k,s)} \leq C \left(e_{(q-1,k-1,s+1)}  + \sum_{i=0}^{m-2}e_{(q-1, i, m-2-i)} \right).
$$
\end{lem}

\begin{proof} 
The basic computation using integration  by parts that corresponds to \eqref{eq:basic-computation} starts with
\[\tag{\ref{eq:basic-computation}$''$}
	dd^c (h^{q+1} \om^{n-m+1}):= T_0 + T_1 + T_2 + T_3 + T_4,
\]
where each term has higher exponent of $\om$.
The estimates for $T_0, T_1$ are the same as the ones in (\ref{eq:T0}$'$) and (\ref{eq:T1}$'$). Precisely,
\[\tag{\ref{eq:T0}$''$}
	\rho T_0\wed \ga^{m-2-s} \wed (dd^cv)^s \leq 0,
\]
and
\[\tag{\ref{eq:T1}$''$}
	\rho T_1 \wed \ga^{m-2-s} \leq (q+1) h^q \ga^{m-2-s} \wed (dd^cv)^{s+1} \wed \om^{n-m+1}.
\]
The one for $T_4$ is 
\[\tag{\ref{eq:T4-rough}$''$} 
\begin{aligned}
&\ga^{m-2-s} \wed (dd^cv)^s \wed dd^c (\om^{n-m+1}) \\
&\leq 	C_{m,n} (\ga+dd^c v)^{m-2} \wed \om^{n-m+2} \\
&\leq C_{m,n} [\ga^{m-2} + \ga^{m-3} \wed dd^c v + \cdots + (dd^cv)^{m-2}] \wed \om^{n-m+2}.
\end{aligned}\]
Integrating both sides  and using the fact that $0\leq h\leq 1$ yield
\[\tag{\ref{eq:T4}$''$}
	\left|\int \rho h^{q+1} \ga^{m-2-s} \wed (dd^cv)^s \wed dd^c (\om^{n-m+1})\right| \leq C \sum_{i=0}^{m-2} e_{(q,i,m-2-i)}.
\]

Next, the corresponding inequalities for $T_2$ and $T_3$ are easier. This is due to the fact that 
$$
	T_2 \wed \ga^{m-2-s} \wed (dd^cv)^s = C_0 h^q d\om \wed d^c h \wed \om^{n-m} \wed \ga^{m-2-s} \wed (dd^cv)^s.
$$
Therefore, the classical Cauchy-Schwarz inequality (Lemma~\ref{lem:CS-classic}) is sufficient. Namely,
\[\notag 
\begin{aligned}
&	I^2:=\left|\int \rho h^{q} d\om \wed d^c h \wed \om^{n-m} \wed \ga^{m-2-s} \wed (dd^c v)^s \right|^2 \\
& \leq C \int h^{q} \ga^{m-2-s}\wed (dd^c v)^s \wed\om^{n-m+2} \\ 
&\qquad \times  \int h^q dh \wed d^c h \wed \ga^{m-2-s}\wed (dd^cv)^s \wed \om^{n-m+1}. 
\end{aligned}\]
The Cauchy-Schwarz inequality  gives 
$$\begin{aligned}
	I 
&	\leq   C \int h^{q} \ga^{m-2-s}\wed (dd^c v)^s \wed\om^{n-m+2} \\
&\qquad +C \int h^q dh \wed d^c h \wed \ga^{m-2-s}\wed (dd^cv)^s \wed \om^{n-m+1}. 
\end{aligned}$$
Here, the first integral in the sum is $e_{(q-1,k-1, s)}$. By applying Lemma~\ref{lem:CLN-grad}-(a) for $k-1, s$ and $p=q+1$ one gets a bound for the second integral by
$$\begin{aligned}
&	\int h^{q+1} \ga^{m-2-s} \wed (dd^cv)^{s+1}\wed\om^{n-m+1} +C \int h^{q+2} (\ga+dd^cv)^{m-2} \wed \om^{n-m+2} \\
&\leq  e_{(q, k-1, s+1)} + C \sum_{i=0}^{m-2} e_{(q+1, i, m-2-i)}.
\end{aligned}$$
Combining this  with the decreasing property in the first parameter of $e_{(q, \bullet,\bullet)}$ we get
\[\tag{\ref{eq:T23}$''$}
\begin{aligned}
	I
&	\leq C [e_{(q-1, k-1, s)} + e_{(q,k-1,s+1)} + \sum_{i=0}^{m-2} e_{(q+1,i,m-2-i)} ] \\
& \leq C [e_{(q-1,k-1,s+1)} + \sum_{i=0}^{m-2} e_{(q-1,i,m-2-i)}].
\end{aligned}\]

Finally, combining  (\ref{eq:T0}$''$), (\ref{eq:T1}$''$), (\ref{eq:T4}$''$) and (\ref{eq:T23}$''$) one completes the proof of the lemma.
\end{proof} 

Lastly, we consider Case 3.

\begin{lem} \label{lem:case3} For $1\leq k \leq m-2$  and $0 \leq s \leq m-2-k$ and $q\geq 1$ we have
$$
	e_{(q,k,s)} \leq C \left[ e_{(q-1, k-1, s+1)} + e_{(q, k-1,s)} \right].
$$
\end{lem}

\begin{proof} We need to estimate  $dd^c \rho \wed \ga^{k-1} \wed (dd^cv)^s \wed \om^{n-k-s}$, where $n-k-s \geq n-m+2$. Then, there is a significant change in basic computation of
\[\tag{\ref{eq:basic-computation}$'''$}
	dd^c (h^{q+1} \om^{n-k-s}) = T_0 + T_1+ T_2+T_3+ T_4,
\]
where all forms $T_i$ contain powers of $\om$ with the exponent  at least $n-m$. 
The estimates for $T_0, T_1$ are the same as in Corollary~\ref{cor:case1} and improved estimates for $T_2, T_3$ are as in Lemma~\ref{lem:case2}. Moreover the bound for $T_4$ is easier. Namely,
since $\ga^{k-1} \wed (dd^c v)^s \wed \om^{n-m}$ is a positive form, one obtains 
\[\tag{\ref{eq:T4-rough}$'''$}
	\ga^{k-1}\wed (dd^c v)^s \wed dd^c (\om^{n-k-s}) \leq C \ga^{k-1} \wed (dd^cv)^s \wed \om^{n-k-s+1}.
\]
Hence, multiplying both sides by $\rho h^{q+1}$ and then integrating we get
\[\tag{\ref{eq:T4}$'''$}
\left| \int \rho h^{q+1} \ga^{k-1} \wed (dd^c v)^s \wed dd^c (\om^{n-k-s}) \right|  
\leq C e_{(q,k-1,s)},
\]
where we used the fact that $-1\leq \rho \leq 0$.

Next, the estimates for $T_2$ and $T_3$ are
\[\notag
\begin{aligned}	
I&:= \left| \int \rho h^q d\om \wed d^c h \wed  \om^{n-k-s-1} \wed \ga^{k-1}\wed (dd^cv)^s  \right|  \\
&\leq C \int h^{q+1} \ga^{k-1}\wed (dd^cv)^s \wed \om^{n-k-s} \\
&\qquad + C\int h^{q-1} dh\wed d^c h \wed \ga^{k-1} \wed (dd^cv)^s \wed \om^{n-k-s}.
\end{aligned}\]
Using Lemma~\ref{lem:CLN-grad}-(b) for $k-1, s$ and $p=q-1$ in the last  integral above yields
\[\tag{\ref{eq:T23}$'''$}
\begin{aligned}
	I 
&	\leq  Ce_{(q,k-1, s)} + C[e_{(q-1, k-1,s+1)} + e_{(q, k-1,s)}] \\
&	\leq C [e_{(q-1, k-1, s+1)} + e_{(q,k-1,s)}].
\end{aligned}
\]
Combining the estimates for $T_0, T_1$ in (\ref{eq:basic-computation}$'''$), (\ref{eq:T4}$'''$) and (\ref{eq:T23}$'''$) we  complete the proof of lemma.
\end{proof}

We are ready to state the main inequality.

\begin{prop}\label{prop:cap-integral} Let $e_{(q,k,s)}$ be the numbers defined by \eqref{eq:integral-fct}. Then, for $q = 3m$,
$$
	e_{(q,m,0)} \leq C \sum_{s=0}^m  e_{(0, 0, s)},
$$
where $C = C(\om, n,m)$ is a uniform constant. 
\end{prop}

\begin{proof} We start with Lemma~\ref{lem:case1} which gives 
\[\label{eq:e1}	e_{(q,m,0)} \leq C \left[e_{(q-1, m-1, 1)} +  e_{(q-1, m-1,0)}\right].
\]
Then, the first term in the bracket is estimated via Corollary~\ref{cor:case1}. Applying this corollary $(m-1)$-times and using decreasing property of $e_{(p,k,s)}$ in the first parameter, we get
\[\label{eq:e2}
	e_{(q-1,m-1,1)} \leq C e_{(q-m, 0, m)} + C\sum_{i=0}^{m-1} e_{(q-m-2, i, m-1-i)}.
\]
The second term in the bracket  in \eqref{eq:e1} satisfies  $e_{(q-1,m-1,0)} \leq e_{(q-m-2,m-1,0)}$. Next, we use Lemma~\ref{lem:case2} for each term $e_{(q', \ell, m-1-\ell)}$ with $q'=q-m-2$ in the sum above. Namely, applying the lemma $\ell$-times and using the decreasing property again, we get
$$
	e_{(q', \ell, m-1-\ell)} \leq C e_{(q'-\ell, 0, m-1)} + C \sum_{i=0}^{m-2} e_{(q'-\ell-1,i,m-2-i)}.
$$
Note that the smallest value of the first parameter in the last sum is $q'-m$ for $\ell =m-1$.  Hence,  
\[\label{eq:e3}
	\sum_{i=0}^{m-1} e_{(q',i,m-1-i)} \leq C e_{(q'-\ell, 0, m-1)} + C \sum_{i=0}^{m-2} e_{(q'-m, i, m-2-i)}.
\]
It remains to apply Lemma~\ref{lem:case3} for each term $e_{(q'',\ell,m-2-\ell)}$ in the sum on the right hand side, where $q''= q'-m=q-2m-2$. Again, we have
$$\begin{aligned}
	e_{(q'', \ell, m-2-\ell)} 
&\leq 	C e_{(q''-1, \ell-1, m-\ell -1)} + C e_{(q'',\ell-1, m-\ell-2)}\\ 
&\leq \cdots \\
&\leq C e_{(q''-\ell, 0, m-2)} + \sum_{i=0}^{\ell-1} e_{(q''-i, \ell-1-i, m-\ell-2+i)}.
\end{aligned}$$
Therefore, an easy induction argument gives us
\[\label{eq:e4}
	e_{(q'',\ell, m-2-\ell)} \leq  C\sum_{s=0}^{m-2} e_{(q''-\ell, 0, s)}. 
\]
Combining \eqref{eq:e1}, \eqref{eq:e2}, \eqref{eq:e3} and \eqref{eq:e4} we arrive at
$$
	e_{(q,m,0)} \leq C \sum_{s=0}^{m} e_{(q''-m+2,0,s)} = C \sum_{s=0}^m e_{(0,0,s)}
$$
as $q''-m+2 = q-3m = 0.$
\end{proof}

So far all considered functions were smooth, however, by \cite[Proposition~2.11]{KN3} we know that the integrands on  both sides of the above statements (Lemmas~\ref{lem:CLN-grad}-\ref{lem:case3}, Corollary~\ref{cor:case1} and Proposition~\ref{prop:cap-integral}) are well-defined for continuous $m-\om$-sh functions. Let us record the following observation.

\begin{remark}\label{rmk:continuous-case}  Let $\Om$ be a strictly $m$-pseudoconvex domain. The statements above are still valid for continuous $m-\om$-sh functions $v, w$ and $-1\leq \rho \leq 0$ satisfying $-1\leq v\leq w \leq 0$ and $v=w$ in a neighborhood of $\d\Om$.

 In fact, there exist decreasing sequences of $m-\om$-sh functions $v_j, w_j$ and $\rho_j$ belonging to $C^\infty(\ov\Om)$ such that $v_j \downarrow v$, $w_j \downarrow w$ and $\rho_j\downarrow \rho$ (uniformly) in $\ov\Om$, and  moreover, 
 $$
 	-1\leq v_j \leq w_j \leq 0, \quad  -1\leq \rho_j \leq 0.
 $$
 For plurisubharmonic functions the usual convolution with standard kernels  produces the approximating sequence and hence, the  property:  $v_j = w_j$ near the boundary is preserved. Then, the integration by parts is not  affected and passing to the limit as $j\to +\infty$ gives the desired inequalities. However, in this new setting we used a different way to obtain the approximating sequence. The property that  $v_j = w_j$  near the boundary of  $\Om$ needs to be verified, which is possible via the stability estimates for complex Hessian equations. However, we can get around this by showing  the uniform convergence to zero of the sequence $h_j=w_j-v_j$ near the boundary. 

Let $\Om' \subset \subset \Om$ be a smooth domain such that $v=w$ outside $\Om'$. Let $T_j =  (dd^c \rho_j)^{k} \wed (dd^c v_j)^s$, where $k+s\leq m$. Then, it follows from the weak convergence \cite[Proposition~2.11]{KN3} and the CLN inequality \cite[Proposition~2.9]{KN3} that for $p\geq 1$, 
$$
	\lim_{j\to \infty} \int_{\Om'} h_j^{p} T_j \wed \om^{n-k-s} = \int_\Om h^{p} T \wed \om^{k-s}.
$$
Therefore, we reduce the required inequality to the case of  smooth functions $v_j, w_j$ and $\rho_j$. 
However the integration by parts in \eqref{eq:int-by-parts} will contain the extra boundary terms:
$$
	\int_{\Om'} h_j^{p} dd^c\rho_j \wed T_j \wed \om^{n-k-s} 
	= \int_{\Om'} \rho_j dd^c(  h_j^p \om^{n-k-s}) \wed T_j + E_1+ E_2,
$$
where 
$$\begin{aligned}
	E_1 &= \int_{\d \Om'} h_j^p d^c\rho_j \wed \om^{n-k-s}\wed T_j; \\
 	E_2 &= - \int_{\d\Om'} \rho_j d^c (h^p_j \om^{n-k-s})\wed T_j  \\
&= - \int_{\d\Om'} \rho_j h_j^{p-1} (pd^ch_j \wed \om^{n-k-s}+ h_j d^c\om^{n-k-s})\wed T_j.
\end{aligned}$$
By the CLN inequality and $h_j \to 0$ uniformly  on a neighborhood of $\d\Om'$ as $j\to \infty$, the two boundary terms go to zero when we pass to the limit. 
\end{remark}

\begin{remark} \label{rmk:bounded-case}We will see later that the above statements also hold for  bounded $m-\om$-sh functions once we define the wedge product for currents related to such functions and prove the weak convergence under decreasing sequences.
\end{remark}

\section{Wedge product for bounded functions}

In this section we prove the existence of the wedge product of currents where $dd^c$ operator is applied to bounded $m-\om$-sh functions. Then, we introduce a weighted Sobolev space associated to these resulted positive currents. This is a crucial technical tool that allows us to establish the integration by part formulae for (non-closed)  currents of order zero. 

\subsection{The wedge product for bounded $m-\om$-sh functions}

Let $\Om \subset \bC^n$ be a bounded open set. 
We are going to show that the wedge product of bounded $m-\om$-sh functions can be defined inductively by
$$
	dd^c v \wed dd^c u = dd^c(v dd^cu)
$$
and similarly for more terms. Let us start with the following observation.

\begin{lem} \label{lem:zero-order-1}
Let $u$ be a bounded $m-\om$-sh function. Then, $dd^c u$ is a real $(1,1)$-current of order zero whose coefficients are singed Radon measures. 
\end{lem} 

\begin{proof} 
Without loss of generality we may assume $-1\leq u\leq 0$. We will show that  there is a signed measure $\mu$ such that 
$$
	dd^c u (\phi) = \int \phi d\mu
$$
for every (smooth) test $(n-1,n-1)$-form.  By a theorem in distribution theory in Federer's book   \cite[Section  4.15]{Fe69} it is enough to show that for all $K \subset \Om$ compact there exists $C = C(K,\Om)$ such that
$$
	|dd^c u (\phi)| \leq C \|\phi\|_{K}
$$
for any smooth test form $\phi$ with $\supp\phi \subset K$. Here we denote 
$$
	\|\phi\|_K = \sum_{i, j}\sup_K |\phi_{i\bar j}|,
$$
where $\phi_{i\bar j}$'s are coefficients of the form $\phi$.
By the definition of $dd^c$-operator for currents
$$
	dd^c u (\phi) = u (dd^c \phi) = \int u \wed dd^c \phi.
$$
Thus, for a smooth  sequence of $m-\om$-sh functions $\{u_s\}$ that decrease to $u$ we have
$$
	dd^c u (\phi) = \lim_{s\to\infty} (dd^c u_s) (\phi) = \lim_{s\to \infty} \int dd^c u_s \wed \phi.
$$
Since $dd^cu_s \in \Ga_m(\om)$, it follows from \cite[Corollary~2.4]{KN3} that
\[\label{eq:variation-dominated}
	|dd^c u_s \wed \phi| \leq C \|\phi\|_K \;  dd^c u_s \wed \om^{n-1}
\]
for some universal constant depending on $n,m$. It follows that
$$
	\left| \int dd^c u_s \wed \phi \right| \leq C \|\phi\|_K \int_K dd^c u_s \wed \om^{n-1}. 
$$
Thanks to the CLN inequality \cite[Proposition~2.9]{KN3} we know that the last integral on the right hand side is bounded by a constant $C(K,\Om) \|u_s\|_{L^\infty}$. This finishes the proof. 

 Furthermore, for any test function $\chi \geq 0$, 
$$
	0 \leq \lim_{s\to \infty} \int \chi \om^{n-1} \wed dd^c u_s = dd^c u(\chi \om^{n-1}). 
$$
In other words,  $dd^c u\wed \om^{n-1}$ is a (positive) Radon measure. Moreover, it follows from \eqref{eq:variation-dominated} that the total variation of coefficients of $dd^cu$ is dominated by  $dd^c u\wed \om^{n-1}$. Therefore, the coefficients of $T = dd^c u$ are signed Radon measures. 
\end{proof}

It follows from Lemma~\ref{lem:zero-order-1} that we can multiply $dd^cu$ by a bounded function $v$. In fact, we can continue to define inductively the wedge product in this way and obtain a real current of order zero for up to $(m-1)$ bounded functions.

\begin{lem}\label{lem:w-p-functions} Let $u_1,...,u_{p}$ with $1\leq p\leq m-1$ be bounded $m-\om$-sh  functions. Assume $T_0 =1$ and $T_{p-1} = dd^c u_{p-1} \wed \cdots \wed dd^c u_1$.   Then, the current 
$$
	dd^c u_p \wed \cdots \wed dd^c u_1 :=  dd^c [u_p T_{p-1}]
$$
is a well-defined real $(p,p)$-current of order zero whose coefficients are signed Radon measures. Moreover, 
$$
	\Lc_p(u_p,...,u_1):=	dd^c u_p \wed \cdots \wed dd^c u_1\wed \om^{n-m}
$$
is a positive $(n-m+p,n-m+p)$-current.
\end{lem}

\begin{proof} We argue by induction in $p$. If $p=1$, then this is Lemma~\ref{lem:zero-order-1}. Assume that it holds for $1\leq p \leq m-2$, we need to show that the lemma holds for $p+1$. For the simplicity of notation we present the argument for $p+1=2$ and note that the proof in the general case is  completely the same. Let us write $u_1 = u$ and $u_2 = v$ and assume that $-1\leq u, v\leq 0$. Since $dd^cu$ has signed measure coefficients, the currents $$
	S = v dd^c u, \quad \text{and } dd^c v\wed dd^c u := dd^c S = dd^c (v dd^c u).
$$
are well-defined. Our next goal is to show that the real (closed) current $dd^c S$ of bidegree $(2,2)$ has order zero, thus its coefficients are signed Radon measures. 

As above we will show that for each compact set $K \subset \Om$, there is a uniform constant $C$ such that
\[\label{eq:zero-order-2a}
	|dd^c S (\phi)| \leq C \|\phi\|_K 
\]
for every smooth test form with $\supp \phi \subset K$, where $\|\phi\|_K$ denotes the $C^0$-norm of the form $\phi$. To this end,  let $\{v_\ell\}_{\ell \geq 1}$ and $\{u_s\}_{s\geq 1}$ be sequences of smooth $m-\om$-sh functions such that 
$$
	v_\ell \downarrow v, \quad \text{and}\quad u_s \downarrow u.
$$
We may also assume that $-1\leq u_s, v_\ell \leq 0$. Fix  $\phi$ a test form. By definition we have 
\[\label{eq:zero-order-2b}
	dd^c S (\phi) = S (dd^c \phi) = \int v dd^c u\wed dd^c\phi. 
\]
Since $dd^c u \wed dd^c \phi$ has signed Radon measure coefficients, it follows from the  dominated convergence theorem that 
$$
	\int v dd^c u\wed dd^c \phi = \lim_{\ell \to \infty} \int v_\ell dd^c u \wed dd^c \phi = \lim_{\ell \to \infty} \lim_{s\to \infty} \int v_\ell dd^c u_s \wed dd^c \phi.
$$
 By the integration by parts formula, 
$$
	\int v_\ell dd^c u_s \wed dd^c \phi = 	\int dd^c v_\ell \wed dd^c u_s \wed \phi.
$$
Notice that $m-1\geq 2$. Applying \cite[Corollary~2.4]{KN3} for $\ga_1 = dd^c u_s, \ga_2 = dd^c v_\ell$ belonging to $\Ga_m(\om)$, we get
\[ \label{eq:zero-order-2c}
	\left|  dd^c v_\ell \wed dd^c u_s \wed \phi \right| \leq  C \|\phi\|_K  [dd^c (u_s +v_\ell)]^2 \wed \om^{n-2}. 
\]
Integrating both sides and then applying the CLN inequality for the right hand side we get
$$
	\left| \int dd^c v_\ell \wed dd^c u_s \wed \phi \right| \leq  C(K,\Om) \|\phi\|_K. 
$$
Letting $s\to\infty$ and then $\ell\to\infty$, in this order,
 we conclude from \eqref{eq:zero-order-2a} the desired estimate \eqref{eq:zero-order-2b}. Hence,  $dd^c S$ is a $(2,2)$ current of order zero. 
 
 Furthermore, for a test function $\chi \geq 0$,
$$
	dd^c S \wed \om^{n-2} (\chi ) = \lim_{\ell \to \infty} \lim_{s\to \infty}  \int \chi dd^c v_\ell \wed dd^c u_s \wed \om^{n-2} \geq 0.
$$
Thus, $dd^c S \wed \om^{n-2}$ is a positive Radon measure. It follows from \eqref{eq:zero-order-2c} (by passing to the limit) that the coefficients of $dd^c S$ are signed Radon measures.

Finally, for smooth $m-\om$-sh functions $u_1^{j_1},...,u_p^{j_p}$, then the form
$$
	\Lc_p(u_p^{j_p},...,u_1^{j_1}) = dd^c u_p^{j_p}\wed \cdots \wed dd^c u_1^{j_1} \wed \om^{n-m}
$$
is positive by G\aa rding's inequality  \cite{Ga59}. Letting $j_1\to \infty$, ...$j_p\to\infty$, in this order, for decreasing sequences approximating  $u_p,...,u_1$ we get that $\Lc_p(u_p,...,u_1)$ is a positive current of bidegree $(n-m+p, n-m+p)$.
\end{proof}

Thus, we defined inductively for bounded $m-\om$-sh functions, $u_1, ..., u_{m-1}$,  a real closed current 
\[\label{eq:(m-1)-w-prod-a} S: = dd^c u_{m-1} \wed \cdots \wed dd^c u_1\] 
of order zero whose coefficients are signed Radon measures. Furthermore, 
\[\label{eq:(m-1)-w-prod-b}
	S \wed \om^{n-m+1}
\]
is a positive Radon measure.  In the last step we multiply $S$ by a bounded function $v=u_m$.

\begin{thm}\label{thm:w-m-functions} Let $S$ be as in \eqref{eq:(m-1)-w-prod-a} and  let $v$ be a bounded $m-\om$-sh function. Then,
$$
	T := dd^c (v S)
$$
is a real closed $(m,m)$-current and  $T \wed \om^{n-m}$ is a positive Radon measure. 
\end{thm}

\begin{proof}
Let $\chi \geq 0$ be a test function. By definition
$$
	T \wed \om^{n-m}(\chi)  = T (\chi \om^{n-m})= \int v S \wed dd^c (\chi \om^{n-m}). 
$$
Since $S$ is of order zero, it follows from the dominated convergence theorem that for a sequence of smooth $m-\om$-sh functions   $v_\ell \downarrow v$,
$$
	 \int v S \wed dd^c (\chi \om^{n-m})  = \lim_{\ell \to \infty} \int v_\ell S \wed dd^c (\chi \om^{n-m}).
$$
Using the inductive definition of $S$ and the integration by parts, we get
$$\begin{aligned}
&	\int v_\ell S \wed dd^c (\chi \om^{n-m}) \\
&= \lim_{j_{m-1}\to \infty} \cdots \lim_{j_1\to \infty} \int dd^c v_\ell \wed dd^c u^{j_{m-1}} \wed \cdots \wed u^{j_1} \wed dd^c (\chi \om^{n-m}) \\
&= \lim_{j_{m-1}\to \infty} \cdots \lim_{j_1\to \infty}  \int dd^cv_\ell \wed dd^c u^{j_{m-1}} \wed \cdots \wed dd^c u^{j_1} \wed  (\chi \om^{n-m}).
\end{aligned}$$
Since $v_\ell, u^{j_1},...,u^{j_{m-1}}$ are smooth $m-\om$-sh functions, the last integrand is positive and the proof follows.
\end{proof}

Consequently the following definition is justified.
\begin{defn}\label{defn:w-prod}  Let $u_1,...,u_{p}$ with $1\leq p\leq m$ be bounded $m-\om$-sh functions in $\Om$. The wedge product is given inductively by
$$	dd^c u_p \wed \cdots \wed dd^c u_1 := dd^c [u_p T_{p-1}],
$$
where $T_0=1$ and $T_{p-1} = dd^c u_{p-1} \wed\cdots \wed dd^c u_1$.
\end{defn}

We are most  interested in the following positive Radon measures.  

\begin{defn} \label{defn:hes-operator}Let $u$ be a bounded $m-\om$-sh function. 
 Then, the Hessian operator $H_m(u)$ is defined by
$$
	H_m(u):=(dd^cu)^m \wed \om^{n-m} = \Lc_m(u,..., u).
$$
Moreover, for $1\leq s \leq m$, we also write $H_s(u) = (dd^c u)^s \wed \om^{n-s}$.
\end{defn}

For convenience, we summarize the results in Lemma~\ref{lem:w-p-functions} and Theorem~\ref{thm:w-m-functions} as follows.  Let $\{u_s^{j}\}_{j\geq 1}$ be decreasing sequences such  that $u_s^j \downarrow u_s$ for each $1\leq s\leq m$. Define
\[\label{eq:inductive-limit-notation}
	T^{j_p \cdots j_1} = dd^c u_p^{j_p} \wed \cdots \wed dd^c u_1^{j_1}.
\]
Then, we have $T^{j_p\cdots j_2} \wed dd^c u_1 = \lim_{j_1\to \infty}T^{j_p \cdots j_1}$, and consequently
\[ \label{eq:inductive-limit-a}
	dd^c u_p \wed \cdots \wed dd^c u_1 = \lim_{j_p\to \infty} \cdots \lim_{j_1\to \infty} T^{j_p \cdots j_1}.
\]
For $1\leq p \leq m-1$, each limit in \eqref{eq:inductive-limit-a} is in the  sense of currents of order zero. For $p=m$, we have
\[\label{eq:inductive-limit-b}dd^c u_p \wed \cdots \wed dd^c u_1 \wed\om^{n-m}= \lim_{j_p\to \infty} \cdots \lim_{j_1\to \infty} T^{j_p \cdots j_1} \wed \om^{n-m},\]
where each limit is also in the  sense of currents of order zero.

\begin{remark}  At this point the definition of the wedge product in Definition~\ref{defn:w-prod}  is not symmetric with respect to $u_1,...,u_p$. Furthermore, if the metric $\om$ is K\"ahler or just the standard K\"ahler form $\be = dd^c|z|^2$ in $\bC^n$, then by \eqref{eq:inductive-limit-b} it follows that Definition~\ref{defn:w-prod} coincides with the ones given before in \cite{Bl05},  \cite{DK14}.
 \end{remark}

By taking the limit inductively as in \eqref{eq:inductive-limit-a} and \eqref{eq:inductive-limit-b} we can state now the CLN inequality for bounded functions.

\begin{prop}[CLN inequality] \label{prop:CLN-bdd} Let $K\subset\subset U \subset\subset \Om$, where $K$ is compact and $U$ is open. Let $u, u_1,...,u_p$, be bounded $m-\om$-sh functions in $\Om$, where $1\leq p\leq m$. Then, there exists a constant $C$ depending on $K, U, \om$ such that
\begin{itemize}
\item 
[(a)] $\int_K (dd^c u)^p \wed \om^{n-p} \leq C (1+ \|u\|_{L^\infty(U)})^p$;
\item
[(b)]
$
	\int_K dd^c u_1 \wed \cdots \wed dd^c u_p \wed \om^{n-p} \leq C \left(1+\sum_{s=1}^p \| u_s\|_{L^\infty(U)}\right)^{p}. 
$
\end{itemize}
\end{prop}

\begin{proof} (a) Without loss of generality we may assume that $\|u\|_{L^\infty(\Om)} \leq 1$. We can cover $K$ by finitely many small balls, hence by localization principle we can assume that $K$ and $U$ are concentric balls. Let $\{u^j\}_{j\geq 1}$ be sequences of smooth $m-\om$-sh functions such that $u^{j} \downarrow u$ as $j \to \infty$. Let $0\leq \chi \leq 1$ be a cut-off function such that $\chi \equiv 1$ on $K$ and $\supp\chi \subset U$. By the CLN inequality for smooth functions \cite[Proposition~2.9]{KN3}
$$
	\int \chi \Lc_p(u^{j_p},...,u^{j_1}) \wed \om^{m-p} \leq  C \left(1+ \| \sum_{s=1}^pu^{j_s}\| _{L^\infty(U)}^p \right). 
$$
By the Hartogs lemma for $\om$-sh functions \cite[Lemma~9.14]{GN18} it follows that 
$$\lim_{j\to \infty}\|1+ u^j\|_{L^\infty(U)} = \lim_{j\to \infty} \sup_{\ov U} (1+u^j)= 1+ \sup_{\ov U} u \leq 1 + \|u\|_{L^\infty(U)}.$$ 
Let $j_1\to \infty$, ....,$j_p\to \infty$, in this order we get 
$$\begin{aligned}
\int_K \Lc_p(u) \wed \om^{m-p} 
&\leq 	\lim_{j_p \to \infty}\cdots\lim_{j_1 \to \infty}\int \chi \Lc_p(u^{j_p},...,u^{j_1}) \wed \om^{m-p} \\
&\leq 	C \lim_{j_p\to \infty} \cdots \lim_{j_1\to\infty} \left(1+ \|\sum_{s=1}^pu^{j_s}\|_{L^\infty}\right) ^p\\
&=  		C (1+ \| u\|_{L^\infty(U)})^p. 
\end{aligned}$$

(b) Observe that for $v:= u_1 +\cdots + u_p$ we have $\Lc_p(v) \geq \Lc(u_1,...,u_p)$ as positive currents. So, (b) is an immediate consequence of (a).
\end{proof}

We also need a simplified version of Lemma~2.3 for particular positive currents. Let $-1\leq u_1,...,u_{m-2} \leq 0$ be  bounded $m-\om$-sh functions in $\Om$. Denote
by $T$ the current
\[\label{eq:cs-weak}
T = dd^c u_1 \wed \cdots \wed dd^cu_{m-2} \wed \om^{n-m}.
\]

\begin{cor} \label{cor:cs-w-bdd} Let $\phi$ be a Borel function such that $\supp\phi = K\subset\subset \Om$. Assume also  $|\phi| \leq 1$. Let $0\leq h \leq 1$ be a smooth $m-\om$-sh function in $\Om$. There exists a constant $C = C(\om,K)$ such that
$$
	\left |\int \phi\; dh \wed d^c \om \wed T \right|^2 \leq C  \int |\phi|\;T \wed \om^2.
$$
\end{cor}

\begin{proof} By Lemma~2.3, it is enough to show that 
$\int |\phi| dh \wed d^c h \wed T \wed \om \leq C.$ In fact, since $h^2$ and $h$ are $m-\om$-sh, we have
$$
	2dh \wed d^c h \wed T\wed\om = [ dd^c h^2 - 2h dd^c h] \wed T\wed\om  \leq dd^c h^2 \wed T\wed\om.
$$
Multiplying both sides by $|\phi| $ and integrating them, we get
$$
	\int |\phi| dh \wed d^c h \wed T \wed \om \leq  \int |\phi| dd^c h^2 \wed T\wed\om.
$$
Using the assumption on $\phi$ the right hand side is bound by
$$
	\int_K dd^c h^2 \wed T \wed \om.
$$
The desired estimate follows from the CLN inequality for bounded functions (Proposition~\ref{prop:CLN-bdd}).
\end{proof}

\subsection{Weighted Sobolev space}
\label{ss:sobolev}

Let $K\subset \Om$ be a compact subset and let $\psi$ be a strictly plurisubharmonic defining function for $\Om$. Let us denote by $P^* = P^*(\Om, K, \psi)$  the set of bounded $m-\om$-sh functions that
$$
	u (z) = \psi, \quad z \in \Om\setminus K.
$$
Let $u_1,..,u_{p} \in P^*$ with $1\leq p\leq m-1$ and denote \[\label{eq:tau}\tau = dd^c u_1 \wed \cdots \wed dd^c u_{p},\] and conventionally $\tau=1$ for $p=0$. By definition this is a closed $(p,p)$-current of order zero.
We first make an additional assumption on the current $\tau$. By considering $u_i = u_i + \de |z|^2$ for some $\de>0$ small, we may assume that there is $c_0>0$ such that
\[\label{eq:add}
	\tau \wed \om^{n-p-1} \geq c_0 \om^{n-1}.
\]
Hence, the trace measure  (with respect to $\om$) of the positive current $\tau\wed \om^{n-p}$ satisfies
\[\label{eq:trace-m}
	 \mu_p:=\tau\wed \om^{n-p} \geq c_0 \om^n.
\]

Let us introduce a norm  associated to $\tau$ on the space smooth functions $C^\infty(\Om, \bR)$.
If $v,w$ are smooth functions in $\Om$, then we can define
$$
	\lc v, w\rc_\tau = \int dv \wed d^c w \wed 
	\tau\wed \om^{n-p-1},
$$
and 
$$
	\| v\|_\tau^2 = \lc v, v\rc_\tau =  \int dv \wed d^c v \wed \tau\wed \om^{n-p} .
$$
Let us consider the following inner product 
$$
	\lc v, w\rc = \int v w \; \be^n + \lc v, w \rc_\tau,
$$
where $\be = dd^c|z|^2$, and the associated norm
$$
	\|v\|^2 = \lc v,v\rc =  \|v\|_{L^2}^2 +  \|v\|_{\tau}^2.
$$
The assumption makes this inner product  positive definite. In fact, it is clear that  $\lc \bullet, \bullet \rc$ satisfies the properties of an inner product, namely
\begin{itemize}
\item[(i)] $\lc v, v \rc \geq 0$ and $\lc v, v\rc = 0$ if and only if $v=0$.
\item[(ii)] $\lc av_1+ bv_2,w\rc = a \lc v_1, w\rc + b \lc v_2, w\rc$.
\item[(iii)] $\lc v, w\rc = \lc w, v\rc$.
\end{itemize}
This implies also that  it satisfies the triangle inequality
$$
	\left|\lc v, w \rc \right| \leq \sqrt{\lc v, v\rc \lc w, w\rc}.
$$
Let us denote by $W^{1,2}(\Om, \mu_p)$  the real Hilbert space completion of $C^\infty(\Om,\bR)$ with respect to this norm. Let us state a property of this space 
immediately following from the definition.

\begin{lem} \label{lem:uni-conv}
$W^{1,2}(\Om, \mu_p)$ is uniformly convex. Consequently, it is reflexive.
\end{lem}

 Our goal is to show that for each $0\leq p\leq m-1$,  $$P^* \subset W^{1,2} (\Om, \mu_p).$$
 If we take $u_{p} = \psi$ the strictly psh defining function for $\Om$, then clearly 
 $$
 	W^{1,2}(\Om, \mu_{p}) \subset W^{1,2}(\Om,\mu_{p-1}).
 $$
 Therefore, we only need to prove two special cases for $p=0$ and $p=m-1$.
 Let us start with this simpler case $p=0$, i.e., $\tau = 1$, $\tau\wed \om^{n-1} =\om^{n-1}$ and $\mu_0 =\om^n$ which  is equivalent to  the Lebesgue measure $dV = \be^n$.  This is  a classical result if $\om =\be$ is the standard K\"aher metric (see e.g. \cite[Proposition~3.4.19]{Hor94}).

\begin{lem}\label{lem:sobolev-space}
All functions in $P^*$ belong to the Sobolev space $W^{1,2}(\Om)$. 
\end{lem}

\begin{proof} Let $u\in P^*$. By the approximation theorem there is $\{u_j\}_{j\geq 1} \subset P^*$ a sequence of smooth functions such that $u_j \downarrow u$ pointwise. We may assume that $0\leq u, u_j \leq 1$. Observe that  the sequence $\{u_j\}_{j\geq 1}$ is uniformly bounded in $W^{1,2}(\Om)$. Indeed, by the $\om$-subharmonicity
$$\begin{aligned}
	\int du_j \wed d^c u_j \wed \om^{n-1} 
&= \int \left(\frac{1}{2} dd^c u_j^2 - 2 u_j dd^c u_j\right) \wed \om^{n-1} \\
&\leq \frac{1}{2}\int dd^c u_j^2 \wed \om^{n-1}.
\end{aligned}$$
The last integral is bounded by a constant $C= C(K, \psi, \Om)$ by an application of the CLN inequality. Therefore, the sequence of (vector-valued) functions $\na u_j$ is uniformly bounded in $L^2(\Om)$. Thus, by the weak compactness theorem, $\na u_j$ converges weakly to a vector-valued function ${\bf v} \in L^2(\Om)$. 

To complete the proof we are going to show that ${\bf v} = \na u$. First, note that  $u_j \to u$  in $L^2(\Om)$ by the dominated convergence theorem. So, $(u_j, \na u_j) \to (u, {\bf v})$ weakly in the Hilbert space $L^2(\Om) \times L^2(\Om, \bR^n)$. The Mazur lemma \cite[Chapter V, Theorem~2]{Yo80} implies that there is a sequence of convex combinations
$$
	\wt u_\ell = \sum_{j=1}^\ell \la_{\ell,j} (u_{j}, \na u_j), \quad \la_{\ell,j} \geq 0 \text{ and } \sum_{j=1}^\ell \la_{\ell,j} =1,
$$
which converges in norm  to $(u,{\bf v})$ in $L^2(\Om) \times L^2(\Om, \bR^n)$. It follows that 
$$
	 w_\ell = \sum_{j=1}^\ell \la_{\ell,j} u_j
$$
is a Cauchy sequence in $W^{1,2}(\Om)$. Hence, there exist a limit function $$\wt u =\lim_{\ell\to \infty} w_\ell \in W^{1,2}(\Om).$$ 
It follows that $u = \wt u$ and ${\bf v} = \na \wt u = \na u$.
\end{proof}

We proceed to prove the most general case $ 1\leq p= m-1$ when
$$\mu= \mu_p:=\tau \wed \om^{n-m+1}$$ is a positive Radon measure. Equivalently, 
we wish to show that if $v \in P^*$, then the form $dv$ has coefficients in $L^2 (\Om, d\mu)$. Note that for smooth function $v$ we have
$$\begin{aligned}
&	\int |\na v|^2 dV \leq \frac{1}{c_0} \int dv \wed d^c v \wed \tau \wed \om^{n-m}. 
\end{aligned}$$
It implies that 
$W^{1,2}(\Om, \mu) \subset W^{1,2}(\Om)$. In particular, the gradient $\na u$ is well-defined in the space $W^{1,2}(\Om, \mu)$.

\begin{prop}\label{prop:weak-conv} Let $v\in P^*$ and $\{v_\ell\}_{\ell\geq 0} \subset P^*$ be a sequence of smooth $m-\om$-sh functions such  that $v_\ell \downarrow v$ point-wise in $\Om$. Then  $v_\ell$ converges weakly  to $v$ in $W^{1,2}(\Om, \mu)$,  i.e., for every $w\in W^{1,2}(\Om, \mu)$,
$$
	\lim_{\ell\to \infty}\lc v_\ell, w \rc = \lc v, w\rc. 
$$
In particular,  $P^*\subset W^{1,2}(\Om,\mu)$.
\end{prop}

\begin{proof} We may assume that $0\leq v\leq v_\ell \leq 1$. Then, there exists a uniform constant $C = C(\om,n,m,\psi,\Om)$ such that 
$$
	\|\na v_\ell\|^2_\tau= \int  dv_\ell \wed d^c v_\ell \wed \tau\wed \om^{n-m} \leq C \|v_{\ell}\|_{L^\infty}^2.
$$
In fact, since 
$
	dv_\ell \wed d^c v_\el = \frac{1}{2} dd^c v_\ell^2 - v_\ell dd^c v_\ell,
$
and $dd^c v_\ell \wed \tau \wed \om^{n-m}$ is a positive measure and $v_\ell\geq 0$, it follows that 
$$
	\int  dv_\ell \wed d^c v_\ell \wed \tau\wed \om^{n-m} \leq \int dd^c v_\ell^2 \wed \tau\wed \om^{n-m} .
$$
Since $v_\ell^2$ is also $m-\om$-sh, the desired inequality follows from the CLN inequality. 

In the same way as in the previous proof, via Mazur's lemma, one proves that $v\in W^{1,2}(\Om,\mu)$.
\end{proof}

\begin{remark} The proofs given in Lemma~\ref{lem:sobolev-space} and Proposition~\ref{prop:weak-conv} are inspired  by the one in the book by Heinonen,  Kilpel\"{a}inen and Martio \cite[Theorem~1.30]{HKM06}.
\end{remark}

\subsection{Integration by parts for bounded functions}
In this section we  show  "integration by parts inequalities" for bounded $m-\om$-sh functions which are smooth near the boundary. In particular, it holds for functions in the class $P^*$.
First we deal  with the case of  the wedge product of $(m-2)$ bounded functions and consider
\[\label{eq:eta}
\eta = dd^c u_1\wed \cdots \wed dd^c u_{m-2}
\]
where $u_1,..., u_{m-2} \in P^*$ as  in Section~\ref{ss:sobolev}. 
Recall that we also impose the assumption \eqref{eq:add} for $\eta$, namely
\[\label{eq:add-a}
	\eta \wed \om^{n-m+1} \geq c_0 \om^{n-1},
\]
and denote the corresponding trace measure by $\mu$:
\[\label{eq:trace-m-1}
	\mu = \eta \wed \om^{n-m+2}.
\]

The following integration by parts formula  for currents of order zero is used frequently below. Let $\phi$ and $\rho$ be smooth functions such that $\phi = 0$ near $\d\Om$. Then,
\[\begin{aligned}
\label{eq:IBP-normal}
	\int \phi \;dd^c \rho \wed \eta \wed \om^{n-m+1} 
&	= -\int d\phi \wed d^c \rho \wed \eta \wed \om^{n-m+1} \\
&\quad + \int\phi d^c \rho \wed d\om \wed\eta \wed\om^{n-m}.
\end{aligned}\]
If $v, w \in P^*\subset W^{1,2}(\Om, \mu)$ are smooth, then 
$$
	\lc v, w\rc_\eta = \int dv \wed d^c w \wed \eta \wed \om^{n-m+1}.
$$
By abusing the notation for general $v, w \in P^*$  we will write the integral on the right understanding that it is the inner product $\lc v, w\rc_\eta$. The first version of the "integration by part inequality" in the class $P^*$ reads as follows.

\begin{lem}\label{lem:w-IBP-a} Let $\{v_\ell\}_{\ell\geq 1}$ be a sequence of smooth functions from $P^*$ and let $w\in P^*$ be smooth. Assume $v_\ell\leq w$ and $v_\ell \downarrow v \in P^*$. Then,
$$\begin{aligned}
	-\int d(w-v) \wed d^c v \wed \eta\wed \om^{n-m+1} 
&\leq \int(w-v) dd^c v \wed \eta\wed\om^{n-m+1} \\
&+ C \left( \int(w-v) \eta\wed \om^{n-m+2}\right)^\frac{1}{2}.
\end{aligned}$$
The constant $C$ depends on $\om, K, \psi$ and the uniform norm of functions, but it is independent of  $\de$ in \eqref{eq:add}.
\end{lem}

\begin{proof} Without loss of generality we may assume that $0\leq v, v_\ell, w \leq 1$ and $0\leq u_i \leq 1$. By the weak convergence   in the space $W^{1,2}(\Om, \mu)$ (Propostion~\ref{prop:weak-conv}) and the dominated convergence theorem  it is enough to prove that
$$\begin{aligned}-\int d(w-v_\ell) \wed d^c v \wed \eta\wed \om^{n-m+1} 
&\leq \int(w-v_\ell) dd^c v \wed \eta\wed\om^{n-m+1} \\
&+ C \left( \int(w-v_\ell) \eta\wed \om^{n-m+2}\right)^\frac{1}{2}.
\end{aligned}$$
To this end we use once more the weak convergence property when $\ell$ is fixed.  Note that $w-v_\ell$ is a continuous function whose support is compact in $\Om$. It follows from the weak convergence of currents (of order zero) that
$$
	 \int (w-v_\ell) dd^c v \wed \eta\wed\om^{n-m+1} = \lim_{j\to \infty}  \int (w-v_\ell) dd^c v_j \wed \eta\wed\om^{n-m+1}
$$ 
Now for each $\ell$ and $j$, the integration by parts formula  \eqref{eq:IBP-normal} gives 
$$\begin{aligned}
-\int d(w - v_\ell) \wed d^c v_j \wed \eta \wed \om^{n-m+1} 
&=  \int (w-v_\ell) dd^c v_j \wed \eta\wed\om^{n-m+1} \\
&\quad + \int (w-v_\ell) d^c v_j \wed d\om \wed \eta\wed\om^{n-m}.
\end{aligned}$$
By Corollary~\ref{cor:cs-w-bdd} applied for $\phi = w-v_\ell$ and CLN inequality (Proposition~\ref{prop:CLN-bdd}) for bounded functions,  the second integral on the right hand side is bounded by
$$	
	C \left(\int (w-v_\ell)  \eta \wed \om^{n-m+2}\right)^\frac{1}{2}.
$$
The desired inequality follows by letting $j\to\infty$.
\end{proof}

\begin{lem}\label{lem:w-IBP-b} Let $v,w$ and $\{v_\ell\}$ be as in Lemma~\ref{lem:w-IBP-a}. Let $\rho$ be a bounded $m-\om$-sh in $\Om$ such that $-1 \leq \rho \leq 0$. Then,
$$\begin{aligned}
	\int_\Om (w - v) dd^c \rho \wed \eta\wed\om^{n-m+1}  
&\leq  - \int_{\Om} d(w-v) \wed d^c \rho \wed \eta\wed\om^{n-m+1} \\ 
&\quad + C \left(\int_\Om (w-v)  \wed \eta\wed \om^{n-m+2} \right)^\frac{1}{2}.
\end{aligned}$$
The dependence of $C$ is the same as in the previous lemma.
\end{lem}

\begin{proof} Since the supports of the integrands on both sides are contained in $K$, we replace $\rho$ by $\max\{\rho, A\psi\}$ for some $A>0$ large such that $A\psi \leq -1$ on a neighborhood of  $K$.
Thus, we may assume that $\rho \in P^*(\Om, K', A\psi)$ for a compact subset $K \subset K' \subset \subset \Om$.  
By dominated convergence theorem for Radon measures and the weak convergence in Proposition~\ref{prop:weak-conv} it is enough to prove the inequality for $v_\ell$ in the place of $v$, that is
\[\label{eq:IBP-b-1}\begin{aligned}
	\int_\Om (w - v_\ell) dd^c \rho \wed \eta\wed\om^{n-m+1}  
&\leq  - \int_{\Om} d(w-v_\ell) \wed d^c \rho \wed \eta\wed\om^{n-m+1} \\ 
&\quad + C \left(\int_\Om (w-v_\ell) \eta \wed \om^{n-m+2} \right)^\frac{1}{2}.
\end{aligned}\]
For a fixed $\ell$ the function  $w-v_\ell$ is continuous, with its support  in $\Om$. 
We can apply the limit argument once more to assume that  $\rho$ also is smooth as follows.
By the approximation theorem for $m-\om$-sh functions we can find a decreasing sequence $\{\rho_j\}_{j\geq 1}\subset P^*$ of smooth $m-\om$-sh functions such that $\rho_j\downarrow \rho$. We may  assume that $-A \leq \rho_j \leq 0$. Then,
$$
	dd^c\rho_j \wed \eta \wed \om^{n-m+1} \to dd^c \rho \wed \eta \wed \om^{n-m+1}
$$
weakly as measures. So, if the inequality holds for $\rho_j$, then so it does for $\rho$. 

Using again the integration by parts formula \eqref{eq:IBP-normal} we have
$$\begin{aligned}
	\int_\Om (w - v_\ell) dd^c \rho_j \wed \eta\wed\om^{n-m+1}  
&=  - \int_{\Om} d(w-v_\ell) \wed d^c \rho_j \wed \eta\wed\om^{n-m+1} \\ 
&\quad + \int_\Om (w-v_\ell)  d^c\rho_j\wed d\om \wed\eta  \wed \om^{n-m}.
\end{aligned}$$
As in the proof of Lemma~\ref{lem:w-IBP-a}, by an application of Cauchy-Schwarz inequality (Corollary~\ref{cor:cs-w-bdd}) and then the CLN inequality for the second integral on the right hand side, we conclude that  the inequality  holds for $\rho_j$ and so does  \eqref{eq:IBP-b-1}. \end{proof}

Lastly we can state the  important estimate for functions in $P^*$ and for  $\eta = dd^c u_1 \wed \cdots \wed dd^c u_{m-2}$,
 where $u_1,..,u_{m-2} \in P^*$ without strict positive assumption on $\eta$.

 \begin{lem}\label{lem:w-L-key}   Let $v,w$ and $\{v_\ell\}$ be as in Lemma~\ref{lem:w-IBP-a}. Let $\rho$ be a bounded $m-\om$-sh in $\Om$ such that $-1 \leq \rho \leq 0$. Then,
$$\begin{aligned}
	\int_\Om (w-v) dd^c \rho \wed \eta \wed \om^{n-m+1} 
&\leq C \left( \int_\Om (w-v) dd^c v\wed \eta\wed \om^{n-m+1} \right)^\frac{1}{2} \\
&\quad + C \left( \int_\Om (w-v) \eta\wed\om^{n-m+2} \right)^\frac{1}{2} \\
&\quad + C \left( \int_\Om (w-v) \eta\wed\om^{n-m+2} \right)^\frac{1}{4},
\end{aligned}$$
with the same dependence of $C$ as in previous lemmata.
\end{lem}

\begin{proof} Observe that it is enough to prove the inequality for $u_i:= u_i + \de |z|^2$ and then let $\de\to 0^+$. Therefore without loss of generality we assume that  $\eta$ satisfies  conditions \eqref{eq:add-a} and \eqref{eq:trace-m-1}. As in the proof of Lemma~\ref{lem:w-IBP-b}  we may assume that $\rho \in P^*$ and it is smooth.  It follows from \eqref{eq:IBP-b-1} that
\[\label{eq:L-key-0}\begin{aligned}
	\int_\Om (w - v_\ell) dd^c \rho \wed \eta\wed\om^{n-m+1}  
&\leq  - \int_{\Om} d(w-v_\ell) \wed d^c \rho \wed \eta\wed\om^{n-m+1} \\ 
&\quad + C \left(\int_\Om (w-v_\ell) \eta \wed \om^{n-m+2} \right)^\frac{1}{2}.
\end{aligned}\]
It remains to bound the first integral on the right hand side. Using the Cauchy-Schwarz inequality (Lemma~2.3) for $T = \eta\wed\om^{n-m+1}$ we get 
\[\label{eq:L-key-1}\begin{aligned}
	 \left| \int_{\Om} d(w-v_\ell) \wed d^c \rho \wed T\right|^2
&\leq  C_1  \int_{\Om} d(w-v_\ell) \wed d^c (w-v_\ell) \wed T,\\ 
\end{aligned}\]
where $C_1$ is the bound for 
$$
	C  \int_{K} d\rho  \wed d^c \rho \wed T
$$
and so the dependence of $C_1$ is as in the statement which is  a simple consequence of  the CLN inequality. 

Next,  we have
\[\label{eq:L-key-2}\begin{aligned}
\int_\Om d(w-v_\ell) \wed d^c (w-v_\ell) \wed T 
&	=  -\int_\Om (w-v_\ell) dd^c(w-v_\ell) \wed T \\
&	\quad + \int_\Om (w-v_\ell) d^c (w-v_\ell) \wed d\om \wed \eta \wed \om^{n-m}. \\
\end{aligned}
\]
Let us estimate the last integral. The Cauchy-Schwarz inequality and  the CLN inequality imply that
\[\label{eq:L-key-3}\begin{aligned}
&\left|\int_\Om (w-v_\ell) d^c (w-v_\ell) \wed d\om \wed \eta\wed \om^{n-m} \right|^2 \\
&	\leq C_2  \int_\Om (w-v_\ell) \eta\wed\om^{n-m+2},\\
\end{aligned}\]
where $C_2$ is the bound for
$$
	\int_\Om (w-v_\ell) d (w-v_\ell)\wed d^c (w-v_\ell) \wed \eta\wed \om^{n-m+1}.
$$
Notice also that as $w-v_\ell\geq 0$, 
\[\label{eq:L-key-4}
	- (w-v_\ell) dd^c (w-v_\ell) \wed T \leq (w -v_\ell) \wed dd^c v_\ell \wed T.
\]
Combining \eqref{eq:L-key-0}, \eqref{eq:L-key-1}, \eqref{eq:L-key-2}, \eqref{eq:L-key-3} and \eqref{eq:L-key-4} we get
\[\label{eq:L-key-1234}
\begin{aligned}
	\int_\Om (w-v_\ell) dd^c \rho \wed \eta\wed\om^{n-m+1} 
&	\leq C \left(\int_\Om(w-v_\ell) dd^c v_\ell \wed T\right)^\frac{1}{2} \\
&	\quad + C \left( \int_\Om (w-v_\ell) \eta\wed \om^{n-m+2}\right)^\frac{1}{2}\\
&	\quad + C \left( \int_\Om (w-v_\ell) \eta\wed \om^{n-m+2}\right)^\frac{1}{4}.
\end{aligned}
\]
Notice that  by smooth sequence $\rho_j \downarrow \rho$ we obtain this inequality for $\rho\in P^*$.
To finish the proof we need to know that  passing to the limit as  $\ell \to \infty$ we get the inequality in the statement.
This is known for the limits of all integrals above except for the first integral on the right hand side. 
To handle this,  we use the fact that  $v_\ell \geq v$. Then
\[\label{eq:L-key-5}
	\int_\Om (w-v_\ell) dd^c v_\ell \wed \eta\wed\om^{n-m+1} \leq \int_\Om (w-v) dd^c v_\ell \wed \eta\wed\om^{n-m+1}. 
\]
Applying Lemma~\ref{lem:w-IBP-b} for the right hand side, we have
\[\label{eq:L-key-6}\begin{aligned}
	\int_\Om (w-v) dd^c v_\ell \wed \eta \wed \om^{n-m+1} 
&\leq  - \int_\Om d(w-v)\wed d^c v_\ell \wed \eta\wed\om^{n-m+1} \\
&\quad + C  \left( \int_\Om (w-v_\ell) \tau \wed \om^{n-m+2} \right)^\frac{1}{2},
\end{aligned}\]
Combining three inequalities \eqref{eq:L-key-1234}, \eqref{eq:L-key-5} and \eqref{eq:L-key-6}, we obtain
$$\begin{aligned}
\int_\Om (w-v_\ell) dd^c\rho \wed \eta\wed \om^{n-m+1}
&\leq	 C \left(- \int_\Om d(w-v)\wed d^c v_\ell \wed \eta\wed\om^{n-m+1} \right)^\frac{1}{2}	 \\
&+\quad 	C  \left( \int_\Om (w-v_\ell) \eta \wed \om^{n-m+2} \right)^\frac{1}{2} \\
&+\quad 	C  \left( \int_\Om (w-v_\ell) \eta \wed \om^{n-m+2} \right)^\frac{1}{4}.
\end{aligned}
$$
Letting $\ell \to \infty$ and invoking Proposition~\ref{prop:weak-conv}  gives
$$\begin{aligned}
\int_\Om (w-v) dd^c\rho \wed \eta\wed \om^{n-m+1}
&\leq	 C \left(- \int_\Om d(w-v)\wed d^c v \wed \eta\wed\om^{n-m+1} \right)^\frac{1}{2}	 \\
&+\quad 	C  \left( \int_\Om (w-v) \eta \wed \om^{n-m+2} \right)^\frac{1}{2}\\
&+\quad 	C  \left( \int_\Om (w-v) \eta \wed \om^{n-m+2} \right)^\frac{1}{4}.
\end{aligned}
$$
Now the desired estimate follows from Lemma~\ref{lem:w-IBP-a}.
\end{proof}

\begin{remark}\label{rmk:continuity-w} In Lemmas~\ref{lem:w-IBP-a}, \ref{lem:w-IBP-b}, \ref{lem:w-L-key} we only need to assume that $w$ is continuous
after using one more approximation argument.  In fact, it is enough to assume $w\in P^*$.
\end{remark}

For further applications we need the following lemma.

\begin{lem}\label{lem:mass-convergence} Let $u$ be a bounded $m-\om$-sh. Suppose $\{u_j\}_{j\geq 1}$ is a decreasing sequence of smooth $m-\om$-sh functions with $\|u_j\|_{L^\infty} \leq 1$ such that  $u_j \downarrow u$ point-wise. Assume also that all $u_j=u$  on a neighborhood of $\d\Om$.  Let $-1 \leq \rho \leq 0$ be  $m-\om$-sh functions. Then, for every $0\leq s \leq m-1$,
\[\notag\label{eq:mass-conv1}
	\lim_{j\to +\infty} \sup_\rho \left\{ \int_\Om (u_j -u) (dd^c \rho)^s \wed \om^{n-s} \right\} =0.
\]
\end{lem}

\begin{proof} Let $K \subset \subset \Om$ be the compact set such that $u_j = u$ on $\Om\setminus K$ for all $j$. First we can modify $\rho$ outside $K$ by $A\psi$ without changing the integral value. Hence, we may assume that $\rho \in P^*$. Then,  applying repeatedly Lemma~\ref{lem:w-L-key}  to replace $dd^c\rho$ by $dd^cu$ we  get
$$
	\int_{\Om} (u_j-u) (dd^c\rho)^{m-1} \wed \om^{n-m+1} \leq C \sum_{s=1}^{m-1} [e_{(0,0,s)}]^{\frac{1}{2^{s}}},
$$
where 
$$
	e_{(0,0,s)} = \int_\Om (u_j -u) (dd^cu)^s \wed \om^{n-s}.
$$
Thus, the proof of the lemma follows by the dominated convergence theorem applied to each term on the right hand side.
\end{proof}

We get immediate consequences by the classical arguments. Let us define the $(m-1)$-capacity in the class of $m-\om$-sh functions. For a Borel set $E\subset \Om$,
$$
	cap_{m-1} (E) = \sup \left\{ \int_E H_{m-1} (\rho) : -1\leq \rho \leq 0,\; \rho \text{ is } m-\om\text{-sh in } \Om\right\}.
$$

\begin{cor} If $\{u_j\}_{j\geq 1}$ is a uniformly bounded decreasing sequence of smooth $m-\om$-sh functions such  that $u_j \downarrow u$ ($m-\om$-sh function). Then, $u_j \to u$ in $cap_{m-1}(\bullet)$.
\end{cor}

\begin{cor} \label{cor:w-quasi-continuity} Every bounded $m-\om$-sh function is quasi-continuous with respect to $cap_{m-1}(\bullet)$. \end{cor}

The quasicontinuity in $(m-1)$-capacity allows us to obtain the weak convergence for the wedge product of $m-\om$-sh  bounded functions (comp.  \cite[Proposition~3.2]{BT87}).  

\begin{prop}\label{prop:w-convergence-m} Let $0\leq p\leq m-1$. Let $v, u_1,...,u_{p}$ be uniformly bounded $m-\om$-sh functions in $\Om$.  Let $\{v_\ell\}_{\ell\geq 1}$, $\{u_s^j\}_{j\geq 1}$ be decreasing sequences of smooth $m-\om$-sh functions such that 
$v_\ell \downarrow v$ and $u_s^j \downarrow u_s$ for each $1\leq s\leq p$. Then
$$
	v_j dd^c u_1^{j} \wed \cdots \wed dd^c u_{p}^j \to v dd^c u_1 \wed \cdots \wed dd^c u_{p}
$$
as $j\to\infty$ in the sense of currents of order zero. Consequently, 
$$
	dd^c v_j  \wed dd^c u_1^j \wed \cdots \wed dd^c u_{p}^j \wed \om^{n-p}\to dd^c v  \wed dd^c u_1 \wed \cdots \wed dd^c u_{p} \wed \om^{n-p}
$$
in the sense of currents.
\end{prop}

\begin{proof} By the  localization principle we may assume $\Om$ is a ball and  all functions belong to $P^*$. Without loss of generality we may assume that $-1\leq v_j, u^j_s , v, u \leq 0$. Denote
$S^j = dd^c u_1^{j} \wed \cdots \wed dd^c u_{p}^j$ and $S = dd^c u_1 \wed \cdots \wed dd^c u_{p}$. Both of them are currents of order zero. Let $\chi$ be a test form with $\supp \chi = K$. We need to show that
$$
	\lim_{j\to \infty}\left| \int  (v_j S^j  - v S) \wed \chi \right| = 0.
$$
We argue by induction over $p$. It is obvious for $p=0$. Assume that it is true for $p-1 \leq m-2$. This means that
$$
	u_1^j dd^c u_2^j \wed \cdots \wed dd^c u^j_p \to u_1 dd^c u_2 \wed \cdots \wed dd^c u_p
$$
in the  sense of currents of order zero. Hence,
$$
	S^j = dd^c [u_1^j dd^c u_2^j \wed \cdots \wed dd^c u^j_p] \to  dd^c [u_1 dd^c u_2 \wed \cdots \wed dd^c u_p] = S.
$$
Furthermore, by \cite[Corollary~2.4]{KN3}, for a test form $\phi$, we have
\[\label{eq:w-conv-m}
	|S^j \wed \phi | \leq C \|\phi\|_K\wt S^j \wed \om^{n-p}, 
\]
where $\wt S^j = (dd^c \sum_{s=1}^{p} u_s^j)^p$ and $\|\phi\|_K$ is the $C^0$-norm of $\phi$ on $K$. Note that $\wt S^j \to \wt S = (dd^c\sum_{s=1}^p u_s)^p$ as the weak convergence holds for the  wedge product of currents associated to $p$ bounded functions. Letting $j\to \infty$ one obtains
\[\label{eq:w-conv-m-0}
	|S \wed \phi| \leq C \|\phi\|_{K} \wt S\wed \om^{n-p}.
\]

We are ready to show that the conclusion holds for $p \leq m-1$. Notice that the uniform constants $C$ below depend additionally on the uniform norm of coefficients of $\chi$.

Let $\veps>0$.  By the quasi-continuity with respect to $(m-1)$-capacity in Corollary~\ref{cor:w-quasi-continuity}, we can find an open set $G$ such that $cap_{m-1}(G)<\veps$ and $u_s^j, u_s, v$ are continuous on $\Om\setminus G$. Note that  $\Om\setminus G$ is compact in $\Om$. It follows from Dini's theorem that $u^j_s \to u_s$ uniformly as $j\to \infty$  on that set.  We first write
\[\label{eq:w-conv-m-1}
	\int (v_j S^j - vS) \wed \chi = \int (v_j - v) S^j \wed \chi+ \int v (S^j -S) \wed \chi.
\]
For the first term on the right, it follows from \eqref{eq:w-conv-m} that
\[\label{eq:w-conv-m-2}
	\left|\int (v_j -v) S^j \wed \chi \right| \leq C \left(  \int_{\Om\setminus G} (v_j-v) \wt S^j \wed \om^{n-p} + n^{p} cap_{m-1} (G) \right),
\]
where we used the fact that 
$$
	\int_G (v_j-v)  \wt S^j \wed \om^{n-p} \leq n^p cap_{m-1} (G).
$$ 
Next, we estimate the second integral in \eqref{eq:w-conv-m-1}. Let $-1\leq \wt v \leq 0$ be a continuous extension of $v$ from $\Om\setminus G$ to $\Om$. By induction hypothesis we know that 
$S^j\to S$ weakly, as currents. So, 
\[\label{eq:w-conv-m-3}
	\lim_{j\to \infty} \int \wt v (S^j -S) \wed \chi  =0.
\]
Moreover, similarly as above
$$
	\left|\int_G \wt v S^j \wed \chi \right| \leq C n^p cap_{m-1} (G),
$$
and by \eqref{eq:w-conv-m-0}
$$
	\left| \int_G \wt v S \wed \chi \right| \leq  C \left| \int_G  \wt S \wed \om^{n-p} \right| \leq Cn^p cap_{m-1} (G).
$$
Then,
\[\label{eq:w-conv-m-4} 
	\left| \int v (S^j -S) \wed \chi \right| \leq \left| \int \wt v (S^j -S)\wed \chi \right| + 4 C n^pcap_{m-1} (G).
\]

Combining \eqref{eq:w-conv-m-1}, \eqref{eq:w-conv-m-2}, \eqref{eq:w-conv-m-3} and \eqref{eq:w-conv-m-4} we conclude that
$$
	\lim_{j\to \infty} \left| \int (v_j S^j - vS) \wed \chi  \right| \leq C \veps.
$$
This holds for arbitrary $\veps$, so the proof of the proposition follows.
\end{proof}

\begin{remark} The proposition showed that the wedge product of currents associated to  continuous $m-\om$-sh functions given in \cite[Proposition~2.11]{KN3} coincides with the one in Definition~\ref{defn:w-prod}.
\end{remark}

\begin{cor} The wedge products $$dd^c u_1 \wed \cdots \wed dd^c u_{m-1}, \quad dd^c u_1 \wed \cdots \wed dd^c u_m \wed \om^{n-m}$$  are symmetric with respect to bounded $m-\om$-sh functions $u_1,...,u_{m-1}, u_{m}$.
\end{cor}

We can now extend the generalized Cauchy-Schwarz inequality in Lemma~2.4 and Corollary~2.5 to the case of bounded functions. We state here a simple version which is essential to prove the quasi-continuity later on. This is a stronger version of Corollary~\ref{cor:cs-w-bdd}. Let $-1\leq u_1,...,u_{m-1} \leq 0$ be bounded $m-\om$-sh functions in $\Om$. Denote
\[\label{eq:tau-cs}
	\tau = dd^c u_1 \wed \cdots \wed dd^c u_{m-1}.
\]
\begin{cor} \label{cor:cs-bdd}Let $0\leq \phi \leq 1$ be a continuous function such that $\supp\phi = K \subset\subset \Om$. Let $0\leq h\leq 1$ be a smooth $m-\om$-sh function. There exists a constant $C = C (\om, K)$ such that
$$
	\left| \int \phi dh \wed d^c \om \wed \tau \wed \om^{n-m-1}\right|^2 \leq C \int \phi \; \wt\tau \wed \om^{n-m+1},
$$
where $\wt\tau = \left(dd^c \sum_{s=1}^{m-1} u_s \right)^{m-1}$.
\end{cor}

\begin{proof} Notice that the left hand side is well-defined because the coefficients of $S$ are signed  Radon measures. Let $\{u^j_s\}_{j\geq 1}$ be a sequence of  smooth $m-\om$-sh functions such that    $u_s^j \downarrow u_s$ for each $1\leq s\leq  m-1$. We may assume $-1\leq u_s^j \leq 0$. Denote $$
	S^{j} = dd^c u^{j}_1 \wed \cdots dd^c u^{j}_{m-1}, 
	\quad \wt S^j = (dd^c u_1^j + \cdots + dd^c u_{m-1}^k)^{m-1}
$$
Applying Corollary 2.5 for $\ga_s = dd^c u^{j}_s$, $s =1,...,p$, we get 
$$\begin{aligned}
&	\left |\int \phi\; dh \wed d^c \om \wed S^{j}  \wed  \om^{n-m-1} \right|^2 \\
&\leq 	C\int \phi \;dh\wed d^c h \wed (\sum_{s=1}^{m-1} \ga_s)^{m-1} \wed \om^{n-m}  \times  \int \phi\;\wt{S}^{j} \wed \om^{n-m+1}.
\end{aligned}$$
As in the proof of Corollary~\ref{cor:cs-w-bdd},  the CLN inequality implies that the first integral inequality on the right hand side is bounded by a constant $C = C(\om, K)$. By Proposition~\ref{prop:w-convergence-m} the last integral converges to
$\int\; \phi \; \wt\tau \wed \om^{n-m+1}.$ Therefore, the proof of the corollary follows.
\end{proof}

\section{Quasi-continuity}
Having  the Hessian measure  defined for bounded $m-\om$-sh functions, we can introduce the $m$-capacity (cf. \cite{BT82}): for a Borel subset $E \subset \Om$,
\[
	cap_{m} (E) = \sup \left\{ \int_E (dd^c\rho)^{m} \wed \om^{n-m} : \rho \text{ is } m-\om \text{-sh in }\Om, -1 \leq \rho \leq 0\right\}.
\]
Here in fact $cap_m(E) = cap_m(E,\Om)$ but we shall often suppress $\Om$ in the notation if the  domain is fixed.
Then, this is an inner capacity, namely,
$$	cap_m(E) = \sup\{cap_m(K): K \text{ is compact subset of } E\}.
$$

\begin{prop} \label{prop:cap-properties}Let $\Om$ be a open set in $\bC^n$ and $cap_m(E) = cap_m(E,\Om)$. Then,
\begin{itemize}
\item
[(a)] If $E_1 \subset E_2$, then $cap_m(E_1) \leq cap_m(E_2)$.
\item
[(b)] If $E \subset \Om_1 \subset \Om_2$, then $cap_m(E,\Om_2) \leq cap_m(E, \Om_1)$.
\item
[(c)] $cap_m( \cup_j E_j) \leq \sum_{j} cap_m(E_j)$.
\item
[(d)] If $E_1\subset E_2 \subset \cdots $ are Borel set in $\Om$ and $E:= \cup_j E_j$, then $cap_m(E) = \lim_j cap_m(E_j)$.
\end{itemize}
\end{prop}

\begin{defn}[Convergence in capacity] A sequence of Borel functions  $u_j$ in $\Om$ is said to converge in capacity (or in $cap_m(\bullet)$) to $u$ if for any $\de>0$ and $K \subset\subset \Om$,
$$
	\lim_{j\to \infty} cap_m(K \cap |u_j-u| \geq \de) =0.
$$
\end{defn}

\begin{prop}\label{prop:decreasing-seq-cap} If $\{u_j\}_{j\geq 1}$ be a uniformly bounded sequence of continuous $m-\om$-sh  functions that  decreases to a bounded $m-\om$-sh function $u$ in $\Om$. Then, $u_j$ converges to $u$ in $cap_{m}(\bullet)$.
\end{prop}

The proof of this proposition will need the improved versions of Lemma~\ref{lem:w-IBP-a}, \ref{lem:w-IBP-b}, \ref{lem:w-L-key} in which $\eta \wed \om$ is replaced by $$\tau = dd^c u_1 \wed \cdots dd^c u_{m-1},$$
where $u_1,...,u_{m-1} \in P^*$ in \eqref{eq:tau} and the positivity assumption \eqref{eq:add} is satisfied. This is done thanks to  Corollary~\ref{cor:cs-bdd}. Since the proofs follow line by line those from the previous section after
 replacing $\eta\wed \om$ by $\tau$ and applying Corollary~\ref{cor:cs-bdd} instead of  Corollary~\ref{cor:cs-w-bdd}, we only state the lemmas.

\begin{lem}\label{lem:IBP-a} Let $\{v_\ell\}_{\ell\geq 1}$ be a sequence of smooth functions from $P^*$ and let $w\in P^*$ be smooth. Assume $v_\ell\leq w$ and $v_\ell \downarrow v \in P^*$. Then,
$$\begin{aligned}
	-\int d(w-v) \wed d^c v \wed \tau\wed \om^{n-m} 
&\leq \int(w-v) dd^c v \wed \tau\wed\om^{n-m} \\
&+ C \left( \int(w-v) \tau\wed \om^{n-m+1}\right)^\frac{1}{2},
\end{aligned}$$
where  the constant $C$ depends on $\om, K, \psi$ and the uniform norm of functions, but it is independent of  $\de$ in \eqref{eq:add}.
\end{lem}

\begin{lem}\label{lem:IBP-b} Let $v,w$ and $\{v_\ell\}$ be as in Lemma~\ref{lem:IBP-a}. Let $\rho$ be a bounded $m-\om$-sh in $\Om$ such that $-1 \leq \rho \leq 0$. Then,
$$\begin{aligned}
	\int_\Om (w - v) dd^c \rho \wed \tau\wed\om^{n-m}  
&\leq  - \int_{\Om} d(w-v) \wed d^c \rho \wed \tau\wed\om^{n-m} \\ 
&\quad + C \left(\int_\Om (w-v)  \wed \tau\wed \om^{n-m+1} \right)^\frac{1}{2}.
\end{aligned}$$
The dependence of $C$ is the same as in the previous lemma.
\end{lem}

Notice that as in Lemma~\ref{lem:w-L-key} the strict positivity assumption \eqref{eq:add} is not needed  in the following statement.

\begin{lem}\label{lem:L-key}   Let $v,w$ and $\{v_\ell\}$ be as in Lemma~\ref{lem:IBP-a}. Let $\rho$ be a bounded $m-\om$-sh in $\Om$ such that $-1 \leq \rho \leq 0$. Then,
$$\begin{aligned}
	\int_\Om (w-v) dd^c \rho \wed \tau \wed \om^{n-m} 
&\leq C \left( \int_\Om (w-v) dd^c v\wed \tau\wed \om^{n-m} \right)^\frac{1}{2} \\
&\quad + C \left( \int_\Om (w-v) \tau\wed\om^{n-m+1} \right)^\frac{1}{2} \\
&\quad + C \left( \int_\Om (w-v) \tau\wed\om^{n-m+1} \right)^\frac{1}{4}
\end{aligned}$$
With the same dependence of $C$ as in previous lemmata.
\end{lem}

\begin{remark}\label{rmk:continuity} In Lemmas~\ref{lem:IBP-a}, \ref{lem:IBP-b}, \ref{lem:L-key} it is enough  to assume that $w\in P^*$ after using one more approximation argument.
\end{remark}

Thanks to Lemma~\ref{lem:L-key}, we can easily prove the last case $s=m$ of Lemma~\ref{lem:mass-convergence}.

\begin{lem}\label{lem:mass-convergence-m} Let $u$ be a bounded $m-\om$-sh. Suppose $\{u_j\}_{j\geq 1}$ is a decreasing sequence of smooth $m-\om$-sh functions with $\|u_j\|_{L^\infty} \leq 1$ such that  $u_j \downarrow u$ point-wise. Assume also that all $u_j=u$  on a neighborhood of $\d\Om$.  Let $-1 \leq \rho \leq 0$ be  $m-\om$-sh functions. Then, 
\[
	\lim_{j\to +\infty} \sup_\rho \left\{ \int_\Om (u_j -u) (dd^c \rho)^m \wed \om^{n-m} \right\} =0.
\]
\end{lem}


\begin{proof}[End of proof of Proposition~\ref{prop:decreasing-seq-cap}]  Because of (b) and (c) in Proposition~\ref{prop:cap-properties}  we can assume that $\Om$ is a ball and all functions are equal near the boundary. Let $\de>0$. We wish to show that 
$$
	\lim_{j\to \infty}  cap_m (\{u_j-u>\de\}) =0.
$$
Argument by contradiction. Suppose that the statement were not true. Then, there would exist $\veps>0$, a subsequence $\{u_{j_s}\} \subset \{u_j\}$ and  a sequence of $m-\om$-sh functions $\rho_{j_s}$ with $-1 \leq \rho_{j_s} \leq 0$ such that  
$$
	\limsup_{j_s\to +\infty} \int_{\{u_{j_s} -u>\de\}} H_m(\rho_{j_s}) \geq \veps.
$$
On the other hand, by Markov's inequality,
$$	
	\int_{\{u_{j_s}-u>\de\}} H_m(\rho_{j_s}) \leq \frac{1}{\de} \int_\Om (u_{j_s}-u) H_m(\rho_{j_s}). 
$$
Lemma~\ref{lem:mass-convergence-m} shows  that the right hand side converges to zero. This leads to a contradiction. Thus, the proof of proposition is completed.
\end{proof}

\begin{thm} \label{thm:quasi-continuity} Let $\Om$ be a bounded open set in $\bC^n$. Let $u$ be a $m-\om$-sh function in $\Om$. Then, for every $\veps>0$, there exists an open subset $U \subset \Om$ with $cap_m(U,\Om) < \veps$ such that $u$ restricted to $\Om\setminus U$ is continuous.
\end{thm}

\begin{proof} The result is local and moreover, it can be reduced to the case of bounded functions by  the property \eqref{eq:cap-sublevel-set} whose proof will use only bounded functions. We can use the classical argument in \cite[Theorem~3.5]{BT82} since  there exists a sequence of smooth $m-\om$-sh functions that decrease to $u$ point-wise (Proposition~\ref{prop:smoothing}) and our capacity is  subadditive by Proposition~\ref{prop:cap-properties}. 
\end{proof}

We obtain the convergence in capacity for monotone sequences of uniformly bounded functions. In particular, the smoothness assumption in Proposition~\ref{prop:decreasing-seq-cap} can be relaxed  by yet another approximation. We recall first  a classical result in measure theory,  giving a short proof for the reader's convenience.

\begin{lem} \label{lem:basic-closed} Let $\mu_j$ be a sequence of positive Radon measures with compact support in $\Om$. Assume that $\mu_j$ converges weakly to $\mu$.  Let $\Om \supset\supset F_1 \supset F_{2} \supset \cdots$ be a sequence of decreasing  closed subsets in $\Om$ satisfying 
$$ \lim_{j\to \infty} \mu( F_j) = 0.$$
Then, 
$
	\lim_{j\to +\infty} \mu_j(F_j) =0.
$
\end{lem}

\begin{proof} 
Fix $\veps>0$. By the assumption there exists $j_0>0$ such that $\mu(F_{j_0}) < \veps$.  Using the inclusions we have
$
	\mu_j(F_j) \leq \mu_j(F_{j_0}), 
$ for $ j>j_0.$
The weak convergence implies
$$
	 \limsup_{j\to \infty} \mu_j (F_j) \leq \limsup_{j\to\infty} \mu_j(F_{j_0})  \leq \mu (F_{j_0}) < \veps.
$$
It follows that
$$
	0\leq \liminf_{j\to \infty} \mu_j(F_j) \leq \limsup_{j\to \infty} \mu_j(F_j) \leq \veps.
$$
This holds for every $\veps>0$. Thus, the conclusion follows.
\end{proof}
 
\begin{cor} \label{cor:monotonicity} Let $\Om$ be a bounded open set in $\bC^n$. Let $\{u_j\}_{j\geq 1}$ be a uniformly bounded and monotone sequence of $m-\om$-sh functions that either $u_j \downarrow u$ point-wise or $u_j \uparrow u$ almost everywhere for a bounded $m-\om$-sh function $u$ in $\Om$. Then, $u_j$ converges to $u$ in $m-$capacity.
\end{cor}

\begin{proof} 
The localization principle applies in both cases, we can assume that $\Om$ is a ball and  $u_j$'s are equal to  a fixed smooth psh function in a neighborhood of the boundary. Hence, $u- u_j =0$ on $\Om\setminus K$ for a fixed compact subset $K$.  Let $\de>0$, we wish to show that 
$$
	\lim_{j\to \infty} cap_m(|u-u_j| \geq \de) =0.
$$
Arguing by contradiction, suppose that this were not true. Then, there would exist $
\veps>0$ and a sequence of $m-\om$-sh functions $\rho_j$ with $-1\leq \rho_j\leq 0$ such that
$$
	\limsup_{j\to \infty} \int_{\{|u-u_j| \geq \de\}} H_m(\rho_j) \geq \veps.
$$
Fix a cut-off function $\chi$ with compact support in $\Om$ and equal $1$ on $K$. By the CLN inequality for bounded functions (Proposition~\ref{prop:CLN-bdd}) we may assume that $\mu_j:=\chi H_m(\rho_j)$ converges weakly to a positive Radon measure $\mu$.  

We use the quasi-continuity. Find an open set $G\subset \Om$ such that $cap_m(G)$  $ \leq \veps/2$ and the restrictions of $u_j, u$ to  $\Om\setminus G$ are continuous functions. Since $u_j$ is a monotone sequence, it follows from Dini's lemma that it converges uniformly on $\Om\setminus G$.  In particular, the sets $F_j:= \{|u-u_j| \geq \de\} \setminus G$ are closed in $\Om$ and satisfy $\lim_{j\to\infty}\mu(F_j)=0$. 
Hence, 
$$
	\int_{\{|u-u_j|\geq \de\}} H_m(\rho_j) \leq \int_{F_j} \chi H_m(\rho_j) + cap_m(G) \leq \mu_j(F_j) + \veps/2.
$$
Letting $j\to \infty$ and using Lemma~\ref{lem:basic-closed} this leads to a contradiction.
\end{proof}

\section{Weak convergence}

\subsection{Convergence theorems for decreasing sequences}

Let $\Om$ be a bounded open set in $\bC^n$. We have continuity of wedge products of $m-\om$-sh functions under decreasing sequences of bounded functions.

\begin{lem} \label{lem:bounded-decreasing} Let $v, u_1,..., u_p$, $1\leq p\leq m$, be a bounded $m-\om$-sh functions in $\Om$. Let $\{v^j\}_{j\geq 1}$ and $\{u^j_s\}_{j\geq 1}$ be uniformly bounded sequences of $m-\om$-sh such that $v^j\downarrow v$ and $u^j_s\downarrow u_s$ as $j \to +\infty$,  for each $s=1,...,p$. Then,
\begin{itemize}
\item
[(a)]  $\lim_{j\to\infty} \Lc(u^j_1,...,u^j_p) = \Lc(u_1,...,u_p)$;
\item
[(b)]
$
	\lim_{j\to +\infty} v^j \Lc(u_1^j,...,u_p^j) =v \Lc(u_1,...,u_p);
$
\end{itemize}
where the convergence is understood in the sense of currents of order zero.
\end{lem}

\begin{proof}  (a) By the localization principle we may assume that all functions are defined in a ball $\Om$ and they are equal to a fixed smooth psh function $\psi$ outside $\Om'\subset\subset \Om$.  Let $\{u_s^{j,\de}\}_{\de>0}$  be decreasing sequences of smooth $m-\om$-sh functions such that $u_s^{j,\de}\downarrow u_s^j$ as $\de\to 0$. Similarly, let $\{u_s^\de\}_{\de>0}$ be approximating sequences for $u_s$. We may assume that all involved functions are negative and of uniform norm less than one.

Let $\chi$ be a test form whose $\supp \chi =K \subset\subset \Om$. We consider the difference
$$
 M_{(j,\de)} = \int \chi \left[ \Lc(u_1^{j,\de},...,u_p^{j,\de}) -  \Lc(u_1^{\de},...,u_p^{\de}) \right].
$$
By Proposition~\ref{prop:w-convergence-m}  
$\lim_{\de\to 0}  \Lc(u_1^{\de},...,u_p^{\de}) =  \Lc(u_1,...,u_p).$
So, the proof is completed as soon as we  show  that
$$
	\lim_{j\to \infty} \lim_{\de\to 0} |M_{(j,\de)}| =0.
$$
We claim that 
\[\label{eq:bdd-decreasing-1}
	|M_{(j,\de)}| \leq C \sum_{s=1}^p \int_K |u_s^{j,\de} - u_s^\de| (dd^c \rho^{j,\de})^{p-1} \wed \om^{n-p+1},
\]
where $\rho^{j,\de} = \frac{1}{2p}\sum_{s=1}^{p} (u^{j,\de}_s + u_{s}^\de)$. In fact, 
$$
	[\Lc(u_1^{j,\de},...,u_p^{j,\de}) - \Lc(u_1^\de,...,u_p^\de)]= \sum_{s=1}^pdd^c (u^{j,\de}_{s} -u^\de_s) \wed T_s \wed \om^{n-m},
$$
where $$T_s(j,\de) =  dd^c u_{1}^{j,\de} \wed \cdots \wed dd^c u_{s-1}^{j,\de} \wed dd^c u_{s+1}^\de \wed \cdots \wed dd^c u_{p}^\de.$$ 
(Here one should use the obvious modifications for $s=1$ and $s=p$.) 
 Notice that $T_s$ is a smooth closed $(p-1,p-1)$-form. 
By integration by parts,
$$
	\int \chi  \om^{n-m} \wed dd^c (u_s^{j,\de} - u_s^\de) \wed T_p   = \int (u_s^{j,\de} - u_s^\de) dd^c (\chi \om^{n-m}) \wed T_s.
$$
Since $0\leq p-1 \leq m-1$, it follows from  \cite[Corollary~2.4]{KN3} that
$$
	| dd^c (\chi \om^{n-m}) \wed T_s | \leq 2^mC [dd^c (u_1^{j,\de} +\cdots + u_{p}^\de)]^{p-1} \wed \om^{n-p+1}
$$
for a uniform constant $C$ depending only on $\om$ and $\chi$. This proved the claim \eqref{eq:bdd-decreasing-1}.

At this point we no longer have  the continuity of $u_s^j$ and $u_s$, however, we can make use of the quasi-continuity. 
Let $\veps>0$. Find an open set $G$ such that all $u^{j,\de}_s$, $u^{j}_s$ and also $u^j_s, u_s$ are continuous on $\Om\setminus G$ and $cap_m(G)<\veps$. We know that $\Om\setminus G$ is compact in $\Om$ and by Dini's theorem for $s=1,...,p$ we have $u_s^{j,\de} \to u_s^j$ and $u^\de_s \to u_s$ uniformly as $\de\to 0$ on that set. Therefore,
$$\begin{aligned}
	\int_K |u_s^{j,\de} - u_s^\de| (dd^c \rho^{j,\de})^{p-1} \wed \om^{n-p+1}
& \leq  \int_{\Om\setminus G} |u_s^{j,\de} - u_s^\de| (dd^c \rho^{j,\de})^{p-1} \wed \om^{n-p+1} \\ 
&\quad	+ cap_m(G).
\end{aligned}$$
To estimate 
the integral on the right hand side we use
$$\begin{aligned}
&	 \int_{\Om\setminus G} |u_s^{j,\de} - u_s^\de| (dd^c \rho^{j,\de})^{p-1} \wed \om^{n-p+1}\\ &\leq   \int_{\Om\setminus G} (u_s^{j,\de} - u^j_s)  (dd^c \rho^{j,\de})^{p-1} \wed \om^{n-p+1} \\
&\quad + \int_{\Om\setminus G} (u_s^{\de} - u_s)  (dd^c \rho^{j,\de})^{p-1} \wed \om^{n-p+1} \\
&\quad+  \int_{\Om\setminus G} (u_s^{j} - u_s)  (dd^c \rho^{j,\de})^{p-1} \wed \om^{n-p+1}.
\end{aligned}$$
By the uniform convergence and then the CLN inequality (Proposition~\ref{prop:CLN-bdd}) the first two terms go to zero as $\de\to 0$. 
Moreover,  $\rho^{j,\de}$ is a sequence of smooth $m-\om$-sh functions decreasing to $\rho^j = \frac{1}{2p} \sum_{s=1}^p (u^j_s + u_s)$ as $\de\to 0$. This implies by Proposition~\ref{prop:w-convergence-m} that $\Lc_{p-1}(\rho^{j,\de})$ converges weakly to $\Lc_{p-1}(\rho^j)$ as $\de\to 0$. Combining this  with the continuity on $\Om\setminus G$ we get
$$
	\lim_{\de \to 0} \int_{\Om\setminus G} (u^{j}_s - u_s) \Lc_{p-1}(\rho^{j,\de}) = \int_{\Om\setminus G} (u^{j}_s - u_s) \Lc_{p-1}(\rho^{j}).
$$
It follows that 
$$
\lim_{\de\to 0} |M_{(j,\de)}| \leq \int_{\Om\setminus G} (u^{j}_s - u_s) \Lc_{p-1}(\rho^{j}) +\veps. 
$$
Letting $j\to +\infty$ we have by the  uniform convergence,
$$
		\lim_{j\to \infty} \lim_{\de\to 0} |M_{(j,\de)}| \leq \veps.
$$
Since $\veps>0$ is arbitrary, the proof is completed.

(b)
 For simplicity we assume that $u_1 =\cdots u_p = u$ and also
 $\{u_s^j\} = \{u^j\}$. The general case follows in the same way. Then we write 
$$
	\Lc_p(u) =\Lc_p(u,...,u).
$$
Since $v^j$ decreases to $v$ and $\Lc_p(u^j)$ converges weakly to $\Lc_p(u)$ thanks to (a), any weak limit $\Te$ of the sequence $v^j \Lc_p(u^j)$ satisfies 
$$
 \Te \leq v \Lc_p(u).
$$
Hence, $v \Lc_p(u) - \Te$ is a positive current. In particular, $$v\Lc_p(u) \wed \om^{m-p} -\Te\wed \om^{m-p} = v H_p(u) - \Te\wed \om^{m-p}$$ is a positive Radon measure. To show the converse let $\chi \geq 0$ be a test function in $\Om$, we will show that
$$
	\int \chi \Te \wed \om^{m-p} \geq  \int \chi v H_p(u).
$$
In fact, since $v^j H_p(u^j)$ converges weakly to $\Te\wed\om^{m-p}$, it is enough to show
\[\label{eq:converse-mass-ineq}
 	\lim_{j\to +\infty}\int \chi v^j H_p(u^j) \geq \int \chi v H_p(u).
\]
Let $\veps>0$ and choose an open set $G\subset \Om$ such that $cap_m(G, \Om) \leq \veps$ and $v, v_s$ are all continuous on $F = \Om\setminus G$. Since $v$ is continuous on $F$, there is a continuous extension $g$ to $\Om$ such that $v = g$ on $F$ with the same uniform norm. 
Without loss of generality we also assume that $0\leq v, v_s, g \leq 1$. Hence,
$$\begin{aligned}
	\int \chi v H_p(u) 
&\leq \int_F \chi v H_p(u) + cap_m(G) \\
&= \int_F \chi g H_p(u) + cap_m(G) \\
&\leq \int \chi g H_p(u) + cap_m(G) \\
&= \lim_{j\to \infty} \int \chi g H_p(u^j) + cap_m(G) \\
&\leq \lim_{j\to\infty} \int_F \chi g H_p(u^j) + 2cap_m(G).
\end{aligned}$$
By Dini's theorem $v^j$ converges to $v=g$ uniformly on $F$, so the last integral does not exceed
$$
	\lim_{j\to \infty} \int_F \chi v^j H_p(u^j) + \veps \leq \lim_{j\to \infty} \int \chi v^j H_p(u^j) + \veps.
$$
Therefore, we have proved that
$$
	\int \chi v H_p(u) 
\leq \lim_{j\to \infty} \int \chi v^j H_p(u^j) + 2 cap_m(G) + \veps .
$$
Since $cap_m(G) \leq \veps$ and $\veps>0$ is arbitrary, the proof of the inequality \eqref{eq:converse-mass-ineq} follows, so does the one of the lemma.
\end{proof}

\begin{cor} Let $u,v$ be bounded $m-\om$-sh functions and $T = dd^c v_1 \wed \cdots dd^c v_{m-p} \wed \om^{n-m}$ for bounded $m-\om$-sh functions $v_1,...,v_{m-p}$, where $1\leq p\leq m$. Then,
$$
	{\bf 1}_{\{u<v\}} (dd^c \max\{u,v\})^p \wed T= {\bf 1}_{\{u<v\}} (dd^c v)^p\wed T.
$$
Consequently,
$$
	(dd^c \max\{u,v\})^p \wed T \geq {\bf 1}_{\{u\geq v\}} (dd^c u)^p\wed T+ {\bf 1}_{\{u<v\}} (dd^cv)^p\wed T.
$$
\end{cor}

\begin{proof} Given the weak convergence results under decreasing sequences in the above lemma, the proof of \cite[Theorem~3.27]{GZ-book} can be easily adapted to the current case. 
\end{proof}

We conclude this section by going back to the extensions of results in Section~\ref{ssec:integral-smooth-fct}, noted in Remark~\ref{rmk:bounded-case}. We give here only the most important statement that will be used later.

 \begin{cor}\label{cor:bounded-case} Let $\Om\subset\subset \bC^n$ be strictly $m$-pseudoconvex. Let $-1 \leq v \leq w\leq 0$ be bounded $m-\om$-sh functions such that $\lim_{z\to \d\Om}(w-v)=0$.  Let $\rho$ be a bounded $m-\om$-sh function such that $-1 \leq \rho \leq 0$. There is a constant $C = C(\om,n,m)$ such that
 $$
 	\int_\Om (w-v)^{3m} (dd^c\rho)^m \wed \om^{n-m} \leq C \sum_{s=0}^m \int_\Om (w-v) (dd^cv)^s\wed \om^{n-s}.
 $$
 \end{cor}
\begin{proof} Let us replace $w$ by $w_\veps = \max\{w-\veps, v\}$ for $\veps>0$, so that $w_\veps = v$ in a neighborhood of $\d\Om$. If we could prove the inequality for $w_\veps$ and $v$, then by letting $\veps\to 0$, the domination convergence theorem would imply the required inequality.  
Let $\Om'\subset\subset \Om$ be a smooth subdomain such that $w=v$ on $\Om\setminus \Om'$. Then the integrals on both sides will not change if we modify $v,w$ outside $\Om'$. Hence, we may further assume that $w=v = \psi$ on $\Om\setminus \Om'$ with $\psi$  a smooth $m-\om$-sh defining function for $\Om$.

Using the quasi-continuity it is easy to see from Lemma~\ref{lem:bounded-decreasing}-(b) that for smooth decreasing sequences $w_j \downarrow w$, $v_j \downarrow v$ and $\rho_j \downarrow \rho$ we have
$$
	\lim_{j\to \infty} \int_\Om (w_j -v_j)^{3m} H_m(\rho_j) = \int_\Om (w -v)^{3m} H_m(\rho), 
$$
and for $0\leq s\leq m$,
$$
	\lim_{j\to \infty} \int_\Om (w_j -v_j)H_s(v_j) = \int_\Om (w -v) H_s(v).
$$
Therefore, it is enough to prove the inequality for smooth functions $ v_j \leq w_j \leq 0$ and $-1\leq \rho\leq 0$. Notice that $w_j \to w$ and $v_j \to v$  uniformly on $\Om\setminus \Om'$ (this is the reason why we modify $w, v$ near the boundary). Thus, we can follow the argument in Remark~\ref{rmk:continuous-case} and conclude that the extra terms will vanish after passing to the limit as $j\to +\infty$. Hence, the proof for the bounded functions case follows.
\end{proof}

\subsection{Convergence theorems for increasing sequences}

With a similar proof as that of Lemma~\ref{lem:bounded-decreasing} we get
\begin{lem} \label{lem:bounded-increasing}Let $1\leq p\leq m$ and $v, u_1,...,u_p$ be  bounded $m-\om$-sh functions.
Suppose that $\{v^j\}_{j\geq 1}$, $\{u_s^j\}_{j\geq 1}$ are uniformly bounded increasing sequences of $m-\om$-sh functions such that  $v^j\uparrow v$ and $u^j_s \uparrow u_s$ (almost everywhere) as $j\to \infty$ for $s=1,...,p$. Then,
\[
	\lim_{j\to +\infty} v^j \Lc(u_1^j,...u_p^j) = v \Lc(u_1,...,u_p)
\] 
in the sense of currents of order zero.
\end{lem}

\begin{cor}\label{cor:inner-zero-cap} Let $\Om$ be a bounded open set. Let  $\cU_m$ be a  uniformly bounded family of $m-\om$-sh functions  in $\Om$. Denote $v(x) = \sup\{v_\al(x): v_\al \in \cU_m\}$. Then, the set 
$$
	N:=	\{v <v^*\}
$$
has zero measure with respect to any measure $\Lc(u_1,...,u_m) = dd^c u_1 \wed \cdots \wed dd^c u_m\wed \om^{n-m}$, where $u_i$'s are bounded $m-\om$-sh functions. In particular, $cap_m(N,\Om) =0$.
\end{cor}

\begin{proof} By Choquet's lemma we can reduce the argument  to the case when $\cU_m$ is an increasing sequence $\{v_j\}_{j\geq 1}$ with $w = \sup_j v_j$ and $N = \{w< v^*\}$. 
It follows from the proof of \cite[Corollary 9.9]{GN18} that $w = v^*$ almost everywhere.
Therefore, Lemma~\ref{lem:bounded-increasing} implies 
$v_j \Lc(u_1,...,u_m) $ converges weakly to $v^* \Lc(u_1,...,u_m) $. Then, the positive currents $(v^*- v_j) \Lc(u_1,...,u_m)$ converge weakly to zero and hence, for any compact set $K\subset \Om$, 
$$
	\lim_{j\to\infty} \int_K (v^*- v_j) \Lc(u_1,...,u_m) =0.
$$
By monotone convergence theorem, 
$
	\lim_{j\to \infty}	\int_{K} (w-v_j) \Lc(u_1,...,u_m) =0.
$
Therefore,  $\int_K (v^*-w)\Lc(u_1,...,u_m)=0$. In other words, $v^*= w$ a.e on $K$ with respect to $\Lc(u_1,...,u_m)$.  The last conclusion follows from the inner regularity of capacity.
\end{proof}

\section{Comparison principle}

Let $\Om$ be a bounded open set which is relatively compact  in a strictly $m$-pseudoconvex bounded domain $D$ in $\bC^n$. 
Fix a constant ${\bf B}$   such that on $\ov \Om$,
$$
	- \bb \om^2 \leq dd^c \om \leq \bb \om^2, \quad
	-\bb \om^3 \leq d\om \wed d^c \om \leq \bb \om^3.
$$
Let $\rho$ be a strictly psh function sasifying $\rho \leq 0$  and $dd^c \rho \geq \om$ in $D$. In this section we assume all function are defined in $D$ which means that they can be approximated by a decreasing sequence of smooth $m-\om$-sh functions in a neighborhood of $\ov\Om$.
\begin{thm}\label{thm:CP1}
Let $u,v$ be bounded $m-\om$-sh functions in $\Om$ such that $d = \sup_{\Om} (v-u) >0,$ and $\liminf_{z\to \d\Om} (u-v)(z) \geq 0$.
Fix $0< \veps < \min\{\frac{1}{2}, \frac{d}{2 \|\rho\|_\infty}\}$. Let us denote for $0< s< \veps_0:=\veps^3/16 \bb,$
$$
	U(\veps,s):=\{u<(v+\veps\rho) + S(\veps) +s \}, \quad\text{where }
	S(\veps)= \inf_\Om[ u - (v+\veps\rho)].
$$
Then,
$$
	\int_{U(\veps,s)} H_m (v+\veps\rho)  \leq 
	\left( 1+ \frac{C s}{\veps^m} \right) \int_{U(\veps,s)} H_m(u), 
$$
where $C$ is a uniform constant depending on $m,n,\om$.
\end{thm}

\begin{proof} If $u, v$ are smooth, then the proof follows from \cite[Lemmas~3.8, ~3.9 and 3.10]{GN18}. To pass from the smooth case to the bounded case we use the quasi-continuity of $m-\om$-sh functions and the argument as the one in \cite[Theorem~4.1]{BT82} (see also \cite[Theorem~1.16]{K05}). The proof is readily adaptable with obvious changes of notation. Here we only indicate the points of difference that we need to take care of. Firstly, replacing $u$ by $u+\de$ with $\de>0$ and then letting $\de\downarrow 0$ we may assume that $\{u<v\} \subset \subset \Om' \subset\subset \Om$ and $u \geq v + \de$ on $\Om\setminus \Om'$. By restricting $u,v$ to a smaller domain we may assume that $u, v$ are defined in a neighborhood of $\ov\Om$.

Let $\{u_k\}_{k\geq 1}, \{v_j\}_{j\geq 1}$ be sequences of smooth $m-\om$-sh functions in a neighborhood of $\ov\Om$ (Proposition~\ref{prop:smoothing}) such that $u_k \downarrow u$ and $v_j \downarrow v$ point-wise in $\ov\Om$. Denote $d_{jk} = \sup_{\ov\Om} (v_j-u_k)$. Then, for $j \geq k >0$ large we have
$$
   	d_{jk} \geq d/2 >0.
$$
In fact, for small $\eps>0$   there exits $x\in \Om$ such that $d-\eps \leq v(x) -u(x)$. So, for $k>k_0$ large enough, $$d-2\eps \leq v(x) - u_k(x) \leq v_j - u_k \leq d_{jk}.$$ We get the desired inequality by letting $j\to \infty$ and then $\eps\to 0$.
Next, since $u \geq v+\de$ on a compact set $K= \ov\Om\setminus \Om'$, we have $u_k \geq v + \de$ for every $k\geq 1$. Since $u_k$ is continuous, by Hartogs' lemma for $\om$-sh functions \cite[Lemma~9.14]{GN18}, there is $j_k \geq k>0$ large enough such that for $j \geq j_k$,
$$
	v_j + \de \leq u_k \quad\text{on } K.
$$
Thus, there exist subsequences of  $\{u_k\}$ and $\{v_j\}$, which can be used in the argument from \cite[Theorem~4.1]{BT82}.
\end{proof}

\begin{cor}\label{cor:DP}\label{cor:CP2} Let $u,v$ be bounded $m-\om$-sh functions in a neighborhood of $\ov\Om$ such that $\liminf_{z\to \d\Om}(u-v)(z) \geq 0$. Assume that $H_m(v)\geq H_m(u)$ in $\Om$. Then, $u\geq v$ on $\Om$.
\end{cor}

\begin{proof} Arguing by contradiction, suppose that $\sup_{\Om}(v-u) = d>0$. Hence, there exist $ \de, a >0$ so small that $\sup_{\Om} [(1+ a) v - (u+\de)] >d/2$ and $\liminf_{z\to \d\Om} [(u+\de) - (1+ a)v](z) \geq 0$. Applying Theorem~\ref{thm:CP1} for $\wt u = u+\de$ and $\wt v = (1+a)v$, we have for $0<s<\veps_0$,
$$
	\int_{U(\veps,s)} H_m(\wt v+\veps\rho) \leq \left( 1+ \frac{C s}{\veps^m} \right) \int_{U(\veps,s)} H_m(u).
$$
Observe that $$H_m(\wt v+\veps\rho) \geq (1+ a)^m H_m(v) + \veps^m H_m(\rho) \geq (1+ a)^m H_m(u) +  \veps^m H_m(\rho).$$ Hence,  we derive from the above inequality that
$$
	\veps^m \int_{U(\veps,s)} H_m(\rho) \leq 0
$$
for $s>0$ so small that $(1+ a)^m \geq 1+ Cs/\veps^m$. Therefore, the Lebesgue measure of  $U(\veps, s)$ is zero. This is impossible as it  is non-empty quasi-open set for $0<s<\veps_0$.
\end{proof}

The above argument also gives

\begin{cor}[domination principle] Let $u, v$ be bounded $m-\om$-sh such that $\limsup_{z\to \d\Om}|u(z) - v(z)| =0$ and $\int_{\{u<v\}}H_m(u) =0$. Then, $u\geq v$ in $\Om$.
\end{cor}

\section{Polar sets and negligible sets}

In this section we study the polar sets  and negligible sets of $m-\om$-sh functions. We obtain here results analogous to those in pluripotential theory from \cite{BT82}. Let us first give the definitions.

\begin{defn}[$m$-polar sets] A set $E$ in $\bC^n$ is $m$-polar if for each $z\in E$ there is an open set $z\in U$ and a $m-\om$-sh function $u$ in $U$ such that $E \cap U \subset \{u=-\infty\}$.
\end{defn}

Let $\{u_\al\}$ be a family of $m-\om$-sh functions in $\Om$ which is locally bounded from above. Then, the function 
$$
	u(z) = \sup_\al u_\al (z)
$$
need not be $m-\om$-sh, but its upper semicontinuous regularization
$$
	u^*(z) = \lim\sup_{x \to z} u(x) \geq u(z)
$$
is $m-\om$-sh (see \cite[Proposition~2.6-(c)]{GN18}). A set of the form
\[\label{eq:defn-negligible}
	N = \{z\in \Om: u(z)<u^*(z)\}
\]
is called {\em $m$-negligible}.

Notice that a $n$-polar/$n$-negligible set is pluripolar/negligible. Clearly a pluripolar set  is a $m$-polar (or $m$-negligible sets) for every $1\leq m \leq n$.  More generally, a $m$-polar (resp. $m$-negligible) is a $(m-1)$-polar (resp. $(m-1)$-negligible) set. 
An effective way to study these sets is by extremal functions.

\begin{defn} Let $E$ be a subset of a bounded open set $\Om \subset \bC^n$. We define
$$
	u_E =	u_{E,\Om} = \sup\{ v(x) : v \text{ is } m-\om\text{-sh in }\Om, \;u\leq 0, \; u\leq -1 \text{ on }E  \}
$$
\end{defn}

By Choquet's lemma $u_E$ is the limit of an increasing sequence of $m-\om$-sh functions. It follows from \cite[Corollary~9.9]{GN18} that $u_E^*$ is $m-\om$-sh and $u_E = u_E^*$ almost everywhere. Moreover, $u_E^*\equiv 0$ if and only if there exists an increasing sequence of $m-\om$-sh functions  $\{v_j\}_{j\geq 1}$ satisfying
\[\label{eq:zero-ext-function}	
	v_j \leq 0, \quad v_j \leq -1 \text{ on } E, \quad \int_\Om |v_j| dV_{2n} \leq 2^{-j}.
\]

\begin{lem}\label{lem:properties-ext-fct} Let $\Om$ be bounded open set in $\bC^n$. Then
\begin{itemize}
\item
[(i)] If $E_1 \subset E_2$, then $u_{E_2} \leq u_{E_2}$.
\item
[(ii)] If $E\subset \Om_1 \subset \Om_2$, then $u_{E,\Om_2} \leq u_{E,\Om_1}$.
\item
[(iii)] Let $K_j$ be non-increasing sequence of compact subset in $\Om$ and $K= \cap_j K_j$. Then, $u_{K_j}^*$ increases almost everywhere to $u_K^*$.
\item
[(iv)]If $u_{E_j}^* \equiv 0$ and $E = \cup_{j=1}^\infty E_j$, then $u_E^*\equiv 0$.

Suppose moreover that  $\Om$ is strictly $m$-pseudoconvex. Then

\item
[(v)] If $E \subset\subset \Om$ , then $\lim_{z\to \d\Om} u_{E}^* =0$.
\item
[(vi)] For every set $E\subset\Om$,  $H_m(u_E^*) \equiv 0 \quad\text{ on }\Om \setminus \ov E.$
\end{itemize}
\end{lem}

\begin{proof} The properties (i) and (ii) are obvious from the definition, and also $\lim_{j} u_{K_j} \leq u_{K}^*$ in (iii). To prove the reverse inequality let $v$ be a $m-\om$-sh with $v\leq 0$ and $u\leq -1$ on $K$. For $\veps>0$, the open set $U_\veps = \{u<-1+\veps\}$ contains $K$. Hence, $K_j \subset U_\veps$ for $j$ large enough. So, $v-\veps \leq u_{K_j}^*$. Taking supremum over all such functions $v$ we get $u_K -\veps \leq u:=\lim_j u_{K_j}$. Letting $\veps\to 0$ we obtain the conclusion. Notice again that the statement that $u= u^*$ almost everywhere follows from \cite[Corollary~9.9]{GN18}.

(iv) Let $\veps>0$. By \eqref{eq:zero-ext-function} we can choose a sequence $v_j\leq 0$, $v_j\leq -1$ on $E_j$ and $\int_\Om |v_j| dV_{2n} \leq \veps 2^{-j}$. Then, $v=\sum_{j} v_j$ is a $m-\om$-sh function  satisfying $v\leq 0$, $v\leq -1$ on $E$ and $\int_\Om |v| dV_{2n} \leq \veps$. Hence, $u_E^* \equiv 0$.

(v) Let $\psi$ be a strictly $m-\om$-sh defining function of $\Om$. Then, for $A>1$ large enough, $A\psi \leq u_E^*$. This finishes the proof.

(vi) Given the unique continuous solution of the Dirichlet problem for the homogeneous Hessian equation \cite[Theorem~3.15]{GN18} in small balls,  the result follows from a classical balayage argument.
\end{proof}

The outer capacity $cap_m^*(\bullet)$  is defined as follows.

\[
	cap_m^*(E) = \inf \left\{ cap_m(U): E \subset U, \; U \subset \Om \text{ is open}\right\}.
\]
Then, we have basic properties which follow easily from the corresponding ones of the capacity $cap_m$.

\begin{prop} Let $\Om$ be a bounded open set in $\bC^n$. Then,
\begin{itemize}
\item
[(i)] $cap_m^*(E_1) \leq cap_m^*(E_2)$ if $E_1\subset E_2\subset \Om$;
\item
[(ii)] $cap_m^*(E_1, \Om_1) \geq cap_m^*(E, \Om_2)$ if $E\subset \Om_1\subset \Om_2$;
\item
[(iii)] $cap_m^*(\cup_j E_j) \leq \sum_j cap_m^*(E_j)$.
\end{itemize}
\end{prop}

\begin{lem}\label{lem:cap-formula} Let $\Om\subset\subset \bC^n$ be a strictly $m$-pseudoconvex domain.  Let $E\subset \subset\Om$ a Borel subset.
$$
	\int_\Om H_m(u_E^*) \leq cap_m^*(E) \leq C \sum_{s=0}^m  \int_\Om (-u_E^*) H_s(u_E^*). 
$$
\end{lem}

\begin{proof} We prove first the left hand side inequality. Assume that $E= \ov E$ is compact. The property Lemma~\ref{lem:properties-ext-fct}-(vi) implies $$\int_\Om H_m(u_K^*) = \int_K H_m(u_K^*) \leq cap_m(K) \leq cap_m^*(K).$$
Assume $E=G$ is an open subset. We can find an increasing sequence of compact sets  $K_j$ such that $\cup_j K_j = G$. It is easy to see that $u_{K_j}^*$ decreases to $u_G= u_G^*$ on $\Om$. Hence,  by the  weak convergence theorem for decreasing sequences, $H_m(u_{K_j}^*) \to H_m(u_G)$ weakly. This implies
$$
	\int_\Om H_m(u_G) = \lim_{j\to\infty} \int_\Om H_m(u_{K_j}^*) \leq \lim_{j\to\infty} cap_m(K_j) \leq cap_m(G).
$$
Since $cap_m(G) =cap_m^*(G)$, the conclusion follows.

Now let $E$ be a Borel subset. By definition there exists a sequence of open sets  $\{O_j\}$ in $\Om$ containing $E$ such that $cap_m^*(E) = \lim_j cap_m(O_j)$. Replacing $O_j$ by $\cap_{1\leq s \leq j} O_j$ we may assume that $\{O_j\}_{j\geq 1}$ is decreasing. Moreover, by Choquet's lemma there exists an increasing sequence $\{v_j\}$ of negative $m-\om$-sh functions in $\Om$ such that $v_j =-1$ on $E$ and $\lim_j v_j = u_E$ almost everywhere on $\Om$. Set $G_j = O_j \cap \{v_j <-1+1/j\}$. Then, $E\subset G_j \subset O_j$ and 
$$
	v_j -1/j \leq u_{G_j} \leq u_E.
$$
So, $\lim_{j\to \infty} cap_m(G_j) = cap_m^*(E)$ and $u_{G_j}$ increases to $u_E$ almost everywhere on $\Om$. Therefore, bythe  weak convergence for increasing sequences (Lemma~\ref{lem:bounded-increasing}), 
$$
	\int_\Om H_m(u_E^*) = \lim_{j\to\infty} \int_\Om H_m(u_{G_j}) \leq \lim_{j\to\infty}cap_m(G_j) = cap_m^*(E). 
$$
Thus, the proof of left hand side inequality  is completed.

Next we prove the other one. Let $E \subset\subset \Om$ be a Borel subset and consider the sets $G_j$ defined as above. Then, $\lim_j cap_m(G_j) = cap_m^*(E)$. We also have for $0\leq s \leq m$,
$$
	\lim_{j\to \infty} \int_\Om (-u_{G_j}) H_s(u_{G_j}) = \int_\Om (-u_{E}^*) H_s(u_{E}^*)
$$
by the weak convergence for increasing sequence again. Thus, it is enough to prove the inequality for $E=G \subset\subset \Om$ an open subset. 

To this end let $-1 \leq \rho \leq 0$ be a $m-\om$-sh function in $\Om$. Since $G \subset \{u_G=-1\}$ and $u_G = u_G^*$ it follows that  for $q\geq 1$,
$$
	\int_G H_m(\rho) \leq \int (- u_G)^q H_m(\rho).
$$ 
Applying Corollary~\ref{cor:bounded-case} for $v=0$ and $u = u_G$ we get
$$
	\int_\Om (-u_G)^{3m} H_m(\rho) \leq C \sum_{s=0}^m  \int_\Om (-u_G) H_s(u_G).
$$
Taking supremum over all such functions $\rho$, we get the desired inequality.
\end{proof}

\begin{remark}\label{rmk:zero-capacity-compact} For a compact set $K$ in a strictly $m$-pseudoconvex domain $\Om$,  $cap(K,\Om) =0$ if and only if  $cap_m^*(K,\Om)=0$.
\end{remark}

\begin{prop}\label{prop:polar-set-characterization} In a strictly $m$-pseudoconvex domain $\Om$ the following are equivalent:
\begin{itemize}
\item
[(a)] $u_{E,\Om}^*=0$. 
\item
[(b)] $E\subset \{u=-\infty\}$ for a $m-\om$-sh function $u<0$ in $\Om$. 
\item
[(c)] $cap_m^*(E,\Om) =0$.
\end{itemize}
\end{prop}

\begin{proof} (a) $\Rightarrow$ (b)  follows from the property \eqref{eq:zero-ext-function} by setting $u= \sum_{j\geq 1} v_j$. Conversely, $E \subset \{v=-\infty\}$, where $v<0$ and $m-\om$-sh, implies $u_E \geq v/j$ for $j=1,2...$. So $u_E =0$ outside $\{v=-\infty\}$ whose Lebesgue measure is zero. Hence, $u_E^*=0$ by \cite[Corollary~9.7]{GN18}. The implication (c) $\Rightarrow$ (a) follows from the fact that $H_m(u_E^*) \equiv 0$ and the domination principle (Corollary~\ref{cor:DP}) if $E$ is relatively compact in $\Om$. The general case follows from the countable subadditivity of $cap_m^*$ and the corresponding property of $u_E^*$ above.

 To prove $(b)\Rightarrow (c)$ 
let us fix an open subset $V\subset\subset\Om$ and denote $\cO_j = \{u<-j\} \cap V$.  Let $\veps>0$. We wish to find  an open subset $E\subset G\subset\Om$ with $cap_m(G) <\veps$. 
Indeed, we have $0\geq u_{\cO_j} \geq \max\{u/j, -1\}$, where $u_{\cO_j}$ is the relative extremal function. Then, $u_{\cO_j} \uparrow 0$ a.e on $\Om$ by using $\om$-subharmonicity. 
Now the right hand side inequality in Lemma~\ref{lem:cap-formula} gives
$$
	cap_m(\cO_j) \leq C \sum_{s=0}^m e_{(0,0,s)},
$$
where $e_{(0,0,s)} = \int_\Om (-u_{\cO_j}) H_s(u_{\cO_j})$. Applying the weak convergence theorem for increasing convergence sequences we get that $H_s(u_{\cO_j}) \to 0$ weakly in $\Om$, $1\leq s\leq m$. Furthermore, for $s=m$ and $s=0$,
$$
	\lim_{j\to \infty} \int_\Om (-u_{\cO_j}) H_m(u_{\cO_j}) =0 = \lim_{j\to \infty} \int_\Om (-u_{\cO_j}) \om^{n}.
$$
Now we claim that for $1\leq s \leq m-1$, 
\[\label{eq:mass-zero}
	\lim_{j\to \infty} \int_\Om (-u_{\cO_j}) H_s(u_{\cO_j}) =0.
\]
Assume this is true for a moment and let us finish the proof. The above facts imply that
\[\label{eq:cap-sublevel-set}
	\lim_{j\to\infty} cap_m(\cO_j) =0.
\]
Take a sequence of  open sets  $V_s$ exhausting $\Om$.   Choose
$\cO_{j_s} = \{u<-j_s\} \cap V_s$ such that $cap_m(\cO_{j_s}) < \veps/2^s$. Define $G = \cup_{s\geq 1}  \cO_{j_s}$ which is an open set containing $E$ and which has capacity less than $\veps$.

Finally, let us verify \eqref{eq:mass-zero}. Let $\rho$ be strictly $m-\om$-sh defining function for $\Om$. Since $\cO_j \subset V \subset \subset \Om$, we have $u_{\cO_j} \geq u_V \geq A\rho$ for a constant $A>0$ depending only on $V, \Om$ by the proof of Lemma~\ref{lem:properties-ext-fct}-(v).
Hence, $u_{\cO_j}$'s can be extended to a neighborhood $\wt \Om$ of $\ov\Om$ by $A\rho$ (see e.g., \eqref{eq:extension}).
By the CLN inequality there is a uniform constant $C = C(\Om, \wt\Om)$ such that for every $j\geq 1$,
$$
	\int_{\Om} H_s(u_{\cO_j}) \leq C.
$$
Then, for a fixed $\veps>0$,
$$\begin{aligned}
	\int_\Om (-u_{\cO_j}) H_s (u_{\cO_j}) 
&\leq A \int_{\Om} |\rho| H_s(u_{\cO_j}) \\
&\leq A \veps \int_{\{|\rho| < \veps\}} H_s(u_{\cO_j}) +  A\int_{\{|\rho|\geq \veps\}} H_s(u_{\cO_j}) \\
&\leq AC \veps +  A\int_{\{|\rho|\geq \veps\}} H_s(u_{\cO_j}).
\end{aligned} $$
Since $H_s (u_{\cO_j})\to 0$  weakly  as $j\to \infty$ and $\{|\rho|\geq \veps\} \subset \Om$ is compact, letting $j\to\infty$ we get
$$
	\lim_{j\to\infty} \int_\Om (-u_{\cO_j}) H_s (u_{\cO_j})  \leq AC \veps.
$$
This holds for arbitrary $\veps>0$, where $A,C$ are uniform constants independent of $\veps$. Hence, the proof of \eqref{eq:mass-zero} is completed.
\end{proof}

\begin{thm} \label{thm:NP} $m$-negligible sets are $m$-polar. 
\end{thm}

\begin{proof} The result is local, so we may assume that all functions are defined on a bounded strictly $m$-pseudoconvex domain $\Om$. Thanks to the characterization in Proposition~\ref{prop:polar-set-characterization} it is enough to show that a negligible set $E$ has outer capacity zero. 
Let $\{u_j\}$ be the sequence in the definition of the negligible set and  put $u=\sup_j u_j$. By Choquet's lemma we may assume  this is an increasing sequence. Let $\veps>0$. By quasi-continuity we can find an open set $G\subset \Om$ such that $cap_m(G)<\veps$ and $u, u_j$'s  are continuous on $F:=\Om\setminus G$.   

Since $cap_m^*(\bullet)$ is countably subadditive, it is enough to show that $$cap_m^*(E \cap K)=0$$ for a fixed compact subset $K\subset \Om$. Observe that for all rational numbers $r<t$, the sets 
$$	K_{rt} =K \cap F \cap \{u \leq r <t \leq u^*\}
$$
are compact, because $u$ is lower semi-continuous in $K\cap F$ and $u^*$ is upper-semi continuous. Also, $(K \cap E) \setminus G$ is contained in the countable union of such compact sets. Thus, by countable subadditivity, it remains to verify $cap_m^*(K_{rt})=0$. 

Since $u$ is lower semi-continuous on $K$, there exists a constant $c$ such that $u\geq c$ on $K$. Denote $u_c=\max\{u,c\}$ and notice that $K_{rt} \subset K'_{rt}$, where
$$
 K'_{rt} = K \cap F \cap \{u_c \leq r<t\leq u_c^*\}
$$
 are compact sets. Since $\{u_c<u_c^*\}$ has inner capacity $cap_m$ zero, we have $cap_m(K'_{rt})=0$.  This implies that  $cap_m^*(K'_{rt})=0$ by Remark~\ref{rmk:zero-capacity-compact}.
\end{proof}

We have the following analogue of Josefson's theorem whose proof is the same as the one of \cite[Theorem~1.23]{K05}.

\begin{thm} For any $m$-polar subset $E$ of $\bC^n$, there exists a $m-\om$-sh function $h$ on $\bC^n$ such that $E \subset \{h = -\infty\}$.
\end{thm}

\section{Dirichlet problem in domains in $\bC^n$}
\label{sec:DP}

Let $\Om$ be a bounded strictly $m$-pseudoconvex domain in $\bC^n$.  
The comparison principle in Corollary~\ref{cor:DP} coupled with  the proof of \cite[Lemma~3.13]{GN18} gives the following stability estimate for the complex Hessian equation:

\begin{prop}\label{prop:stability-Lp}
Let $u,v \in C^0(\ov\Om)$ be $m-\om$-sh in $\Om$ and satisfy
$$
	H_m(u) = f dV_{2n}, \quad  H_m (v) = g dV_{2n}
$$
with $0\leq f,g \in L^p(\Om)$ and $p>n/m$. Then
$$
	\|u-v\|_{L^\infty} \leq \sup_{\d\Om} |u-v| + C \|f-g\|_{L^p(\Om)}^\frac{1}{m},
$$
where $C = C(m,n,p,\Om)$.
\end{prop}

Let $\psi \in C^\infty(\d\Om)$. Given a smooth  positive function $f\in C^\infty(\ov\Om,\bR)$, there is always a smooth $m-\om$-sh subsolution $\ul u \in C^\infty(\ov\Om)$,  that is
\[
	H_m(\ul u) \geq f(z), \quad \ul u =\psi \quad\text{on }\d\Om.
\]
 The stability estimate and an easy approximation argument implies that  we can solve the Dirichlet problem  when the right hand side in $L^p$, $p>n/m$ after invoking the solution  for the  smooth data due to Collins and Picard \cite{CP22}.

\begin{thm}\label{thm:Dirichlet-Lp} Let $0\leq f \in L^p(\Om)$ for some $p>n/m$. Suppose  $\vphi \in C^0(\d\Om)$.  Then there exists a unique continuous $m-\om$-sh functions $u\in C^0(\ov\Om)$  solving the Dirichlet problem
$$
	H_m(u) = f\om^n, \quad
	u = \vphi \text{ on } \d\Om.
$$
\end{thm}

Now we wish to solve the equation with the right hand side just being a positive Radon measure assuming the existence of a subsolution. 
We use recent ideas from \cite{KN22}, however  several steps require  very different proofs.

Assume $\wt \Om$ is a neighborhood of $\ov\Om$. Let us define a slightly modified  Cegrell class
\[\label{eq:Cegrell-class} 
	\wt \cE_0(\Om) = \left\{ u \text{ is bounded and } \om-m\text{-sh in }\wt\Om:  \lim_{z\to \d\Om} u(z) =0\right\}. 
\]
The set $\wt \Om$ is suppressed in this notation.
By the CLN inequality for $u\in \wt\cE_0(\Om)$ we have  $\int_\Om H_m(u) <+\infty$. 
We  introduce this modified class  to control  the integrals of  the wedge products of currents associated to bounded $m-\om$-sh functions. 

Now we follow the steps in \cite{KN22}. 
The first one  corresponds to \cite[Lemma~2.1]{KN22} which in turn was inspired by \cite[Lemma~5.2]{Ce98}. 

\begin{lem}\label{lem:L1-convergence}
Let $\la$ be a finite positive Radon measure on $\Om$ which vanishes on $m$-polar sets. Let $\{u_j\}_{j\geq 1} \subset \wt\cE_0(\Om)$ be a uniformly bounded in $\wt \Om$ sequence that converges $dV$-a.e to $u \in \wt\cE_0(\Om)$. Then, there exists a subsequence $u_{j_s}$ such that 
$$
	\lim_{j_s\to \infty} \int_\Om u_{j_s} d\la= \int_\Om u d\la.
$$
\end{lem}

\begin{proof} By the comparison principle, all functions are negative on $\Om$. The proof of \cite[Lemma~2.1]{KN22} is applicable provided that the $m$-negligible sets are $m$-polar and this is the content of Theorem~\ref{thm:NP}.
\end{proof}

Applying the lemma twice for the sequences $\{u_j\},\max\{u_j,u\}$ and combining with the identity $2\max\{u_j,u\} = u_j+u+ |u_j-u|$ we easily get a corollary.

\begin{cor} \label{cor:L1-convergence}
Let $\la$ and $\{u_j\}_{j\geq 1}$ be as in Lemma~\ref{lem:L1-convergence}. Then, there exists a subsequence, still denoted by $\{u_j\}$, such that
$
	\lim_{j\to \infty} \int_\Om |u_j-u| d\la =0.
$
\end{cor}

The following result is crucial for proving the weak convergence later.

\begin{lem}\label{lem:energy-convergence} Let $d\la$ and $\{u_j\}_{j\geq 1}$ be as in Lemma~\ref{lem:L1-convergence}. Let $\{w_j\}_{j\geq 1} \subset \wt\cE_0(\Om)$ be  uniformly bounded in $\wt \Om$.  Suppose that $w_j$ converges in capacity - $cap_m(\bullet, \Om)$ - to $w\in \wt\cE_0(\Om)$. Then,
$$
	\lim_{j\to \infty} \int_\Om|u-u_j| H_m(w_j) =0.
$$
\end{lem}

\begin{remark} Since all functions are uniformly bounded on the fixed neighborhood $\wt\Om$ of $\ov\Om$, with $\|u_j\|_{L^\infty}, \|w_j\|_{L^\infty} \leq A$, 
by Proposition~\ref{prop:CLN-bdd} there exist two positive constants $C_1,C_2$ depending only on sup-norm of $u_j$'s and $w_j$'s (and the domains $\Om,\wt\Om$), such that $$\sup_j \int_\Om H_m(u_j) \leq C_1, \quad \sup_j \int_\Om H_m(w_j) \leq C_2.$$
\end{remark}

\begin{proof} Note that $|u-u_j| = (\max\{u, u_j\} - u_j) + (\max\{u, u_j\} - u)$. Observe first that by the Hartogs lemma and quasi-continuity of $u$ (\cite[Lemma~9.14]{GN18} and Theorem~\ref{thm:quasi-continuity}) $\phi_j := \max\{u, u_j\} \to u$ in capacity. Fix $\veps>0$. We have for $j$ large,
$$\begin{aligned}
	\int _{\Om} (\max\{u, u_j\} -u) H_m(w_j) 
&\leq \int_{\{|\phi_j -u|>\veps\}} H_m(w_j) + \veps \int_{\Om} H_m(w_j) \\
&\leq A^m\; cap_m (|\phi_j -u|>\veps) + C_2 \veps.
\end{aligned}$$
Therefore, $\lim_{j\to \infty}  \int (\phi_j-u) H_m(w_j) =0$. Here and in what follows we drop the domain $\Om$ in the integrals if no confusion arises.

Next, we consider for $j>k$,
$$
\int (\phi_j -u_j) H_m(w_j) - \int (\phi_j - u_j) H_m(w_k) = \int (\phi_j - u_j) dd^c (w_j - w_k) \wed T \wed \om^{n-m} ,
$$
where $T= T(j,k) = \sum_{s=1}^{n-1} (dd^c w_j)^s \wed (dd^c w_k)^{m-1-s}$. Let us write $h_j = \phi_j -u_j$. Now the proof gets more complicated than the one in \cite[Lemma~2.3]{KN22} as the integration by parts produces more terms involving the torsion of $\om$.

By the  integration by parts 
\[\label{eq:compare-a}
\begin{aligned}
&	\int h_j dd^c (w_j - w_k) \wed T \wed \om^{n-m} = \int (w_j - w_k) dd^c (h_j\om^{n-m}) \wed T. 
\end{aligned}
\]
This integration by parts formula is justified by an approximation argument as follows. All functions are continuous on the boundary $\d\Om$ with zero value there, so the approximating sequences of smooth functions  converge to zero uniformly on $\d\Om$. Moreover, the functions are defined and uniformly bounded on a neighborhood $\wt\Om$ so the total masses of wedge products are uniformly bounded by the CLN inequality. Hence, the boundary terms vanish after passing to the limit (see Remark~\ref{rmk:continuous-case} and also \cite[Proposition~3.7]{GZ-book}). 

By a direct calculation, 
\[\label{eq:4-terms}
\begin{aligned}
	dd^c (h_j \om^{n-m}) 
&= dd^c h_j \wed \om^{n-m} + h_j dd^c \om^{n-m} \\
&\quad + (n-m ) [d h_j \wed d^c \om + d\om \wed d^c h_j]\wed \om^{n-m-1}.
\end{aligned}\]
For the first term we obtain a bound
\[\label{eq:first-term}
	\int (w_j-w_k)dd^c h_j \wed \om^{n-m} \wed T \leq  \int |w_j-w_k|dd^c (\phi_j + u_j) \wed T \wed \om^{n-m},
\]
and using inequality \cite[Lemma~2.3]{KN3} for the  second term
\[\label{eq:second-term}
\begin{aligned}
&	\int (w_j-w_k) h_j dd^c\om^{n-m} \wed T \\ &\leq C\int |w_j-w_k|h_j  [dd^c (w_j +w_k)]^{m-1} \wed \om^{n-m+1}.
\end{aligned}\]
Next, since two terms in the bracket of \eqref{eq:4-terms} are mutually conjugate, we only  estimate the first one.  To this end we will use Cauchy-Schwarz' inequality (Corollary~\ref{cor:CS}):
$$\begin{aligned}
&\left|\int (w_j-w_k)d h_j \wed d^c \om \wed \om^{n-m-1} \wed T \right|^2\\ 
&\leq C\int |w_j-w_k| dh_j \wed d^c h_j \wed [dd^c (w_j+w_k)]^{m-1} \wed \om^{n-m} \\
&\qquad\times \int |w_j-w_k| [dd^c (w_j+w_k)]^{m-1} \wed \om^{n-m+1}.
\end{aligned}$$
Let us consider the second factor in the product. Since $\|w_j \|_{L^\infty}, \|u_j\|_{L^\infty} \leq A$ in $\Om$, it follows that 
\[\label{eq:compare-jk}
\begin{aligned}
&	 \int_{\Om} |w_j - w_k| [dd^c (\phi_j + u_j)]^{m-1} \wed \om^{n-m+1} \\
&\leq  A \int_{\{|w_j -w_k| >  \veps\}}  [dd^c (\phi_j + u_j)]^{m-1} \wed \om^{n-m+1} \\ &\quad+ \veps\int_{\{|w_j -w_k| \leq \veps\}}  [dd^c (\phi_j + u_j)]^{m-1} \wed \om^{n-m+1}  \\
&\leq  (2A)^{m} cap_m( |w_j -w_k| >  \veps)  + 	C \veps,
\end{aligned}\]
where the uniform bound for the integral in the third line follows from the CLN inequality as all functions are uniformly bounded in $\wt\Om$. It means that  the left hand side of the inequality \eqref{eq:compare-jk} is less than $2C \veps$ for some $k_0$ and every $j>k\geq k_0$.

As for the first factor in the product we observe that 
$$ 
	dh_j \wed d^c h_j \leq 2 du_j \wed d^c u_j + 2 d\phi_j \wed d^c \phi_j,
$$
and $2 d\phi_j \wed d^c \phi_j = dd^c \phi_j^2 - 2 \phi_j dd^c \phi_j$ and similarly for $u_j$. Therefore, we can apply the  estimate as in \eqref{eq:compare-jk} for this integral. The same for the integrals on the right hand sides of \eqref{eq:first-term} and \eqref{eq:second-term}.

 Thus, 
$$\begin{aligned}
\int (\phi_j -u_j) H_m(w_j) 
&\leq  \int (\phi_j -u_j) H_m(w_k) \\ 
&\quad +	\left|\int (\phi_j -u_j) H_m(w_j) - \int (\phi_j - u_j) H_m(w_k) \right| \\
&\leq  \int (\phi_j -u_j) H_m(w_k)  + 8C \veps \\
&\leq \int |u-u_j| H_m(w_{k}) + 8C \veps.
\end{aligned}$$
Fix $k=k_0$ and apply Corollary~\ref{cor:L1-convergence} for $d\la = H_m(w_{k_0})$ to get that for $j \geq k_1 \geq k_0$
$$
	 \int (\phi_j -u_j) H_m(w_j)  \leq (8C + 1) \veps.
$$
Since $\veps>0$ was arbitrary, the proof of the lemma is completed.
\end{proof}

 Let $\mu$ be a positive Radon measure on a bounded strictly $m$-pseudoconvex domain $\Om$. Assume that there exists a bounded $m-\om$-sh function $\ul u$ in $\Om$ such that 
\[\label{eq:bounded-subsol}
	H_m(\ul u) \geq \mu, \quad \lim_{x\to z \in \d \Om} \ul u(x) = 0.
\]
This function $\ul u$ is called a subsolution for $d\mu$. Our goal is to prove the following.

\begin{thm}\label{thm:bounded-subsolution}
Let $\vphi\in C^0(\d\Om)$ and let $\mu$ be a positive Radon measure in $\Om$. Assume that $\mu$ admits a bounded subsolution $\ul u$ as in \eqref{eq:bounded-subsol}. Then, there exists a unique bounded $m-\om$-sh function $u$ solving  $\lim_{z\to x} u(z) = \vphi(x)$  for  $x\in\d\Om$,
$$
	H_m(u) = \mu \quad\text{in } \Om.
$$
\end{thm}

We first make some reduction steps. In fact it is enough to prove the statement
 under the following additional assumptions on the measure $d\mu$, the boundary data $\vphi$ and the subsolution $\ul u$.

\begin{lem} We may assume additionally that
\begin{itemize}
\item
[(a)] $\vphi$ is smooth on $\d\Om$;
\item
[(b)] $\mu$ has compact support in $\Om$;
\item
[(c)]  the support of $\nu = (dd^c \ul u)^m\wed \om^{n-m}$ is compact in $\Om$;
\item
[(d)] $\ul u$ can be extended as a $m-\om$-sh function to a neighborhood $\wt\Om$ of $\ov\Om$.
\end{itemize}
\end{lem}

\begin{proof}
{\em Step 1:} {\bf (c)$\Rightarrow$(d)}. Denote $K:= \supp \nu$. Let $\rho$ be a strictly $m-\om$-sh defining function for $\Om$. In particular, $\rho$ is defined in a neighborhood $\wt\Om$ of $\ov\Om$. For $A>0$ (to be chosen), we define
\[\label{eq:extension}
	v = \begin{cases}
\max\{\ul u, A\rho\} \quad\text{in } \Om, \\
A\rho \quad \text{ on } \wt\Om \setminus\Om.
	\end{cases}
\]
Notice that $\lim_{x\to \d\Om} (\ul u -A\rho) =0$, so the function $v$ is well-defined. We claim that for a sufficiently large $A$ we have $\ul u\geq A\rho$ on $\Om$, and then $v$ is a required extension. 
Indeed, 
let $U \subset \subset \Om$ be a neighborhood of $K$. Since $\sup_{\ov U} \rho \leq -\de$ for some $\de>0$ and $\ul u$ is bounded, we can choose $A>0$ large enough so that $v \geq A\rho$ on $\ov U$. Furthermore, $(dd^cv)^m \wed \om^{n-m} \equiv 0$ on $\Om \setminus \ov U$ and $v \geq A\rho$ on the boundary $\d (\Om \setminus \ov U)$. The domination principle implies that $v\geq A\rho$ on $\Om\setminus \ov U$.

\begin{remark} \label{rmk:ext-constant}The constant $A>0$ chosen in this argument depends only on the defining function $\rho$, the support of $\nu$ and the sup-norm of $\ul u$.
\end{remark}

\bigskip
{\em Step 2:} {\bf (b)$\Rightarrow$(c)}.
This follows from the classical balayage argument. However, several ingredients are only available recently. We give a detailed argument.
Assume that $\supp \mu \subset U$ which is an open subset  relatively compact in $\Om$. We define an envelope
$$
	v = \sup \left\{ w: w \text{ is } m-\om \text{-sh in }\Om,\, w \leq \ul u \text{ on }U, \, w\leq 0 \right\}.
$$
It follows from \cite[Proposition~2.6]{GN18} that the upper semicontinuous regularization $v^*\geq v$ is also $m-\om$-sh function in $\Om$. Hence, $v^*=v$ belongs to the family in the definition of the envelope. So $\ul u \leq v$. Thus, $\ul u = v$ on $U$  containing $\supp \mu$. Therefore,
$$
	H_m(v) \geq \mu \quad \text{ in } \Om.
$$
Now we verify that $$H_m(v)\equiv 0\quad \text{on }  \Om \setminus \ov U.$$ 
Let $B(a,r) \subset \subset \Om\setminus \ov U$ be a small ball. By using the solution to the homogeneous Hessian equation (Theorem~\ref{thm:Dirichlet-Lp}) one can find a continuous $m-\om$-sh function $h\geq v$ in $\Om$ which is maximal in $B(a,r)$ in the sense that $H_m(h) \equiv 0$ in $B(a,r)$ and $h = v$ on $\Om \setminus B(a,r)$. 

Observe also $U \subset \Om \setminus B(a,r)$ and the function $h$ is  a candidate in the envelope. So, $h  \leq  v$ in $\Om$. Hence, $h= v$ everywhere. We have  $H_m(v)  \equiv 0$ on $B(a,r)$. The ball is arbitrary so we get the desired property for $v$.

\bigskip

{\em Step 3:} {\bf (b).} If the problem is solvable for measures with compact support, then  it is solvable for a general measure. In fact, 
let $\eta_j \uparrow 1$ be a sequence of cut-off functions. Then, $\eta_j \mu$ admits $\ul u$ as a bounded subsolution. Solve the equation $H_m(w) = \eta_j\mu$ to obtain a bounded $m-\om$-sh function $u_j$ with $u_j=\vphi$ on $\d \Om$. Denote by $h \in C^0(\ov \Om)$ a unique $m-\om$-sh solution to $H_m(h) \equiv 0 \text{ in }\Om$ with $h=\vphi$ on $\d\Om$. Then, 
 $H_m(\ul u+h)\geq H_m(u_j)$.  The domination principle gives 
$$
	\ul u + h \leq u_j \leq h,
$$
and the sequence $u_j$ is decreasing. Define $u = \lim_j u_j$. Then, it is a solution to $H_m(u) =\mu$ with the boundary data $\vphi$.

\bigskip
{\em Step 4:} {\bf (a).} If we can solve the Dirichlet problem for smooth boundary data, then it is solvable for the continuous one. Indeed, let $\vphi_j$ be a sequence of smooth functions decreasing to $\vphi$. Find a bounded $m-\om$-sh function $u_j$ in $\Om$ such that
$$
	H_m(u_j) = \mu, \quad \lim_{z\to x} u_j(z) = \vphi_j (x)\text{ for every }x\in \d\Om.
$$ 
Let $h_j$ be the solution the homogeneous equation $H_m(h_j) \equiv 0$ with the boundary data $\vphi_j$.  By the domination principle, $u_j\geq h_j+\ul u$ and $u_j$ is a decreasing sequence.  
Since $\vphi_j$ is uniformly bounded, so is $h_j$. Therefore, the function $u=\lim_j u_j$ is a required solution to the equation.
\end{proof}

We proceed to prove the theorem under the assumptions (a)-(d).

\begin{proof}
Recall that we assumed that subsolution $\ul u$ is defined in a neighborhood $\wt\Om$  of $\ov\Om$. Hence, $\ul u \in \wt\cE_0(\Om)$ in the sense of \eqref{eq:Cegrell-class}. By Proposition~\ref{prop:smoothing} we can find a decreasing sequence of smooth $m-\om$-sh functions $v_j$ defined in $\wt\Om$ and such that $v_j \downarrow \ul u$ point-wise. Since $\ul u$ is bounded, $\{v_j\}$ is uniformly bounded.
Next, let us write
$$
	H_m(v_j) = f_j dV \quad\text{in } \wt\Om,
$$
where $dV$ is the Euclidean volume form on $\bC^n$. Observe that $v_j$ is no longer zero
 on the boundary of $\Om$, however we can modify it by solving the Dirichlet problem to find $\bar v_j \in  C^0(\ov\Om),$  $m-\om$-sh satisfying
$$
	H_m(\bar v_j)= f_j dV \quad \text{in }\Om, \quad \bar v_j =0 \text{ on } \d \Om.
$$
Since $\ul u$ continuous on $\d\Om$, by Dini's theorem, $v_j$ converges uniformly to $\ul u$ on $\d\Om$. As a consequence of the  stability estimate for the right hand side in $L^p$, we have 
$$\|\bar v_j - v_j\|_{L^\infty(\Om)} \leq \sup_{\d \Om}|\bar v_j -v_j| = \sup_{\d\Om}|v_j|.$$ 
Hence, $\{\bar v_j\}$ is also uniformly bounded on $\ov\Om$. Also by Dini's theorem $v_j\to \ul u$ uniformly on compact sets where $\ul u$ is continuous. By the stability estimate above $\bar v_j \to \ul u$ uniformly on such compact sets. Combining  this with the quasi-continuity of $\ul u$, we get also that 
$$ 
	\bar v_j \to  \ul u  \quad \text{ in capacity as } j\to \infty.
$$
Thus  $\bar v_j$'s have zero boundary values, but in general those are only continuous functions. Note also that $H_m(\bar v_j)$ converges weakly to $\nu$ whose support is compact in $\Om$. Thus, 
$$
	\sup_j \int_\Om H_m(\bar v_j) \leq C.
$$
Moreover  all $\bar v_j$ can be extended so that they form a uniformly bounded sequence in  $\wt\cE_0(\Om)$ as follows. 
Set  $\cw v_j:= \max\{\bar v_j, A\rho\}$ where $A$ is the same as in (8.2).
We verify that 
\[ \label{eq:zero-boundary-sub}\cw v_j  \text{ converges in capacity to }  \ul u \text{ as } j\to \infty.
\]
In fact, we fix a $\veps>0$. Then,
$$
	\{z\in \Om:|\cw v_j - \ul u| >\veps\} = \{z\in \Om: \cw v_j - \ul u >\veps\} \cup \{z\in \Om:\cw v_j -\ul u <-\veps\}.
$$
Note that $A\rho \leq \ul u$ in $\Om$, which implies $\{\cw v_j - \ul u > \veps\} = \{\bar v_j -\ul u >\veps\}$. So, 
$$
	cap_m(\cw v_j -\ul u>\veps) \to 0 \quad \text{as } j\to \infty.
$$
On the other, since $\bar v_j \leq  \cw v_j$ we have
$$
	\{\cw v_j -\ul u <-\veps \} \subset 	\{\bar v_j -\ul u <-\veps \} .
$$
The capacity of the latter set tends to zero, so the same is true for the former.
This finishes the proof of the claim $\cw v_j \to \ul u$ in the capacity $cap_m$.

We are now ready to produce  a sequence of functions whose limit point will be the desired solution.
By the Radon-Nikodym theorem we can write $\mu = g \nu$ for a Borel measurable function $0\leq g\leq 1$. Assume first that $g$ is continuous  (we will relax the assumption on $g$ at the end of the proof) and solve for $u_j \in SH_m(\om) \cap C^0(\ov \Om)$, the equation
$$
	H_m(u_j) = g \chi f_j dV, \quad u_j = \vphi \text{ on } \d \Om,
$$
where $0\leq \chi \leq 1$ is the cut-off function in $\Om$ such that $\chi \equiv 1$ on a neighborhood of $\supp \nu$ and $\supp\chi \subset \subset \Om$.
Set
\[\label{eq:sol}u = (\limsup_{j \to \infty} u_j)^*.
\]
Let us show first that $\{u_j\}$ is uniformly bounded. Indeed, let $\psi$ be a smooth $m-\om$-sh solution to the equation $(dd^c \psi)^m \wed \om^{n-m} = 1$ in $\ov\Om$ with $\psi= \vphi$ on $\d \Om$. Thus, we may assume that $\psi$ is $m-\om$-sh on $\wt\Om$. At this point we used the smoothness of $\vphi$. Furthermore, let  $h \in C^0(\ov\Om)$ be a $m-\om$-sh  solution to $$
	H_m(h) =0 \quad\text{in } \Om, \quad h = \vphi \text{ on } \d\Om.
$$ It follows from the domination principle that
$$
	\psi+ \cw v_j \leq u_j \leq h.
$$
(Here we may increase $A$ so that $\cw v_j = \max\{\bar v_j, A\rho\} = \bar v_j$ in a neighborhood of $\supp \nu$ and $\supp\chi$.)
Note that the above lower and upper bounds of $u_j$ are continuous on the boundary $\d\Om$ and equal to $\vphi$ there. So  $u_j$ and $u$ have the same property. Thus passing to a subsequence we may assume that 
\[\label{eq:improve-subsequence}
	u_j \to u \text{ in } L^1(\Om), \quad u_j \to u \quad\text{a.e. in } dV.
\]

The next step is to prove that $H_m(u) = \mu$.  Observe that $\psi+ \cw v_j$ is defined in the neighborhood $\wt\Om$ of $\ov \Om$.
 This combined with the inequality $\psi+ \cw v_j \leq u_j$  allows  to extend $u_j$ to $\wt\Om$ by setting
$$
	\wt u_j = \begin{cases}
	\max\{u_j, \psi+\cw v_j\}  &\quad\text{on } \Om, \\
	\psi + \cw v_j &\quad\text{on } \wt\Om\setminus \Om.
	\end{cases}
$$
Using again the smoothness of $\vphi$, we can find  $\psi'$  a strictly $m$-sh function on a (possibly smaller) neighborhood $\wt\Om$ of $\ov\Om$ satisfying 
$$
 \psi'=-\vphi \text{ on } \d\Om.
$$
Then, we have clearly
\[\label{eq:zero-boundary-sol}
	\wh u_j := u_j +\psi' \in \wt\cE_0(\Om)
\]
for all $j\geq 1$. Consequently,
we get the  important uniform bound for the total mass of mixed Hessian operators. 

\begin{lem} Let $T_{j,k}= (dd^c \wh u_j)^s \wed (dd^c \cw v_k)^{\ell} \wed \om^{n-s-\ell}$, where $0\leq s +\ell \leq m$. There exists a uniform constant $C$ independent of $j,k$ such that 
$$
	\int_\Om T_{j,k} \leq C.
$$
\end{lem}

\begin{proof} All functions belong to $\wt\cE_0(\Om)$  defined on the fixed neighborhood $\wt\Om$ of $\ov\Om$. Thanks to Remark~\ref{rmk:ext-constant} we know that  the extensions of $\cw v_i$ are uniformly bounded on $\wt\Om$. Hence, the extensions of $u_j$ above are uniformly bounded as well. The proof follows from the CLN inequality.
\end{proof}

The next result corresponds to \cite[Lemma~3.5]{KN22}.

\begin{lem} There exists a subsequence $\{u_{j_s}\}$ such that for 
$$
	w_s: = \max\{u_{j_s}, u-1/s\}
$$
the following claims hold
\begin{itemize}
\item
[(a)] $\lim_{s\to \infty} \int_\Om |u_{j_s} - u| H_m(u) =0$;
\item
[(b)] $\lim_{s\to \infty} \int_\Om |u_{j_s} - u| H_m(w_s) =0$;
\item
[(c)] $\lim_{s\to \infty} \int_\Om |u_{j_s} - u| H_m(u_{j_s})=0$.
\end{itemize}
\end{lem}

\begin{proof} (a) Since $u$ is bounded in $\Om$ by Proposition~\ref{prop:polar-set-characterization} the measure $H_m(u)$ vanishes on $m$-polar sets. Denote by  $\wh u = u+\psi'$, then $\wh u$ is the limit (a.e-$dV$) of the sequence $\{\wh u_j\}_{j\geq 1} \subset \wt\cE_0(\Om)$ from \eqref{eq:zero-boundary-sol}. Thus, $\wh u \in \wt\cE_0(\Om)$. Notice that $\wh u_{j} - u = u_j -u$, so the assumptions of Corollary~\ref{cor:L1-convergence} are satisfied and it concludes the proof of (a). 

Let us prove (b). Clearly, $\wh w_s = w_s + \psi' \in \wt\cE_0(\Om)$. Since $w_s\to u$ in capacity as $s\to\infty$, 
the same convergence holds for  $\wh w_s \to \wh u$. It follows from Lemma~\ref{lem:energy-convergence} that
$$
	0 = \lim_{s\to \infty} \int_\Om |\wh w_s - \wh u| H_m(\wh w_s) \geq \lim_{s\to \infty}\int_\Om |u_{j_s} -u| H_m(w_s).
$$
The proof of (b) finished.

For the last item (c), we use the equation $$H_m(u_{j_s}) = g \chi f_{j_s}dV \leq \chi H_m(\bar v_{j_s}) \leq H_m(\cw v_{j_s}),$$
where the first inequality followed from the fact $0\leq g\leq 1$ and the last one is by $\bar v_j  =\cw v_j$ in neighborhood of $\supp\nu$ which can be taken to be the neighborhood of $\supp\chi$. Taking into account the convergence in capacity of $\cw v_{j_s}$ to $\ul u$, as $j_s\to \infty$, the proof of (c) follows again from Lemma~\ref{lem:energy-convergence}.
\end{proof}

We are in the position to conclude that $u$ from \eqref{eq:sol} is indeed the solution. 

The argument of \cite[Lemma~3.6]{KN22} is readily applicable to conclude that there exists a subsequence $\{u_{j_s}\}_{s\geq 1}$ of $\{u_j\}_{j\geq 1}$ such that 
$$
	H_m(u_{j_s}) \to H_m(u) \quad \text{weakly}.
$$
Hence, if $0\leq g \leq 1$ is a continuous function whose support is compact in $\Om$, then there exists a unique bounded $m-\om$-sh function with $u = \vphi$ on $\d\Om$ and $H_m(u) = g H_m(\ul u)$. The general case  of a  Borel function $0\leq g \leq 1$ follows from the argument in \cite[page 11]{KN22} at the end of the  proof of Theorem~3.1.
\end{proof}

\section{Hessian equations on Hermitian manifolds with boundary}
\label{sec:subsolution}
Let $(\ov M,\om)$ be a smooth compact Hermitian manifold of dimension $n$ with non-empty boundary $\d M$. Then, $\ov M = M \cup \d M$, where $M$ is a complex manifold of dimension $n$. Let $1\leq m\leq n$ be an integer and $\al \in \Ga_m(\om)$ be a real $(1,1)$-form.

Recently  Collins and Picard \cite{CP22} solved the Dirichlet problem in $M$ for the Hessian equation $(\al + dd^c u)^m\wed \om^{n-m} = f\om^n$, for smooth data, assuming the existence of a subsolution. 
The goal of this section is to extend  this result to the case of bounded functions.
The special case of the Monge-Amp\`ere equation was treated in \cite[Theorem~1.2]{KN22}.
The theorem below is also a significant improvement of \cite[Theorem~1.3]{GN18}.

Recall from \cite[Definition~2.4, Lemma~9.10]{GN18} that a function $u: M \to [-\infty,+\infty)$ is called $(\al,m)-\om$-subharmonic if it can be written locally as a sum of a smooth function and a $\om$-sh function, and globally for any collection $\ga_1,...,\ga_{m-1} \in \Ga_m(M,\om)$,
\[\label{eq:m-subharmonic-mfd}
	(\al + dd^c u) \wed \ga_1\wed\cdots\wed \ga_{m-1}\wed \om^{n-m} \geq 0 \quad\text{ on } M
\]
in the weak sense of currents. Denote $SH_{\al,m}(M, \om)$ or $SH_{\al,m}(\om)$ be the set of all $(\al,m)-\om$-sh function on $M$.  

If $\Om$ is a local coordinate chart on $M$ and  $\rho$ is a strictly psh function on $\Om$ such that 
$$dd^c \rho \geq \al \quad \text{on }\Om,$$
then $u+\rho$ is a $m-\om$-sh function on $\Om$. 
Using this fact, we can easily extend the definition of the wedge product
for currents associated to  bounded $(\al,m)-\om$-sh functions  by using partition of unity and the local one (Definition~\ref{defn:w-prod}). Namely, write $\tau = \al -dd^c\rho$ which is a smooth $(1,1)$-form. Then, $\al + dd^c u = dd^c (u+\rho) + \tau$. We define
$$\begin{aligned}
	(\al + dd^c u)^m \wed \om^{n-m} 
&:= 	\sum_{k=0}^m  \binom{m}{k}[dd^c (u+\rho)]^k \wed \tau^{m-k} \wed \om^{n-m}\\ 
&= 	\sum_{k=0}^m \binom{m}{k} \cL_k(u+\rho) \wed \tau^{m-k}. 
\end{aligned}$$
This gives a positive Radon measure on $M$ by the weak convergence theorem. 
Similarly, the wedge product for bounded $(\al,m)-\om$-sh functions $u_1,...,u_m$
$$
	(\al+ dd^cu_1) \wed \cdots \wed (\al + dd^c u_m) \wed \om^{n-m}
$$
is a well-defined positive Radon measure.
Since the definition is local, all local results for $m-\om$-sh functions in a local coordinate chart  transfer to $(\al,m)-\om$-sh functions on the manifold $M$. For simplicity we denote $\al_u := \al + dd^c u$ and
$$
	H_{m,\al}(u)  = (\al + dd^cu)^m \wed \om^{n-m}.
$$

Now, given  a positive Radon measure $\mu$ on $M$ and a continuous boundary data $\vphi\in C^0(\d M,\bR)$ we wish to solve the Dirichlet problem
\[\label{eq:DP-mfd}
\begin{cases}	
	u \in SH_{\al,m}(\om) \cap L^\infty(\ov M),\\
	H_{m,\al}(u)  = \mu,\\
	\lim_{z\to x} u(x) = \vphi(x) \quad\text{for } x\in \d M.
\end{cases}\]

Let us state a general existence result.

\begin{thm}\label{thm:bounded-sub-mfd} Assume there exists a bounded $(\al,m)-\om$-sh function $\ul u$ on $M$ such that $\lim_{z\to x}\ul u(z) = \vphi(x)$ for $x\in\d M$ and 
$
	H_{m,\al}(\ul u) \geq \mu \quad \text{on } M.
$
Then, there is a solution to the Dirichlet problem \eqref{eq:DP-mfd}.
\end{thm}

\begin{remark} In the general setting the uniqueness is not known, unlike in a bounded strictly $m$-pseudoconvex domain (Theorem~\ref{thm:bounded-subsolution}). On the other hand, if we assume further that  either the manifold  $M$ is Stein, or both $\om$ and $\al$ are closed forms, then the solution will be unique.
\end{remark}

As we use  the Perron envelope method to show the theorem the most important ingredient 
 is the proof of the special case $M\equiv \Om$  a ball in $\bC^n$.

\begin{lem}\label{lem:bounded-ss-local} Let $\vphi\in C^0(\d\Om, \bR)$. Suppose $\mu \leq H_m(v)$ for some bounded $m-\om$-sh function $v$ in $\Om$ with $\lim_{z\to x} v(z) =0$ for $x\in \d\Om$. Then, there exists a unique $(\al,m)-\om$-sh function $u$ in $\Om$ solving
\[\label{eq:DP-local}
	H_{m,\al}(u) = \mu \quad\text{in }\Om, \quad \lim_{z\to x} u (z) = \vphi(x) \text{ for } x\in \d\Om.
\]
\end{lem}

The proof of this lemma is a straightforward extension of Theorem~\ref{thm:bounded-subsolution} so we omit the proof.

\begin{proof}[Proof of Theorem~\ref{thm:bounded-sub-mfd}] 
Let us proceed with the proof of the bounded subsolution theorem on $\ov M$. Consider the following set of functions
\[\label{eq:class-B}
	\cB(\vphi,\mu): = \left\{w \in SH_{\al,m}(M, \om) \cap L^\infty(\ov M): H_{m,\al} (w) \geq \mu, w^*_{|_{\d M}} \leq \vphi \right\},
\]
where  $w^*(x) = \limsup_{M \ni z\to x} w(z)$ for every $x\in \d M$.
Clearly, $\ul{u} \in \cB(\vphi, \mu)$. Let us solve the linear PDE finding  $h_1 \in C^0(\ov{M}, \bR)$ such that 
\[\label{eq:omega-laplace}\begin{aligned}
	(\al + dd^c h_1) \wed \om^{n-1} =0, \\
	h_1 = \vphi \quad\text{on } \d M.
\end{aligned}
\]
Since $(\al +dd^c w)\wed \om^{n-1} \geq 0$ for $w \in SH_{\al,m}(M, \om)$, the maximum principle for the Laplace operator with respect to $\om$  gives 
$$
	w \leq h_1 \quad \text{for all } w  \in \cB(\vphi, \mu).
$$
Set
\[ \label{eq:supremum}
	u (z)  = \sup_{w \in \cB(\vphi, \mu)} w (z)\quad \text{for every } z\in M.
\]
Then, by Choquet's lemma and the fact that $\cB(\vphi, \mu)$ satisfies the lattice property, $u = u^* \in \cB(\vphi, \mu)$. Again by the definition of $u$, we have   $\ul{u} \leq u \leq h_1$. It follows that 
\[\label{eq:boundary} \lim_{z \to x} u(z) = \vphi(x) \quad \text{for every } x\in \d M.
\]

\begin{lem}[Lift] \label{lem:lift}  Let $v \in \cB(\vphi, \mu)$. Let $B \subset\subset M$ be a small coordinate ball (a chart biholomorphic to a ball in $\bC^n$). Then, there exists $\wt v \in \cB(\vphi, \mu)$ such that $v \leq \wt v$  and $H_{m,\al} (\wt v) = \mu$ on $B$.
\end{lem}

\begin{proof} Given the solution in a small coordinate ball in Lemma~\ref{lem:bounded-ss-local}, the proof from \cite[Lemma~3.7]{KN22} is readily adaptable here.
\end{proof}

 By \eqref{eq:boundary} it remains to show that the function $u$ above satisfies $H_{m,\al}(u) = \mu$. Let $B \subset\subset M$ be a small coordinate ball. It is enough to check $H_{m,\al}(u) = \mu$ on $B$. Let $\wt u$ be the lift of $u$ as in Lemma~\ref{lem:lift}. It follows that $\wt u \geq u$ and $H_{m,\al} (\wt u) = \mu$ on $B$. However, by the definition $\wt u \leq u$ on $M$.  Thus, $\wt u = u$ on $B$, in particular on $B$ we have $H_{m,\al} (\wt u)  = H_{m,\al}(u) = \mu$.
\end{proof}

\begin{remark} We can also study the continuity of the solution to the Dirichlet problem \eqref{eq:DP-mfd} for a measure that is well-dominated by capacity as in \cite[Section~4]{KN22} and the weak solution to the complex Hessian type equations such as a generalization of the Monge-Amp\`ere equation in \cite{KN23}. We leave these to the future projects.
\end{remark}

\end{document}